\documentclass[12pt]{amsart}

\usepackage[utf8]{inputenc}
\usepackage[T1,T2A]{fontenc}
\usepackage[english]{babel}
\usepackage{amsfonts}
\usepackage{amsbsy}
\usepackage{amssymb}
\usepackage{graphicx}
\usepackage{mathrsfs}
\usepackage{hyperref}
\usepackage{longtable}
\usepackage{enumitem}
\usepackage{tikz}
\usepackage{comment}

\catcode`~=11 
\newcommand{\urltilde}{\kern -.15em\lower .7ex\hbox{~}\kern .04em}
\catcode`~=13 

\renewcommand{\abovecaptionskip}{0pt}
\renewcommand{\belowcaptionskip}{6pt}
\makeatletter

\renewcommand{\@makecaption}[2]{
\vspace{\abovecaptionskip}%
\sbox{\@tempboxa}{#1. #2}%
\global\@minipagefalse \hbox to \hsize {{\scshape \hfil #1.
#2\hfil}} \vspace{\belowcaptionskip}}

\makeatother

\newcommand{\Hom}{\operatorname{Hom}}
\newcommand{\rk}{\operatorname{rk}}

\newcommand{\Ker}{\operatorname{Ker}}
\newcommand{\ad}{\operatorname{ad}}
\newcommand{\Gr}{\operatorname{Gr}}
\newcommand{\Ad}{\operatorname{Ad}}
\newcommand{\diag}{\operatorname{diag}}

\newcommand{\GL}{\operatorname{GL}}

\newcommand{\SL}{\operatorname{SL}}
\newcommand{\Sp}{\operatorname{Sp}}
\newcommand{\Spin}{\operatorname{Spin}}
\newcommand{\SO}{\operatorname{SO}}

\newcommand{\sat}{\mathrm{sat}}

\newcommand{\Supp}{\operatorname{Supp}}

\newcommand{\ZZ}{\mathbb Z}
\newcommand{\QQ}{\mathbb Q}
\newcommand{\CC}{\mathbb C}

\newtheorem{theorem}{Theorem}

\newtheorem{proposition}[theorem]{Proposition}
\newtheorem{lemma}[theorem]{Lemma}
\newtheorem{corollary}[theorem]{Corollary}

\newtheorem*{question*}{Question}

\theoremstyle{definition}

\newtheorem{definition}[theorem]{Definition}
\newtheorem{example}[theorem]{Example}

\theoremstyle{remark}

\newtheorem{remark}[theorem]{Remark}

\makeatletter

\@addtoreset{theorem}{section}

\makeatother

\numberwithin{equation}{section}

\makeatletter

\@addtoreset{theorem}{section}

\makeatother

\numberwithin{equation}{section}

\newcounter{num}[table]
\newcommand{\no}{\refstepcounter{num}\arabic{num}}

\newcounter{alg}
\newcounter{stepalg}[alg]

\newcommand{\newalg}{\refstepcounter{alg}\Alph{alg}}
\newcommand{\step}{\refstepcounter{stepalg}\Alph{alg}\arabic{stepalg}}

\renewcommand{\arraystretch}{1.2}

\begin{document}

\renewcommand{\proofname}{Proof}
\renewcommand{\abstractname}{Abstract}
\renewcommand{\refname}{References}
\renewcommand{\figurename}{Figure}
\renewcommand{\tablename}{Table}

\title[Degenerations of spherical subalgebras and spherical roots]
{Degenerations of spherical subalgebras\\ and spherical roots}

\author{Roman Avdeev}


\address{%
{\bf Roman Avdeev} \newline\indent National Research University ``Higher School of Economics'', Moscow, Russia}

\email{suselr@yandex.ru}




\begin{abstract}
We obtain several structure results for a class of spherical subgroups of connected reductive complex algebraic groups that extends the class of strongly solvable spherical subgroups.
Based on these results, we construct certain one-parameter degenerations of the Lie algebras corresponding to such subgroups.
As an application, we exhibit explicit algorithms for computing the set of spherical roots of such a spherical subgroup.
\end{abstract}

\maketitle

\section{Introduction}

Let $G$ be a connected reductive complex algebraic group.
A $G$-variety (that is, an algebraic variety equipped with a regular action of~$G$) is said to be \textit{spherical} if it is irreducible, normal, and possesses an open orbit with respect to the induced action of a Borel subgroup $B \subset G$.
A closed subgroup $H \subset G$ is said to be \textit{spherical} if the homogeneous space $G/H$ is spherical as a $G$-variety.

To every spherical homogeneous space $G/H$ one assigns three main combinatorial invariants: the \textit{weight lattice} $\Lambda_G(G/H)$, the finite set $\Sigma_G(G/H)$ of \textit{spherical roots}, and the finite set $\mathcal D_G(G/H)$ of \textit{colors}.
These invariants play a crucial role in the combinatorial classification of spherical homogeneous spaces that was established with contributions of Luna~\cite{Lu01}, Losev~\cite{Lo1}, and Bravi and Pezzini~\cite{BP14, BP15, BP16} (see also an alternative approach to this classification in the preprint~\cite{Cu}).
Besides, according to the celebrated theory of Luna and Vust (see~\cite{LV} or~\cite{Kn91}), knowing $\Lambda_G(G/H)$, $\Sigma_G(G/H)$, and $\mathcal D_G(G/H)$ one obtains a complete combinatorial description of all spherical $G$-varieties containing $G/H$ as an open $G$-orbit in terms of so-called colored fans.
Due to the latter fact, we call $\Lambda_G(G/H)$, $\Sigma_G(G/H)$, and $\mathcal D_G(G/H)$ the \textit{Luna--Vust invariants} of~$G/H$.

In view of the importance of the Luna--Vust invariants in the theory of spherical varieties, a~natural problem is to compute them starting from an explicit form of a spherical subgroup~$H \subset G$.
A~standard way of specifying $H$ is to use a regular embedding of $H$ in a parabolic subgroup $P \subset G$, where ``regular'' means the inclusion $H_u \subset P_u$ of the unipotent radicals of $H$ and~$P$, respectively.
By now, a complete solution of the above problem is known essentially in the following two particular cases:
\begin{enumerate}[label=\textup{(\arabic*)},ref=\textup{\arabic*}]
\item \label{case_P=G}
$P = G$ (that is, $H$ is reductive);

\item \label{case_P=B}
$P = B$ (subgroups contained in a Borel subgroup of $G$ are called \textit{strongly solvable}).
\end{enumerate}
In case~(\ref{case_P=G}), there is a complete classification of spherical homogeneous spaces $G/H$ with reductive $H$ due to \cite{Kr, Mi, Br87}.
A description of the weight lattices and colors for this case follows from results in \cite{Kr, Avd_EWS}, and the sets of spherical roots are known thanks to the paper~\cite{BP15}.
In case~(\ref{case_P=B}), all the Luna--Vust invariants for spherical homogeneous spaces $G/H$ with $H \subset B$ were computed in~\cite{Avd_solv_inv} by using a structure theory of such subgroups developed in~\cite{Avd_solv}.

For arbitrary spherical subgroups, the state of the art in computing the Luna--Vust invariants is as follows.
First, there is a general method tracing back to Panyushev~\cite{Pa94} for computing the weight lattice of a spherical homogeneous space $G/H$ in terms of a regular embedding of $H$ in a parabolic subgroup of~$G$; see Theorem~\ref{thm_criterion_spherical}.
Second, the author is not aware of any general results on computing the colors\footnote{While this paper was under review, our subsequent paper~\cite{Avd_onEWS} appeared where the problem of computing the colors is reduced to that of computing the spherical roots.}.
Third, there are two different general approaches for computing the set of spherical roots: one is due to Luna and Vust (the method of formal curves; see~\cite[\S\,4]{LV} or~\cite[\S\,24]{Tim}) and the other one is due to Losev~\cite{Lo3}.
(In fact, both approaches deal with a generalization of the set of spherical roots to arbitrary $G$-varieties.)
However, these two approaches seem to work effectively only for some special classes of spherical homogeneous spaces.
As a consequence, the problem of finding effective algorithms for computing the spherical roots and colors for arbitrary spherical subgroups still remains of great importance.

In this paper, we are concerned with the problem of computing the set of spherical roots for a given spherical subgroup $H \subset G$.
We propose a general strategy for solving this problem based on the following idea.
First of all, a standard argument reduces the consideration to the case where $G$ is semisimple and $H$ coincides with its normalizer in~$G$.
In this case, $H$ equals the stabilizer of $\mathfrak h$ under the adjoint action of $G$ on~$\mathfrak g$, where $\mathfrak h, \mathfrak g$ are the Lie algebras of $H,G$, respectively.
Then the $G$-orbit $G\mathfrak h$ of $\mathfrak h$ in the Grassmannian $\Gr_{\dim \mathfrak h}(\mathfrak g)$ of $(\dim \mathfrak h)$-dimensional subspaces in~$\mathfrak g$ is isomorphic to $G/H$, and the closure $X$ of this orbit is said to be the \textit{Demazure embedding} of~$G/H$.
It is known from~\cite{Bri90} and~\cite{Lo2} that $X$ is a so-called wonderful $G$-variety; see the precise definition in \S\,\ref{subsec_wv}.
The latter implies that there is a bijection $J \mapsto O_J$ between the subsets of $\Sigma_G(G/H)$ and the $G$-orbits in~$X$ such that for every two subsets $J, J' \subset \Sigma_G(G/H)$ the following properties hold:
\begin{itemize}
\item
$O_J$ is a spherical $G$-variety whose set of spherical roots is~$J$;

\item
the codimension of $O_J$ in $X$ equals $|\Sigma_G(G/H) \setminus J|$;

\item
$O_J$ lies in the closure of $O_{J'}$ if and only if $J \subset J'$.
\end{itemize}
In particular, the open $G$-orbit in $X$ corresponds to the whole set $\Sigma_G(G/H)$.
Now suppose $\mathfrak h$ degenerates into two subalgebras $\mathfrak h_1,\mathfrak h_2 \subset \mathfrak g$ by taking the limit under the action of one-parameter subgroups $\phi_1,\phi_2 \subset G$, respectively, and the $G$-orbits $G\mathfrak h_1, G\mathfrak h_2 \subset \Gr_{\dim \mathfrak h}(\mathfrak g)$ are different and both have codimension~$1$ in~$X$.
Then $\Sigma_G(G\mathfrak h_1) = \Sigma_G(G/H) \setminus \lbrace \sigma_1 \rbrace$ and $\Sigma_G(G\mathfrak h_2) = \Sigma_G(G/H) \setminus \lbrace \sigma_2 \rbrace$ for two different elements $\sigma_1, \sigma_2 \in \Sigma_G(G/H)$ and hence $\Sigma_G(G/H) = \Sigma_G(G\mathfrak h_1) \cup \Sigma_G(G\mathfrak h_2)$.
Thus the problem of finding the set of spherical roots for $G/H$ reduces to the same problem for $G/N_1$ and $G/N_2$ where $N_1,N_2$ are the stabilizers in $G$ of the subalgebras $\mathfrak h_1, \mathfrak h_2$, respectively.
Since the number of spherical roots for $G/N_1$ and $G/N_2$ is strictly less than that for $G/H$, iterating the above procedure in a finite number of steps leads to a finite number of spherical homogeneous spaces having only one spherical root, and for such spaces the unique spherical root is read off directly from the weight lattice.
In principle, this strategy yields a recursive algorithm for computing the set of spherical roots for a given spherical subgroup~$H$, but the main difficulty here is to find explicit constructions of one-parameter degenerations of~$\mathfrak h$ having all the required properties.

The main goal of the present paper is to implement the above strategy for a class of spherical subgroups extending that of strongly solvable ones.
Namely, we consider spherical subgroups $H \subset G$ regularly embedded in a parabolic subgroup $P \subset G$ such that a Levi subgroup $K \subset H$ satisfies $L' \subset K \subset L$ for a Levi subgroup $L \subset P$ and its derived subgroup~$L'$.
For such spherical subgroups~$H$, we generalize several structure results obtained in~\cite{Avd_solv} for strongly solvable spherical subgroups and show in particular that, up to conjugation by an element of the connected center $C$ of~$L$, $H$ is uniquely determined by the pair $(K, \Psi)$ where $\Psi$ is the finite set of so-called \textit{active $C$-roots}, which naturally generalize active roots introduced in loc.~cit.
For subgroups $H$ with $|\Psi| \ge 2$ the above-mentioned structure results enable us to present (at least) two one-parameter subgroups $\phi_1, \phi_2 \subset G$, describe explicitly the corresponding degenerations $\mathfrak h_1, \mathfrak h_2$, and prove that they satisfy all the properties mentioned in the previous paragraph.
More specifically, if $H_u$ is not normalized by~$C$ then it is possible to choose $\phi_1,\phi_2$ to be one-parameter subgroups of~$C$.
In the opposite case we take $\phi_1,\phi_2$ to be certain root unipotent one-parameter subgroups of~$G$.
It is worth mentioning that in both cases the two resulting subgroups $N_1, N_2$ still belong to the class of spherical subgroups under consideration, which enables us to repeat our procedure for them.
When $|\Psi|=1$ (we call such cases \textit{primitive}), computation of the set of spherical roots readily reduces to the case where $P$ is a maximal parabolic subgroup.
In turn, all such cases can be easily classified and moreover for each of them the set of spherical roots turns out to be known.
As a result, this yields an explicit algorithm (we call it the \textit{base algorithm}) for computing the set of spherical roots that terminates at a finite number of primitive cases.

As can be seen from its description, the base algorithm is rather slow: to compute the set of spherical roots for a spherical subgroup $H$ with $|\Sigma_G(G/H)| = r$, in the worst case one needs to calculate $2^r - 2$ intermediate subalgebras.
Keeping this in mind, we propose an optimization of the base algorithm for spherical subgroups $H$ satisfying $K = L$.
The optimized algorithm after performing no more than $r^2-r$ degenerations reduces the problem of computing $\Sigma_G(G/H)$ to the same problem for some other spherical subgroups $H_1, \ldots, H_p$ ($p \le r$) satisfying $|\Psi(H_i)| \le 2$ for all~$i$.
The latter condition is very restrictive and in the case $|\Psi(H_i)| = 1$ we get one of the primitive cases, which requires no further work for~$H_i$.
In principle, it is also feasible to classify all possible $H_i$'s with $|\Psi(H_i)|=2$ that may appear out of arbitrary subgroups $H$ under consideration and then compute the spherical roots for them using our methods; however, these issues are not addressed in the present paper.
Regardless of the above, if $H$ is an arbitrary subgroup in the class from the previous paragraph (not necessarily satisfying $K=L$) then it seems to be very likely that a deeper analysis of the set $\Psi$ along with the base algorithm may lead to a simple combinatorial description of the set $\Sigma_G(G/H)$ purely in terms of the pair $(K,\Psi)$, without performing any degenerations, and it would be very interesting to find such a description.

We now mention two other directions of further research related to this paper.
Firstly, an important problem is to generalize the base algorithm to the case of arbitrary spherical subgroups of $G$ regularly embedded in a parabolic subgroup~$P$.
As was already mentioned above, the main difficulty here is to find constructions of one-parameter degenerations of the corresponding Lie algebras having special properties.
Secondly, a more general problem is to find explicitly (a collection of) degenerations of the Lie algebra $\mathfrak h$ of a self-normalizing spherical subgroup $H$ via which $\mathfrak h$ can reach \textit{any} $G$-orbit in the Demazure embedding of $G/H$.
This problem is quite nontrivial even for the class of spherical subgroups considered in this paper: using our degenerations constructed for such subgroups one can reach only a small part of $G$-orbits of codimension~$1$ (and no $G$-orbits of higher codimensions).
We note that the problem of reaching all $G$-orbits in the Demazure embedding is closely related to determining all the satellites of a given spherical subgroup, which were introduced in~\cite{BM}.

This paper is organized as follows.
In~\S\,\ref{sect_prelim}, we fix notation and conventions and discuss several general results used in this paper.
In~\S\,\ref{sect_gen_SV}, we collect all the necessary material on spherical varieties and provide a detailed presentation of our general strategy for computing the spherical roots.
In~\S\,\ref{sect_active_C-roots} we prove several structure results for the class of spherical subgroups focused on in this paper.
In~\S\,\ref{sect_comp_SR} we apply the results of~\S\,\ref{sect_active_C-roots} to construct one-parameter degenerations of Lie algebras of spherical subgroups under consideration and exhibit the base algorithm for computing the spherical roots for them.
In \S\,\ref{sect_optimization} we propose our optimization of the base algorithm.
Finally, \S\,\ref{sect_examples} contains several examples of computing the spherical roots via our methods.

\subsection*{Acknowledgements}

Some parts of this work were done while the author was visiting the Institut Fourier in Grenoble, France, in the summer of~2018 and in October~2019.
He thanks this institution for hospitality and excellent working conditions and also expresses his gratitude to Michel Brion for support, useful discussions, and valuable comments on this work.
Thanks are also due to Dmitry Timashev for drawing the author's attention to Proposition~\ref{prop_limits}, which led to simplifications in some proofs.
The author is especially grateful to the referees for carefully reading previous versions of this paper, pointing out many inaccuracies in them, and helpful remarks and suggestions.

This work was supported by the Russian Science Foundation, grant no.~18-71-00115.

\section{Preliminaries}
\label{sect_prelim}

\subsection{Notation and conventions}

Throughout this paper, we work over the field $\CC$ of complex numbers.
All topological terms relate to the Zariski topology.
All subgroups of algebraic groups are assumed to be algebraic.
The Lie algebras of algebraic groups denoted by capital Latin letters are denoted by the corresponding small Gothic letters.
A~\textit{$K$-variety} is an algebraic variety equipped with a regular action of an algebraic group~$K$.

$\ZZ^+ = \lbrace z \in \ZZ \mid z \ge 0 \rbrace$;

$\QQ^+ = \lbrace q \in \QQ \mid q \ge 0 \rbrace$;

$(\CC^\times, \times)$ or just $\CC^\times$ is the multiplicative group of the field~$\CC$;

$(\CC, +)$ is the additive group of the field~$\CC$;

$|X|$ is the cardinality of a finite set~$X$;

$\langle v_1,\ldots, v_k \rangle$ is the linear span of vectors $v_1,\ldots, v_k$ of a vector space~$V$;

$V^*$ is the space of linear functions on a vector space~$V$;

$\mathrm S(V)$ is the symmetric algebra of a vector space~$V$;

$\mathrm S^k(V)$ is the $k$th symmetric power of a vector space~$V$;

$\wedge^k(V)$ is the $k$th exterior power of a vector space~$V$;

$L^0$ is the connected component of the identity of an algebraic group~$L$;

$L'$ is the derived subgroup of a group~$L$;

$L_u$ is the unipotent radical of an algebraic group~$L$;

$Z(L)$ is the center of a group~$L$;

$\mathfrak X(L)$ is the character group (in additive notation) of an algebraic group~$L$;

$Z_L(K)$ is the centralizer of a subgroup $K$ in a group~$L$;

$N_L(K)$ is the normalizer of a subgroup $K$ in a group~$L$;

$N_L(\mathfrak u) = \lbrace g \in L \mid \Ad(g) \mathfrak u = \mathfrak u \rbrace$ is the normalizer in a subgroup~$L \subset G$ of a subalgebra $\mathfrak u \subset \mathfrak g$;

$\CC[X]$ is the algebra of regular functions on an algebraic variety~$X$;

$\CC(X)$ is the field of rational functions on an irreducible algebraic variety~$X$;

$\Gr_k(V)$ is the Grassmannian of $k$-dimensional subspaces of a vector space~$V$;

$G$ is a connected reductive algebraic group;

$B \subset G$ is a fixed Borel subgroup;

$T \subset B$ is a fixed maximal torus;

$B^-$ is the Borel subgroup of~$G$ opposite to~$B$ with respect to~$T$, so that $B \cap B^- = T$;

$(\cdot\,,\,\cdot)$ is a fixed inner product on~$\QQ\mathfrak X(T)$ invariant with respect to the Weyl group $N_G(T)/T$;

$\Delta \subset \mathfrak X(T)$ is the root system of~$G$ with respect to~$T$;

$\Delta^+ \subset \Delta$ is the set of positive roots with respect to~$B$;

$\Pi \subset \Delta^+$ is the set of simple roots with respect to~$B$;

$\Lambda_G^+ \subset \mathfrak X(T)$ is the set of dominant weights of $G$ with respect to~$B$;

$\alpha^\vee \in \Hom_\ZZ(\mathfrak X(T), \ZZ)$ is the coroot corresponding to a root $\alpha \in \Delta$;

$h_\alpha \in \mathfrak t$ is the image of $\alpha^\vee$ in $\mathfrak t$ under the chain $\Hom_\ZZ(\mathfrak X(T), \ZZ) \hookrightarrow (\mathfrak t^*)^* \xrightarrow{\sim} \mathfrak t$;

$\mathfrak g_\alpha \subset \mathfrak g$ is the root subspace corresponding to a root~$\alpha \in \Delta$;

$e_\alpha \in \mathfrak g_\alpha$ is a fixed nonzero element.

The simple roots and fundamental weights of simple algebraic groups are numbered as in~\cite{Bo}.

For every $\beta = \sum \limits_{\alpha \in \Pi} k_\alpha \alpha \in \ZZ^+\Pi$, its \textit{support} is defined as $\Supp \beta = \lbrace \alpha \in \Pi \mid k_\alpha > 0 \rbrace$.

We fix a nondegenerate $G$-invariant inner product on~$\mathfrak g$ and for every subspace $\mathfrak u \subset \mathfrak g$ let $\mathfrak u^\perp$ be the orthogonal complement of $\mathfrak u$ in~$\mathfrak g$ with respect to this form.

The groups $\mathfrak X(B)$ and $\mathfrak X(T)$ are identified via restricting characters from~$B$ to~$T$.

Given a parabolic subgroup $P \subset G$ such that $P \supset B$ or $P \supset B^-$, the unique Levi subgroup $L$ of $P$ containing~$T$ is called the \textit{standard Levi subgroup} of~$P$.
By abuse of language, in this situation we also say that $L$ is a standard Levi subgroup of~$G$.
The unique parabolic subgroup $Q$ of $G$ such that $\mathfrak p + \mathfrak q = \mathfrak g$ and $P \cap Q = L$ is said to be \textit{opposite} to~$P$.

Let $L \subset G$ be a standard Levi subgroup and let $K \subset G$ be a reductive subgroup (not necessarily connected) satisfying $L' \subset K \subset L$.
In this situation, we put $B_L = B \cap L$ and $B_K = B \cap K$, so that $B_L$ is a Borel subgroup of~$L$ and $B_K^0$ is a Borel subgroup of~$K$.
If $V$ is a simple $K$-module, by a highest (resp. lowest) weight vector of~$V$ we mean a $B_K$-semi-invariant (resp. $(B^- \cap K)$-semi-invariant) vector in~$V$.
The highest (resp. lowest) weight of~$V$ is the $(T \cap K)$-weight of a highest (resp. lowest) weight vector in~$V$.
These conventions on~$V$ are also valid if $L = K = G$.

Given a standard Levi subgroup $L \subset G$, we let $\Delta_L \subset \Delta$ be the root system of~$L$ and put $\Delta^+_L = \Delta^+ \cap \Delta_L$ and $\Pi_L = \Pi \cap \Delta_L$, so that $\Delta^+_L$ (resp. $\Pi_L$) is the set of all positive (resp. simple) roots of~$L$ with respect to the Borel subgroup~$B_L$.

Let $K$ be a group and let $K_1,K_2$ be subgroups of~$K$.
We write $K = K_1 \rightthreetimes K_2$ if $K$ is a semidirect product of~$K_1,K_2$ with $K_2$ being a normal subgroup of~$K$.
We write $K= K_1 \cdot K_2$ if $K$ is an almost direct product of~$K_1,K_2$, that is, $K = K_1K_2$, $K_1$ and $K_2$ commute with each other, and the intersection $K_1 \cap K_2$ is finite.

\subsection{Levi roots and their properties}
\label{subsec_Levi_roots}

Let $L$ be a standard Levi subgroup of~$G$ and let $C$ be the connected center of~$L$.
We consider the natural restriction map $\varepsilon \colon \mathfrak X(T) \to \mathfrak X(C)$ and extend it to the corresponding map $\varepsilon_\QQ \colon \QQ\mathfrak X(T) \to \QQ\mathfrak X(C)$.
Let $(\Ker \varepsilon_\QQ)^\perp \subset \QQ\mathfrak X(T)$ be the orthogonal complement of $\Ker \varepsilon_\QQ$ with respect to the inner product $(\cdot\,, \cdot)$.
Under the map $\varepsilon_\QQ$ the subspace $(\Ker \varepsilon_\QQ)^\perp$ maps isomorphically to $\QQ\mathfrak X(C)$;
we equip $\QQ\mathfrak X(C)$ with the inner product transferred from $(\Ker \varepsilon_\QQ)^\perp$ via this isomorphism.

Consider the adjoint action of~$C$ on~$\mathfrak g$.
For every $\lambda \in \mathfrak X(C)$, let $\mathfrak g(\lambda) \subset \mathfrak g$ be the corresponding weight subspace of weight~$\lambda$.
It is well known that $\mathfrak g(0) = \mathfrak l$.
We put
\[
\Phi = \lbrace \lambda \in \mathfrak X(C) \setminus \lbrace 0 \rbrace \mid \mathfrak g(\lambda) \ne \lbrace 0 \rbrace \rbrace.
\]
Then there is the following decomposition of $\mathfrak g$ into a direct sum of $C$-weight subspaces:
\begin{equation} \label{eqn_decomposition}
\mathfrak g = \mathfrak l \oplus \bigoplus\limits_{\lambda \in \Phi} \mathfrak g(\lambda).
\end{equation}
In what follows, elements of the set $\Phi$ will be referred to as \textit{$C$-roots}.
It is easy to see that $\Phi = \varepsilon(\Delta \setminus \Delta_L)$.
In particular, $\Phi = - \Phi$.

Now consider the adjoint action of~$L$ on~$\mathfrak g$.
Then each $C$-weight subspace of $\mathfrak g$ becomes an $L$-module in a natural way.

The following proposition seems to be well known; for convenience, we provide a proof of part~(\ref{prop_properties_of_g(lambda)_c}) due to the lack of reference.

\begin{proposition} \label{prop_properties_of_g(lambda)}
The following assertions hold:
\begin{enumerate}[label=\textup{(\alph*)},ref=\textup{\alph*}]
\item \label{prop_properties_of_g(lambda)_a}
\textup(see \cite[Theorem~1.9]{Ko} or \cite[Ch.~3, Lemma~3.9]{GOV}\textup) for every $\lambda \in \Phi$ the subspace $\mathfrak g(\lambda)$ is a simple $L$-module;

\item \label{prop_properties_of_g(lambda)_b}
\textup(see \cite[Lemma~2.1]{Ko}\textup) for every $\lambda, \mu, \nu \in \Phi$ with $\lambda = \mu + \nu$ one has $\mathfrak g(\lambda) = [\mathfrak g(\mu), \mathfrak g(\nu)]$;

\item \label{prop_properties_of_g(lambda)_c}
for every $\lambda \in \Phi$ there is an $L$-module isomorphism $\mathfrak g(-\lambda) \simeq \mathfrak g(\lambda)^*$.
\end{enumerate}
\end{proposition}

\begin{proof}[Proof of~\textup(\ref{prop_properties_of_g(lambda)_c}\textup)]
It is easy to see that the highest weight of $\mathfrak g(-\lambda)$ is opposite to the lowest weight of $\mathfrak g(\lambda)$, which implies the required result.
Alternatively, it suffices to notice that the fixed $G$-invariant inner product on~$\mathfrak g$ yields a nondegenerate $L$-equivariant pairing between $\mathfrak g(\lambda)$ and~$\mathfrak g(-\lambda)$.
\end{proof}

\begin{proposition} \label{prop_bracket}
Suppose that $\lambda,\mu,\nu \in \Phi$ and $\lambda = \mu + \nu$. Then
\begin{enumerate}[label=\textup{(\alph*)},ref=\textup{\alph*}]
\item \label{prop_bracket_a}
$\mathfrak g(\lambda)$ is isomorphic as an $L$-module to a submodule of $\mathfrak g(\mu) \otimes \mathfrak g(\nu)$;

\item \label{prop_bracket_b}
if $\mu = \nu$ then $\mathfrak g(\lambda)$ is isomorphic as an $L$-module to a submodule of $\wedge^2 \mathfrak g(\mu)$.
\end{enumerate}
\end{proposition}

\begin{proof}
The bilinear map $\mathfrak g(\mu) \times \mathfrak g(\nu) \to \mathfrak g(\lambda)$, $(x,y) \mapsto [x,y]$, induces an $L$-module homomorphism $\mathfrak g(\mu) \otimes \nobreak \mathfrak g(\nu) \to \mathfrak g(\lambda)$ in case~(\ref{prop_bracket_a}) and $\wedge^2 \mathfrak g(\mu) \to \mathfrak g(\lambda)$ in case~(\ref{prop_bracket_b}).
The claim now follows from Proposition~\ref{prop_properties_of_g(lambda)}%
(\ref{prop_properties_of_g(lambda)_a},\,%
\ref{prop_properties_of_g(lambda)_b}) and complete reducibility of $L$-modules.
\end{proof}

\begin{proposition}[{see~\cite[Theorem~2.3]{Ko}}] \label{prop_pairs_of_C-roots}
Suppose that $\lambda, \mu \in \Phi$.
Then
\begin{enumerate}[label=\textup{(\alph*)},ref=\textup{\alph*}]
\item \label{prop_pairs_of_C-roots_a}
if $(\lambda, \mu) < 0$ and $\lambda + \mu \ne 0$ then $\lambda + \mu \in \Phi$;

\item \label{prop_pairs_of_C-roots_b}
if $(\lambda, \mu) > 0$ and $\lambda - \mu \ne 0$ then $\lambda - \mu \in \Phi$.
\end{enumerate}
\end{proposition}

\subsection{One-parameter degenerations in complete varieties}

The following result is a direct consequence of the valuative criterion of properness; see \cite[Chapter~II, Theorem~4.7]{Har}.

\begin{proposition} \label{prop_limits}
Suppose $X$ is a complete variety equipped with an action of a one-parameter group $\mathbb G$ \textup(either multiplicative $(\CC^\times, \times)$ or additive $(\CC, +)$\textup).
Then for every point $x \in X$ there exists $\lim \limits_{t \to \infty} tx = x_\infty$ and $x_\infty$ is a $\mathbb G$-fixed point of~$X$.
\end{proposition}

In this paper, we shall apply the above result in the situation where $X = \Gr_k(V)$ for a finite-dimensional vector space~$V$ and some~$k$.
Note that, if $V$ is a Lie algebra, $x \in X$ is a Lie subalgebra of~$V$, and the action of $\mathbb G$ on~$X$ is induced from a one-parameter group action on~$V$ by Lie algebra automorphisms, then the limit $x_\infty$ is automatically a Lie subalgebra of~$V$.

\subsection{Additive degenerations of subspaces in simple \texorpdfstring{$\mathfrak{sl}_2$}{sl\_2}-modules}
\label{subsec_degen_subsp}

Consider the Lie algebra $\mathfrak{sl}_2$ with canonical basis $\lbrace e,h,f\rbrace$, so that $[e,f] = h$, $[h,e] = 2e$, and $[h,f] = -2f$.
Let $V$ be a simple $\mathfrak{sl}_2$-module with highest weight $p \in \ZZ^+$.
Fix a basis $\lbrace v_{p-2i} \mid i=0,\ldots,p \rbrace$ of $V$ such that $h \cdot v_{p-2i} = (p-2i)v_{p-2i}$ for all $i=0,\ldots, p$, $f \cdot v_{p-2i} = v_{p-2i-2}$ for all $i = 0,\ldots, p-1$, and $f \cdot v_{-p} = 0$.

Consider the one-parameter unipotent subgroup $\phi \colon \CC \to \GL(V)$ given by $\phi(t) = \exp (tf)$ and let $U \subset V$ be a subspace with $\dim U = k$.
According to Proposition~\ref{prop_limits}, there exists the limit $\lim \limits_{t \to \infty}\phi(t) U = U_\infty$ and it is a $\phi(t)$-stable subspace of~$U$.

\begin{proposition} \label{prop_limitII_prelim}
One has $U_\infty = \langle v_{-p+2i} \mid i=0,\ldots, k-1 \rangle$.
\end{proposition}

\begin{proof}
Being $\phi(t)$-stable, $U_\infty$ is automatically $f$-stable.
It remains to observe that $\langle v_{-p+2i} \mid i=0,\ldots, k-1 \rangle$ is the only $f$-stable $k$-dimensional subspace of~$V$.
\end{proof}

In \S\,\ref{subsec_degen_add}, we shall apply Proposition~\ref{prop_limitII_prelim} in situations where the subspace $U$ is $h$-stable.
In this case, $U = \langle v_{n_0}\rangle \oplus \ldots \oplus \langle v_{n_{k-1}} \rangle$ for some $n_0 < \ldots < n_{k-1}$ and
\[
U_\infty = \langle v_{-p} \rangle \oplus \langle v_{-p+2} \rangle \oplus \ldots \oplus \langle v_{-p+2k-2} \rangle.
\]
For describing $U_\infty$ in our applications, it will be convenient for us to use the following terminology: for every $i = 0,\ldots,k-1$ we say that $\langle v_{n_i} \rangle$ \textit{shifts to} $\langle v_{-p+2i} \rangle$ under the degeneration.
(Warning: in general, $\langle v_{-p+2i} \rangle$ is not the limit of $\langle v_{n_i} \rangle$ regarded as a one-dimensional subspace of~$V$!)
Observe that $-p+2i \le n_i$ for all~$i$.

\begin{example}
Let $V$ be the simple $\mathfrak{sl}_2$-module with highest weight~$6$ and consider the $h$-stable subspace $U = \langle v_{-6} \rangle \oplus \langle v_{-2} \rangle \oplus \langle v_4 \rangle \subset V$.
Then Proposition~\ref{prop_limitII_prelim} yields $U_\infty = \langle v_{-6} \rangle \oplus \langle v_{-4} \rangle \oplus \langle v_{-2} \rangle$.
This degeneration is illustrated in Figure~\ref{fig1} as follows.
The set of weights of~$V$ is expressed as a row of boxes (the ordering of weights is shown at the top).
Each $h$-stable subspace of $V$ is specified by a diagram obtained by putting balls into the boxes corresponding to the $h$-weights of the subspace.
Then, informally speaking, the diagram for $U_\infty$ is obtained from that for~$U$ by applying a horizontal leftward-directed force, which presses all the balls to the left edge of the row.
According to our terminology, under this degeneration the subspaces $\langle v_{-6} \rangle, \langle v_{-2} \rangle, \langle v_4 \rangle$ shift to $\langle v_{-6} \rangle, \langle v_{-4} \rangle, \langle v_{-2} \rangle$, respectively, which is indicated by arrows and agrees with the real shift of each ball.
\end{example}

\begin{figure}[h]
\caption{}\label{fig1}
\begin{tikzpicture}
\draw[->,semithick] (-1.8,-0.4) -- (-1.8, -1.4);
\draw[->,semithick] (-0.6,-0.4) to [out=-90, in=90] (-1.2, -1.4);
\draw[->,semithick] (1.2,-0.4) to [out=-90, in=90] (-0.6, -1.4);
\node [left] at (-2.3,0) {$U\colon$};
\node[above] at (-1.8,0.3) {\footnotesize --$6$};
\draw[fill] (-1.8,0) circle(0.19);
\draw (-2.1,0.3) rectangle (-1.5,-0.3);
\node[above] at (-1.2,0.3) {\footnotesize --$4$};
\draw (-1.5,0.3) rectangle (-0.9,-0.3);
\node[above] at (-0.6,0.3) {\footnotesize --$2$};
\draw[fill] (-0.6,0) circle(0.19);
\draw (-0.9,0.3) rectangle (-0.3,-0.3);
\node[above] at (0,0.3) {\footnotesize $0$};
\draw (0.3,0.3) rectangle (-0.3,-0.3);
\node[above] at (0.6,0.3) {\footnotesize $2$};
\draw (0.3,0.3) rectangle (0.9,-0.3);
\node[above] at (1.2,0.3) {\footnotesize $4$};
\draw[fill] (1.2,0) circle(0.19);
\draw (0.9,0.3) rectangle (1.5,-0.3);
\node[above] at (1.8,0.3) {\footnotesize $6$};
\draw (1.5,0.3) rectangle (2.1,-0.3);
\node [left] at (-2.3,-1.8) {$U_\infty\colon$};
\draw[fill] (-1.8,-1.8) circle(0.19);
\draw (-2.1,-1.5) rectangle (-1.5,-2.1);
\draw[fill] (-1.2,-1.8) circle(0.19);
\draw (-1.5,-1.5) rectangle (-0.9,-2.1);
\draw[fill] (-0.6,-1.8) circle(0.19);
\draw (-0.9,-1.5) rectangle (-0.3,-2.1);
\draw (0.3,-1.5) rectangle (-0.3,-2.1);
\draw (0.3,-1.5) rectangle (0.9,-2.1);
\draw (0.9,-1.5) rectangle (1.5,-2.1);
\draw (1.5,-1.5) rectangle (2.1,-2.1);
\draw[->,thick] (4.2,0.1) to (3,0.1);
\draw[->,thick] (4.2,-0.1) to (3,-0.1);
\node[above] at (3.65,0.1) {$\phi(t)$};
\end{tikzpicture}
\end{figure}

\section{Generalities on spherical varieties}
\label{sect_gen_SV}

Recall from the introduction that a normal irreducible $G$-variety $X$ is said to be spherical if $X$ possesses an open orbit for the induced action of~$B$ and a subgroup $H \subset G$ is said to be spherical if $G/H$ is a spherical homogeneous space.
In this situation, $\mathfrak h$ is referred to as a \textit{spherical subalgebra} of~$\mathfrak g$.

\subsection{Some combinatorial invariants of spherical varieties}
\label{subsec_comb_inv_SV}

Let $X$ be a spherical $G$-variety.
In this subsection, we introduce several combinatorial invariants of $X$ that will be needed in our paper.
All the invariants depend on the fixed choice of a Borel subgroup $B \subset G$.

For every $\lambda \in \mathfrak X(T)$ let $\CC(X)_\lambda^{(B)}$ be the space of $B$-semi-invariant rational functions on~$X$ of weight $\lambda$.
Then the \textit{weight lattice} of $X$ is by definition
\[
\Lambda_G(X) = \lbrace \lambda \in \mathfrak X(T) \mid \CC(X)_\lambda^{(B)} \ne \lbrace 0 \rbrace \rbrace.
\]
The \textit{rank} of~$X$ is defined as $\rk_G(X) = \rk \Lambda_G(X)$.
Since $B$ has an open orbit in~$X$, it follows that for every $\lambda \in \Lambda_G(X)$ the space $\CC(X)^{(B)}_\lambda$ has dimension~$1$ and hence is spanned by a nonzero function~$f_\lambda$.

We note that for two spherical subgroups $H_1,H_2 \subset G$ with $H_1 \subset H_2$ the lattice $\Lambda_G(G/H_2)$ is naturally identified with a sublattice of $\Lambda_G(G/H_1)$.
Moreover, if $H_1$ has a finite index in~$H_2$ then $\Lambda_G(G/H_2)$ is of finite index in $\Lambda_G(G/H_1)$ (see, for instance, \cite[Lemma~2.4]{Gan} or general results in \cite[Corollary~2(i,\,ii)]{Pa90}, \cite[Proposition~5.6, Theorem~9.1]{Tim}), so that $\QQ\Lambda_G(G/H_2) = \QQ\Lambda_G(G/H_1)$.

Put $\mathcal Q_G(X) = \Hom_\ZZ(\Lambda_G(X), \QQ)$.

Every discrete $\QQ$-valued valuation $v$ of the field $\CC(X)$ vanishing on $\CC^\times$ determines an element $\rho_v \in \mathcal Q_G(X)$ such that $\rho_v(\lambda) = v(f_\lambda)$ for all $\lambda \in \Lambda_G(X)$.
It is known that the restriction of the map $v \mapsto \rho_v$ to the set of $G$-invariant discrete $\QQ$-valued valuations of $\CC(X)$ vanishing on~$\CC^\times$ is injective (see~\cite[7.4]{LV} or~\cite[Corollary~1.8]{Kn91}) and its image is a finitely generated cone containing the image in~$\mathcal Q_G(X)$ of the antidominant Weyl chamber (see \cite[Proposition~3.2 and Corollary~4.1,~i)]{BriP} or~\cite[Corollary~5.3]{Kn91}).
We denote this cone by~$\mathcal V_G(X)$.
Results of~\cite[\S\,3]{Bri90} imply that $\mathcal V_G(X)$ is a cosimplicial cone in~$\mathcal Q_G(X)$.
Consequently, there is a uniquely determined linearly independent set $\Sigma_G(X)$ of primitive elements in~$\Lambda_G(X)$ such that
\[
\mathcal V_G(X) = \lbrace q \in \mathcal Q_G(X) \mid q(\sigma) \le 0 \ \text{for all} \ \sigma \in \Sigma_G(X) \rbrace.
\]
Elements of $\Sigma_G(X)$ are called \textit{spherical roots} of~$X$ and $\mathcal V_G(X)$ is called the \textit{valuation cone} of~$X$.
The above discussion implies that every spherical root is a nonnegative linear combination of simple roots.

\begin{proposition}[{see~\cite[Corollary~5.3]{BriP}}] \label{prop_finite_index_in_norm}
Let $H \subset G$ be a spherical subgroup.
The set $\Sigma_G(G/H)$ is a basis of the vector space $\QQ\Lambda_G(G/H)$ if and only if the group $N_G(H)/H$ is finite.
\end{proposition}

As can be easily seen from the definitions, the weight lattice and spherical roots depend only on the open $G$-orbit in~$X$.

If $X$ is quasi-affine then there is a finer invariant than the weight lattice.
Consider the natural action of $G$ on $\CC[X]$.
The highest weights of all simple $G$-modules occurring in $\CC[X]$ form a monoid, called the \textit{weight monoid} of~$X$; we denote it by~$\Gamma_G(X)$.
The following result is well known; for a proof see, for instance,~\cite[Proposition~5.14]{Tim}.

\begin{proposition} \label{prop_SM_WL_and_WM}
If $X$ is quasi-affine then $\Lambda_G(X) = \ZZ \Gamma_G(X)$.
\end{proposition}

Suppose again that $X$ is quasi-affine.
Then the spherical roots of~$X$ admit an alternative description as follows.
Thanks to~\cite[Theorem~2]{VK78}, the sphericity of $X$ is equivalent to $\CC[X]$ being a multiplicity-free $G$-module.
For every $\lambda \in \Gamma_G(X)$, let $\CC[X]_\lambda \subset \CC[X]$ be the unique simple $G$-submodule with highest weight~$\lambda$.
Let $\mathcal T_G(X) \subset \Lambda_G(X)$ be the monoid generated by all weights of the form $\lambda + \mu - \nu$ where $\lambda, \mu, \nu \in \Gamma_G(X)$ and the linear span of the product $\CC[X]_\lambda \cdot \CC[X]_\mu$ contains $\CC[X]_\nu$.
The following result is a particular case of~\cite[Lemma~6.6, iii)]{Kn96}.

\begin{proposition} \label{prop_tail_cone}
If $X$ is quasi-affine then $\QQ^+\Sigma_G(X) = \QQ^+\mathcal T_G(X)$.
\end{proposition}

As a concluding remark, we mention the following observation.
If a central subgroup $Z \subset G$ acts trivially on~$X$ then $X$ can be regarded as a spherical $G/Z$-variety.
In this case it is easy to see that all the above invariants of~$X$ as a spherical $G$-variety naturally identify with the corresponding invariants of~$X$ as a spherical $G/Z$-variety.

\subsection{Parabolic induction}
\label{subsec_par_ind}

Let $Q \supset B^-$ be a parabolic subgroup of~$G$ and let $L$ be the standard Levi subgroup of~$Q$.
Given a spherical subgroup $F \subset L$, put $H = F \rightthreetimes Q_u$.
Then $H$ is a spherical subgroup of~$G$.
In this case, we say that the homogeneous space $G/H$ is \textit{parabolically induced} from $L/F$.

In this paper, we shall need the following statement, which is implied by the general result \cite[Proposition~20.13]{Tim}.

\begin{proposition} \label{prop_par_ind}
Under the above notation, $\Lambda_G(G/H) = \Lambda_L(L/F)$ and $\Sigma_G(G/H) = \Sigma_L(L/F)$.
\end{proposition}

\subsection{Spherical modules}
\label{subsec_SM_gen}

Let $V$ be a finite-dimensional $G$-module.
We say that $V$ is a \textit{spherical $G$-module} if $V$ is spherical as a $G$-variety.
As was already mentioned in~\S\,\ref{subsec_comb_inv_SV}, according to~\cite[Theorem~2]{VK78}, $V$ is spherical if and only if the $G$-module $\CC[V]$ is multiplicity free.
From the latter property (or directly from the definition) it is easily deduced that every submodule of a spherical $G$-module is again spherical and $V$ is spherical if and only if so is~$V^*$.

Until the end of this subsection we assume that $V$ is a spherical $G$-module.
It is well known that the weight monoid $\Gamma_G(V)$ is free; see, for instance,~\cite[Theorem~3.2]{Kn98}.
Let $F_G(V)$ denote the set of indecomposable elements of~$\Gamma_G(V)$.
By definition, $F_G(V)$ is a linearly independent subset of $\Lambda^+_G$ such that $\Gamma_G(V) = \ZZ^+F_G(V)$.

Consider a decomposition
\begin{equation} \label{eqn_direct_sum}
V = V_1 \oplus \ldots \oplus V_k
\end{equation}
into a direct sum of simple $G$-modules and for each $i=1,\ldots, k$ let $\lambda_i$ be the highest weight of~$V_i^*$.
As $\CC[V] \simeq \mathrm{S}(V^*)$, we find that $\lambda_i \in F_G(V)$ for all $i=1,\ldots,k$.
For every $\lambda \in \Gamma_G(V)$, let $R_\lambda \subset \mathrm{S}(V^*)$ be the simple $G$-submodule with highest weight~$\lambda$.

Now suppose $V$ is a spherical module with respect to an action of another connected reductive algebraic group $K$ satisfying $K'=G'$.
For every $\lambda \in \Gamma_G(V)$, let $\varphi(\lambda)$ be the highest weight of $R_\lambda$ regarded as a $K$-module.
The following proposition is well known; we provide it together with a proof for convenience.

\begin{proposition} \label{prop_free_generators}
The following assertions hold.
\begin{enumerate}[label=\textup{(\alph*)},ref=\textup{\alph*}]
\item \label{prop_free_generators_a}
The above-defined map $\varphi \colon \Gamma_G(V) \to \Gamma_K(V)$ is an isomorphism.
In particular, $\varphi$ induces a bijection $F_G(V) \to F_K(V)$.

\item \label{prop_free_generators_b}
For every $\lambda \in \Gamma_G(V)$, there is an expression $\lambda = \sum \limits_{i=1}^k a_i \lambda_i - \sum \limits_{\alpha \in \Pi} b_\alpha \alpha$ with $a_i, b_\alpha \in \ZZ^+$ such that $\varphi(\lambda) = \sum \limits_{i=1}^k a_i \varphi(\lambda_i) - \sum \limits_{\alpha \in \Pi} b_\alpha \alpha$.
In particular, this property holds for all $\lambda \in F_G(V) \setminus \lbrace \lambda_1,\ldots, \lambda_k \rbrace$.

\item \label{prop_free_generators_c}
$\varphi$ extends to an isomorphism $\varphi \colon \Lambda_G(V) \xrightarrow{\sim} \Lambda_K(V)$.

\item \label{prop_free_generators_d}
$\Sigma_G(V) = \Sigma_K(V)$.
\end{enumerate}
\end{proposition}

\begin{proof}
(\ref{prop_free_generators_a})
Clearly, $\varphi$ is bijective.
It remains to notice that for every $\lambda, \mu \in \Gamma_G(V)$ the product of highest weight vectors of $R_\lambda$ and $R_\mu$ is a highest weight vector of $R_{\lambda+\mu}$.

(\ref{prop_free_generators_b})
Decomposition~(\ref{eqn_direct_sum}) yields a $G$-module isomorphism
\begin{equation} \label{eqn_S(V)}
\mathrm S(V^*) \simeq \bigoplus \limits_{(s_1,\ldots,s_k) \in (\ZZ^+)^k} \mathrm S^{s_1}(V_1^*) \otimes \ldots \otimes \mathrm S^{s_k}(V_k^*).
\end{equation}
Then for every $\lambda \in \Gamma_G(V)$ there is a tuple $(a_1,\ldots,a_k) \in (\ZZ^+)^k$ such that $R_\mu$ is isomorphic to a submodule of $\mathrm S^{a_1}(V_1^*) \otimes \ldots \otimes \mathrm S^{a_k}(V_k^*)$.
Since each $T$-weight of $V_i^*$ is obtained from $\lambda_i$ by subtracting a linear combination of the roots in~$\Pi$ with nonnegative integer coefficients, it follows that $\lambda = \sum \limits_{i=1}^k a_i \lambda_i - \sum \limits_{\alpha \in \Pi} b_\alpha \alpha$ for some $b_\alpha \in \ZZ^+$.
Clearly, in this situation one also has $\varphi(\lambda) = \sum \limits_{i=1}^k a_i \varphi(\lambda_i) - \sum \limits_{\alpha \in \Pi} b_\alpha \alpha$.

(\ref{prop_free_generators_c})
This is implied by Proposition~\ref{prop_SM_WL_and_WM}.

(\ref{prop_free_generators_d})
Note that the isomorphism $\varphi$ in~(\ref{prop_free_generators_c}) sends each element of $\Lambda_G(V) \cap \QQ\Pi$ to itself.
Then an argument similar to that in~(\ref{prop_free_generators_b}) yields $\mathcal T_G(V) = \mathcal T_K(V)$, hence $\varphi(\mathcal T_G(V)) = \mathcal T_K(V)$.
Since the set of spherical roots is linearly independent and each spherical root is primitive in the weight lattice, Proposition~\ref{prop_tail_cone} yields
$\Sigma_G(V) = \Sigma_K(V)$.
\end{proof}

\begin{remark} \label{rem_restriction}
~
\begin{enumerate}[label=\textup{(\alph*)},ref=\textup{\alph*}]
\item \label{rem_restriction_a}
If $K$ is a subgroup of $G$ then the map $\varphi$ and its extension to $\Lambda_G(V)$ are given by restricting characters from $B$ to~$B_K$.

\item \label{rem_restriction_b}
The proof of Proposition~\ref{prop_free_generators}(\ref{prop_free_generators_b}) implies $\Gamma_G(V) \subset \ZZ^+F_G(V) - \mathcal T_G(V)$.
In view of Propositions~\ref{prop_tail_cone} and~\ref{prop_SM_WL_and_WM} this also yields $\Gamma_G(V) \subset \ZZ^+F_G(V) - \ZZ^+\Sigma_G(V)$.
\end{enumerate}
\end{remark}

Consider again decomposition~(\ref{eqn_direct_sum}) and
let $Z$ be the subgroup of $\GL(V)$ consisting of the elements that act by scalar transformations on each~$V_i$, $i = 1,\ldots,k$.
Let $C \subset Z$ be the image in~$\GL(V)$ of the connected center of~$G$.
We say that $V$ is \textit{saturated} (as a $G$-module) if $C = Z$.
In the general case, one can find a subtorus $C_0 \subset Z$ such that $Z = C \times C_0$, and then $V$ becomes saturated (and of course remains spherical) when regarded as a ($G \times C_0$)-module.
Thus an arbitrary spherical module is obtained from a saturated one by reducing the connected center of the acting group.

\begin{proposition} \label{prop_SR_of_SM}
One has $\QQ\Lambda_G(V) = \QQ\lbrace \lambda_1,\ldots, \lambda_k \rbrace \oplus \QQ \Sigma_G(V)$.
\end{proposition}

\begin{proof}
In view of Proposition~\ref{prop_free_generators}(\ref{prop_free_generators_c},\,\ref{prop_free_generators_d}) and Remark~\ref{rem_restriction}(\ref{rem_restriction_a}) it suffices to prove the assertion for saturated~$V$.
Taking into account Remark~\ref{rem_restriction}(\ref{rem_restriction_b}) and Proposition~\ref{prop_SM_WL_and_WM} we obtain $\QQ\Lambda_G(V) = \QQ\lbrace \lambda_1, \ldots, \lambda_k \rbrace + \QQ\Sigma_G(V)$.
As the restrictions of $\lambda_1,\ldots,\lambda_k$ to the connected center of~$G$ are linearly independent and those of $\Sigma_G(V)$ are trivial, we conclude that $\QQ\lbrace \lambda_1, \ldots, \lambda_k \rbrace \cap \QQ\Sigma_G(V) = \lbrace 0 \rbrace$.
\end{proof}

\begin{remark} \label{rem_disconnected}
In~\S\,\ref{subsec_gen_SS}, we shall also deal with disconnected groups of the form $\widehat G = \widehat Z \cdot G$ where $\widehat Z$ is a finite Abelian group.
A finite-dimensional $\widehat G$-module will be called spherical if $V$ is spherical as a $G$-module.
One easily extends the notions of weight lattice and weight monoid to spherical $\widehat G$-modules by considering $(\widehat Z \cdot B)$-weights instead of $B$-weights in both definitions, so that $\Lambda_{\widehat G}(V), \Gamma_{\widehat G}(V) \subset \mathfrak X(\widehat Z \cdot B)$ for every spherical $\widehat G$-module~$V$.
All results mentioned in this subsection remain valid for spherical $\widehat G$-modules.
In particular, the monoid $\Gamma_{\widehat G}(V)$ is free thanks to the natural restriction isomorphism $\Gamma_{\widehat G}(V) \xrightarrow{\sim} \Gamma_G(V)$.
\end{remark}

\subsection{A reduction for spherical modules}
\label{subsec_char_of_SM}

Let $V$ be a finite-dimensional $G$-module (not necessarily simple) and let $\omega$ be a highest weight of~$V$.
Fix a highest-weight vector $v_\omega \in V$ of weight~$\omega$.
Put $Q = \lbrace g \in G \mid g\langle v_\omega \rangle = \langle v_\omega \rangle \rbrace$; this is a parabolic subgroup of~$G$ containing~$B$.
Let $M$ be the standard Levi subgroup of~$Q$ and let $M_0$ be the stabilizer of~$v_\omega$ in~$M$.
Let $Q^- \supset B^-$ be the parabolic subgroup of~$G$ opposite to $Q$.
Fix a lowest weight vector $\xi \in V^*$ of weight~$-\omega$, so that $\xi(v_\omega) \ne 0$.
Put
\[
\widetilde V = \lbrace v \in V \mid (\mathfrak q_u \xi)(v) = 0 \rbrace = \lbrace v \in V \mid \xi(\mathfrak q_u v) = 0 \rbrace
\]
and $V_0 = \widetilde V \cap \Ker \xi$.
Both $V$ and $V_0$ are $M$-modules in a natural way, and there are the decompositions $\widetilde V = \langle v_\omega \rangle \oplus V_0$ and $V = (\mathfrak q^-_u v_\omega) \oplus \widetilde V$ into direct sums of $M$-submodules.
The following result is extracted from the proof of \cite[Theorem~3.3]{Kn98}.

\begin{theorem} \label{thm_sph_modules_imp}
The following assertions hold:
\begin{enumerate}[label=\textup{(\alph*)},ref=\textup{\alph*}]
\item \label{thm_sph_modules_imp_a}
$V$ is a spherical $G$-module if and only if $\widetilde V$ is a spherical $M$-module.

\item \label{thm_sph_modules_imp_b}
Under the conditions of~\textup(\ref{thm_sph_modules_imp_a}\textup), one has $\Lambda_G(V^*) = \Lambda_M(\widetilde V^*)$.
\end{enumerate}
\end{theorem}

In the next two propositions, we assume that $V$ is a spherical $G$-module.

\begin{proposition} \label{prop_SM_red_SR1}
One has
$\Sigma_M(V_0^*)=\Sigma_M(\widetilde V^*) \subset \Sigma_G(V^*)$.
\end{proposition}

\begin{proof}
Applying \cite[Proposition~3.2]{Gag} in the situation described in the proof of \cite[Theorem~3.3]{Kn98} we find that the so-called homogeneous spherical datum of the open $M$-orbit in $\widetilde V^*$ is obtained from that of the open $G$-orbit in $V^*$ via a so-called localization at the set of simple roots~$\Pi_M$.
The latter implies $\Sigma_M(\widetilde V^*) \subset \Sigma_G(V^*)$.
The $M$-module isomorphism $\widetilde V^* \simeq \langle v_\omega \rangle^* \oplus V_0^*$ yields the equalities $\Lambda_M(\widetilde V^*) = \ZZ \omega \oplus \Lambda_M(V_0^*)$ and $\mathcal T_M(\widetilde V^*) = \mathcal T_M(V_0^*)$, which imply $\Sigma_M(\widetilde V^*) = \Sigma_M(V_0^*)$ thanks to Proposition~\ref{prop_tail_cone}.
\end{proof}

\begin{proposition} \label{prop_SM_red_SR2}
Suppose $W \subset V$ is a simple $G$-submodule with highest weight~$\lambda$ and $W' \subset W \cap V_0$ is a simple $M$-submodule with highest weight~$\mu$.
Then $\lambda - \mu \in \QQ\Sigma_G(V^*)$.
\end{proposition}

\begin{proof}
In view of Proposition~\ref{prop_free_generators}(\ref{prop_free_generators_c},\,\ref{prop_free_generators_d}) and Remark~\ref{rem_restriction}(\ref{rem_restriction_a}) it suffices to assume that $V$ is saturated.
Clearly, $\lambda - \mu \in \Lambda_M(\widetilde V^*)$, therefore $\lambda - \mu \in \Lambda_G(V^*)$ by Theorem~\ref{thm_sph_modules_imp}(\ref{thm_sph_modules_imp_b}).
The condition $W' \subset W$ implies that $\lambda - \mu$ is a linear combination of simple roots with nonnegative coefficients.
Since $V$ is saturated, the claim follows from Proposition~\ref{prop_SR_of_SM}.
\end{proof}

\subsection{Classification of spherical modules and some consequences}
\label{subsec_SM_classification}

All modules considered in this subsection are assumed to be finite-dimensional.
The terminology in this subsection follows Knop \cite[\S\,5]{Kn98}.

Given two connected reductive algebraic groups $G_1,G_2$, for $i=1,2$ let $V_i$ be a $G_i$-module and let $\rho_i \colon G \to \GL(V_i)$ be the corresponding representation.
The pairs $(G_1,V_1)$ and $(G_2,V_2)$ are said to be \textit{geometrically equivalent} (or just \textit{equivalent} for short) if there exists an isomorphism of vector spaces $V_1 \xrightarrow{\sim} V_2$ identifying the groups $\rho_1(G_1)$ and $\rho_2(G_2)$.
As an important example, note that for any $G$-module $V$ the pairs $(G,V)$ and $(G,V^*)$ are equivalent.

Observe that for a $G$-module $V$ the property of being spherical depends only on the equivalence class of the pair~$(G,V)$.

A complete classification of simple spherical modules was obtained in \cite{Kac} and is given by the following theorem.

\begin{theorem} \label{thm_simple_sph_modules}
Let $V$ be a simple $G$-module and let $\chi$ be the character via which the connected center of $G$ acts on~$V$.
The $G$-module $V$ is spherical if and only if the following two conditions hold:
\begin{enumerate}[label=\textup{(\arabic*)},ref=\textup{\arabic*}]
\item
up to equivalence, the pair $(G',V)$ appears in Table~\textup{\ref{table_spherical_modules}};

\item
the character $\chi$ satisfies the conditions listed in the fourth column of Table~\textup{\ref{table_spherical_modules}}.
\end{enumerate}
\end{theorem}

\begin{table}[h]

\caption{} \label{table_spherical_modules}

\begin{center}

\begin{tabular}{|c|c|c|c|c|c|}
\hline

No. & $G'$ & $V$ & Conditions on $\chi$ & $\wedge^2 V$ & Note \\

\hline

\no \label{no1} &
$\lbrace e \rbrace$ & $\CC^1$ & $\chi \ne 0$ & $\lbrace 0 \rbrace$ & \\

\hline

\no & $\SL_n$ & $\CC^n$ & & $\wedge^2 \CC^n$ & $n \ge 2$ \\

\hline

\no & $\Sp_{2n}$ & $\CC^{2n}$ & & $R(\varpi_2) \oplus \CC^1$ & $n \ge 2$ \\

\hline

\no & $\SO_n$ & $\CC^n$ & $\chi \ne 0$ & $\wedge^2 \CC^n$ & $n \ge 5$ \\

\hline

\no & $\SL_n$ & $\mathrm S^2 \CC^n$ & $\chi \ne 0$ & $R(2\varpi_1 + \varpi_2)$ & $n \ge 2$ \\

\hline

\no \label{no6} & $\SL_n$ & $\wedge^2 \CC^n$ & $\chi \ne 0$ for $n$ even & $R(\varpi_1+\varpi_3)$ &
$n \ge 5$ \\

\hline

\no \label{no7} & $\SL_n \times \SL_m$ & $\CC^n {\otimes} \CC^m$ & $\chi \ne 0$ for
$n = m$ &
\renewcommand{\tabcolsep}{0pt}%
\begin{tabular}{l}
$\wedge^2 \CC^n {\otimes} \mathrm{S}^2\CC^m \oplus$\\
$\mathrm{S}^2 \CC^n {\otimes} {\wedge^2} \CC^m$
\end{tabular} &
$n, m \ge 2$
\\

\hline

\no & $\SL_2 \times \Sp_{2n}$ & $\CC^2 {\otimes} \CC^{2n}$ & $\chi \ne
0$ &
\renewcommand{\tabcolsep}{0pt}%
\begin{tabular}{l}
$\mathrm{S}^2\CC^{2n} \oplus \mathrm{S}^2 \CC^{2} \oplus$\\
$\mathrm{S}^2 \CC^2 {\otimes} R(\varpi_2')$
\end{tabular} &
$n \ge 2$
\\

\hline

\no &
$\SL_3 \times \Sp_{2n}$ &
$\CC^3 {\otimes} \CC^{2n}$ &
$\chi \ne 0$ &
\renewcommand{\tabcolsep}{0pt}%
\begin{tabular}{l}
$\wedge^2 \CC^3 {\otimes} \mathrm{S}^2\CC^{2n} \oplus$\\
$\mathrm{S}^2 \CC^3 {\otimes} R(\varpi'_2) \oplus \mathrm{S}^2 \CC^3$
\end{tabular} &
$n \ge 2$
\\

\hline

\no \label{no10} &
$\SL_n \times \Sp_4$ &
$\CC^n {\otimes} \CC^4$ &
$\chi \ne 0$ for $n = 4$ &
\renewcommand{\tabcolsep}{0pt}%
\begin{tabular}{l}
$\wedge^2 \CC^n {\otimes} \mathrm{S}^2\CC^{4} \oplus$\\
$\mathrm{S}^2 \CC^n {\otimes} R(\varpi'_2) \oplus \mathrm{S}^2 \CC^n$
\end{tabular} &
$n \ge 4$ \\

\hline

\no \label{no11} & $\Spin_7$ & $R(\varpi_3)$ & $\chi \ne 0$ & $R(\varpi_1) \oplus R(\varpi_2)$ & \\

\hline

\no \label{no12} & $\Spin_9$ & $R(\varpi_4)$ & $\chi \ne 0$ & $R(\varpi_2) \oplus R(\varpi_3)$ & \\

\hline

\no \label{no13} & $\Spin_{10}$ & $R(\varpi_5)$ & & $R(\varpi_3)$ & \\

\hline

\no \label{no14} & $\mathsf G_2$ & $R(\varpi_1)$ & $\chi \ne 0$ & $R(\varpi_1) \oplus R(\varpi_2)$ & \\

\hline

\no \label{no15} & $\mathsf E_6$ & $R(\varpi_1)$ & $\chi \ne 0$ & $R(\varpi_3)$ & \\

\hline
\end{tabular}
\end{center}
\end{table}

In Table~\ref{table_spherical_modules}, the notation $R(\lambda)$ stands for the simple $G'$-module with highest weight~$\lambda$ and $\varpi_i$ (resp.~$\varpi_i'$) denotes the $i$th fundamental weight of the first (resp.~second) factor of~$G'$.

In this paper, we shall need Propositions~\ref{prop_simple_SM_wedge1} and~\ref{prop_simple_SM_wedge2} stated below.
Their proofs require the decomposition of $\wedge^2 V$ into simple $G$-submodules for all simple spherical $G$-modules~$V$.
For each case in Table~\ref{table_spherical_modules}, we provide this decomposition (into simple $G'$-modules) in the fifth column.
For cases~\ref{no1}--\ref{no6} this decomposition is classical.
For cases~\ref{no7}--\ref{no10} it follows from the well-known $\GL(U) \times \GL(W)$-module isomorphism
\[
\wedge^2 (U \otimes W) \simeq \wedge^2 U \otimes \mathrm S^2 W \oplus \mathrm S^2 U \otimes \wedge^2 W
\]
(where $U,W$ are arbitrary finite-dimensional vector spaces).
For cases~\ref{no11}--\ref{no15}, the decomposition can be computed using the program LiE~\cite{LiE}.

\begin{proposition} \label{prop_simple_SM_wedge1}
Suppose that $V$ is a simple spherical $G$-module.
Then $V$ is not isomorphic to a submodule of $\wedge^2 V$.
\end{proposition}

\begin{proof}
Assuming the converse we obtain that the connected center of~$G$ acts trivially on~$V$.
Then the proof is completed by a case-by-case check of all such cases in Table~\ref{table_spherical_modules}.
\end{proof}

Before discussing the classification of non-simple spherical modules, we need to introduce several additional notions.

We say that a $G$-module $V$ is \textit{decomposable} if there exist connected reductive algebraic groups $G_1, G_2$, a $G_1$-module~$V_1$, and a $G_2$-module~$V_2$ such that the pair $(G,V)$ is equivalent to $(G_1 \times G_2, V_1 \oplus V_2)$.
Clearly, in this situation $V$ is a spherical $G$-module if and only if $V_i$ is a spherical $G_i$-module for $i=1,2$, in which case there is a disjoint union $F_G(V) = F_{G_1}(V_1) \cup F_{G_2}(V_2)$.
We say that $V$ is \textit{indecomposable} if $V$ is not decomposable.

Recall the notation from the paragraph after Remark~\ref{rem_restriction}.
The $(G \times C_0)$-module $V$ is called the \textit{saturation} of the $G$-module~$V$.
Note that the pair $(G \times C_0, V)$ is equivalent to $(G' \times Z, V)$.

A complete classification (up to equivalence) of all indecomposable saturated non-simple spherical modules was independently obtained in~\cite{BR} and~\cite{Lea}.
Both papers contain also a description of all spherical modules with a given saturation, which completes the classification of all spherical modules.
The weight monoids of all indecomposable saturated spherical modules are known thanks to the papers~\cite{HU} (the case of simple modules) and~\cite{Lea} (the case of non-simple modules).
A complete list (up to equivalence) of all indecomposable saturated spherical modules can be found in~\cite[\S\,5]{Kn98} along with various additional data, including the indecomposable elements of the weight monoids.

\begin{proposition} \label{prop_simple_SM_wedge2}
Suppose that $V$ is a simple spherical $G$-module and $W$ is a nonzero $G$-submodule of $\wedge^2 V$ such that $V \oplus W$ is a spherical $G$-module.
Then one of the following two possibilities is realized:
\begin{enumerate}[label=\textup{(\arabic*)},ref=\textup{\arabic*}]
\item
$G' \simeq \SL_n$ \textup($n \ge 2$\textup) and there are $G'$-module isomorphisms $V \simeq \CC^n$, $W \simeq \wedge^2 \CC^n$;

\item
$G' \simeq \Sp_{2n}$ \textup($n \ge 2$\textup) and there are $G'$-module isomorphisms $V \simeq \CC^{2n}$, $W \simeq \CC^1$.
\end{enumerate}
\end{proposition}

\begin{proof}
For each pair $(G',V)$ in Table~\ref{table_spherical_modules}, we consider all possible nontrivial $G'$-submodules $W \subset \wedge^2 V$ and check via the list in~\cite[\S\,5]{Kn98} whether the saturation of the $G'$-module $V \oplus W$ is a spherical module.
This case-by-case analysis is substantially shortened by observing that the saturation of the $G'$-module $W$ itself should be a spherical module.
The analysis ultimately yields only the two cases listed in the statement, which completes the proof.
\end{proof}

\subsection{Generalities on spherical subgroups}
\label{subsec_gen_SS}

Let $H \subset G$ be a subgroup.
By \cite[\S\,30.3]{Hum} there exists a parabolic subgroup $P \subset G$ such that $H \subset P$ and $H_u \subset P_u$.
In this situation, we say that $H$ is \textit{regularly embedded} in~$P$.
One can choose Levi subgroups $L \subset P$ and $K \subset H$ in such a way that $K \subset L$.
Then by \cite[Lemma~1.4]{Mon} there is a $K$-equivariant isomorphism $P_u/H_u \simeq \mathfrak p_u / \mathfrak h_u$.

Let $S \subset K$ be a generic stabilizer for the natural action of $K$ on $\mathfrak l / \mathfrak k$.
According to \cite[Theorem~2.3, Corollary~2.4, Corollary~8.2]{Kn90} and \cite[Theorem~1(iii), Theorem~3(ii), and~\S\,2.1]{Pa90}, the subgroup $S$ has the following properties:
\begin{enumerate}[label=\textup{(S\arabic*)},ref=\textup{S\arabic*}]
\item \label{S1}
$S$ is reductive;

\item \label{S2}
there are a parabolic subgroup $Q_L \subset L$ and a Levi subgroup $M \subset Q_L$ such that $S = Q_L \cap K$ and $M' \subset S \subset M$.
\end{enumerate}
It is worth mentioning that $S$ may be disconnected; however, the second property guarantees that $S = Z(S)\cdot (S^0)'$.
Then the notions of a spherical $S$-module and its weight lattice extend as described in Remark~\ref{rem_disconnected}.

Replacing $H$, $S$, and $Q_L$ with conjugate subgroups we may assume that $P \supset B^-$, $L$ is the standard Levi subgroup of~$P$, $Q_L \supset B_L$, and $M$ is the standard Levi subgroup of~$Q_L$.
Let $\iota \colon \mathfrak X(T) \to \mathfrak X(T \cap S)$ be the character restriction map.

The following theorem is implied by~\cite[Proposition~I.1]{Br87} and~\cite[Theorem~1.2]{Pa94}; see also~\cite[Theorem~9.4]{Tim}.

\begin{theorem} \label{thm_criterion_spherical}
Under the above notation and assumptions, the following conditions are equivalent.
\begin{enumerate}[label=\textup{(\arabic*)},ref=\textup{\arabic*}]
\item
$H$ is a spherical subgroup of~$G$.

\item
$P/H$ is a spherical $L$-variety.

\item
$K$ is a spherical subgroup of~$L$ and $\mathfrak p_u/\mathfrak h_u$ is a spherical $S$-module.
\end{enumerate}
Moreover, if these conditions hold then $\Ker \iota = \Lambda_L(L/K)$ and $\Lambda_G(G/H) = \iota^{-1}(\Lambda_S(\mathfrak p_u/\mathfrak h_u))$ where the lattices $\Lambda_L(L/K)$ and $\Lambda_S(\mathfrak p_u/\mathfrak h_u)$ are taken with respect to $B_L$ and~$B_S$, respectively.
\end{theorem}

It follows from Theorem~\ref{thm_criterion_spherical} that the problem of computing the weight lattice for a spherical homogeneous space $G/H$ reduces to determining the subgroup~$S$ for the affine spherical homogeneous space $L/K$ and finding the weight lattice for the spherical $S$-module $\mathfrak p_u/ \mathfrak h_u$.
As was already mentioned in the introduction, there is a complete classification of all affine spherical homogeneous spaces, and it turns out that the subgroups $S$ for all such spaces are known; see, for instance,~\cite{KnVS}.
Similarly, as was already discussed in~\S\,\ref{subsec_SM_classification}, there is a complete classification of all spherical modules, and the weight monoids (and hence weight lattices) are also known for all of them.
Thus we get an effective procedure for computing $\Lambda_G(G/H)$.
(In fact, this procedure is not yet completely explicit since one needs to take care of appropriate choices of $K$ and $S$ within their conjugacy classes.)

Throughout this paper, we shall work only with subgroups $H$ satisfying $L' \subset K \subset L$.
For such subgroups, $K$ acts trivially on $\mathfrak l/\mathfrak k$, hence $S$ is unique and coincides with $K$.
In this case, $\iota$ is the character restriction map from $T$ to $T \cap K$ and Theorem~\ref{thm_criterion_spherical} takes the following simpler form.

\begin{proposition} \label{prop_S=K}
Suppose that $L' \subset K \subset L$.
Then $H$ is a spherical subgroup of~$G$ if and only if $\mathfrak p_u/\mathfrak h_u$ is a spherical $K$-module.
Moreover, under these conditions one has $\Lambda_G(G/H) = \iota^{-1}(\Lambda_K(\mathfrak p_u / \mathfrak h_u))$ where the lattice $\Lambda_K(\mathfrak p_u /\mathfrak h_u)$ is taken with respect to~$B_K$.
\end{proposition}

\begin{remark}
Under the conditions of Theorem~\ref{thm_criterion_spherical}, some partial results on the set of spherical roots of $G/H$ were obtained in~\cite{Pez}.
Namely, Corollary~8.12 and Theorem~6.15 in loc. cit. assert that
\[
\Sigma_L(P/H) = \lbrace \sigma \in \Sigma_G(G/H) \mid \Supp \sigma \subset \Pi_L \rbrace \ \: \text{and} \ \: \Sigma_G(G/H) \setminus \Sigma_L(P/H) \subset \Delta^+ \setminus \Delta^+_L,
\]
respectively.
In particular, in the situation of Proposition~\ref{prop_S=K} one has $\Sigma_L(P/H) = \Sigma_K(\mathfrak p_u / \mathfrak h_u)$ and thus each spherical root of the spherical $K$-module $\mathfrak p_u / \mathfrak h_u$ is automatically a spherical root of $G/H$.
Since the spherical roots of all spherical modules are known from~\cite[\S\,5]{Kn98}, in this way one may obtain all spherical roots of $G/H$ whose support is contained in~$\Pi_L$.
\end{remark}

\subsection{Normalizers of spherical subgroups}
\label{subsec_norm_SS}

Let $H \subset G$ be a spherical subgroup.

\begin{theorem}[{\cite[Corollary~5.2]{BriP}}] \label{thm_BriP}
The following assertions hold:
\begin{enumerate}[label=\textup{(\alph*)},ref=\textup{\alph*}]
\item \label{thm_BriP_a}
The group $N_G(H)/H$ is diagonalizable.

\item  \label{thm_BriP_b}
$N_G(H) = N_G(H^0)$.
\end{enumerate}
\end{theorem}

\begin{corollary} \label{crl_unipotent_elements}
Every unipotent element of $N_G(H)$ is contained in~$H$.
\end{corollary}

In what follows, we put $N = N_G(H)$ for short.

\begin{corollary} \label{crl_Nu=Hu}
One has $N_u = H_u$.
\end{corollary}

\begin{proof}
Clearly, $N$ normalizes $H_u$, which implies $N_u \supset H_u$.
Then $N_u \subset H$ by Corollary~\ref{crl_unipotent_elements}.
Since $H$ normalizes $N_u$, it follows that $N_u = H_u$.
\end{proof}

\begin{corollary}
Suppose that $H$ is reductive.
Then $N$ is also reductive and $N^0 = (Z_G(H^0) (H^0)')^0$.
\end{corollary}

\begin{proof}
The group $N$ is reductive by Corollary~\ref{crl_Nu=Hu}.
Then $(N^0)' = (H^0)'$ by Corollary~\ref{crl_unipotent_elements} and hence $N^0 \subset Z_G(H^0) (H^0)'$.
On the other hand, $Z_G(H^0) (H^0)' \subset N_G(H^0) = N$ where the equality holds by Theorem~\ref{thm_BriP}(\ref{thm_BriP_b}).
\end{proof}

The next result follows essentially from~\cite[\S\,5.4]{BriP}; see \cite[Lemma~4.25]{ACF} for details.

\begin{proposition} \label{prop_norm_sph_roots}
Let $\widetilde H \subset G$ be a spherical subgroup satisfying $H \subset \widetilde H \subset N$.
Then, modulo the inclusion $\Lambda_G(G/\widetilde H) \hookrightarrow \Lambda_G(G/H)$, one has $\QQ^+\Sigma_G(G/\widetilde H) = \QQ^+\Sigma_G(G/H)$.
\end{proposition}

Now suppose that $H$ is regularly embedded in a parabolic subgroup $P \subset G$ and moreover $K \subset L$ for Levi subgroups $K \subset H$ and $L \subset P$.
The following proposition provides an explicit description of the group $N^0$.

\begin{proposition} \label{prop_normalizer}
Under the above assumptions, put $A = Z_L(K^0) \cap N_L(\mathfrak h_u)$.
Then $N^0 = A^0(K^0)' \rightthreetimes H_u$.
\end{proposition}

\begin{proof}
Thanks to Theorem~\ref{thm_BriP}(\ref{thm_BriP_b}), it suffices to consider the case of connected~$H$, in which we need to prove that $N^0 = A^0K' \rightthreetimes H_u$.
It is clear from the definition that $A \subset N$, hence $A^0K' \rightthreetimes H_u \subset N^0$.
Next, we know from Corollary~\ref{crl_Nu=Hu} that $N_u = H_u$.
Let $M$ be a Levi subgroup of~$N^0$ such that $M \supset K$ and let $C_M$ be the connected center of~$M$.
As the group $M'$ is generated by its unipotent subgroups, Corollary~\ref{crl_unipotent_elements} implies $M' \subset K$ and $M' = K'$, so that $N^0 = C_MK' \rightthreetimes H_u$.
By \cite[Proposition~3.26]{Avd_solv_inv} we also know that $N^0$ is contained in~$P$.
It remains to show that $C_M \subset A^0K' \rightthreetimes H_u$.
We shall prove the stronger claim $\mathfrak c_M \subset \mathfrak a + \mathfrak h_u$.
Take any element $x \in \mathfrak c_M$ and consider the decomposition $x = y + z$ where $y \in \mathfrak l$ and $z \in \mathfrak p_u$.
Clearly, $[x, \mathfrak k] = 0$, hence $[y,y'] + [z,y'] = 0$ for any $y' \in \mathfrak k$.
Since $[y,y'] \in \mathfrak l$ and $[z,y'] \in \mathfrak p_u$, it follows that $[y,y']=[z,y'] = 0$.
The latter means that $y \in \mathfrak z_{\mathfrak l}(\mathfrak k)$ and $\langle z \rangle \subset \mathfrak p_u$ is a trivial $K^0$-module.
As $\mathfrak p_u /\mathfrak h_u$ is a spherical $K$-module, it is also spherical as a $K^0$-module and therefore cannot contain a trivial one-dimensional $K^0$-submodule, hence $z \in \mathfrak h_u$.
It follows that $y = x - z \in \mathfrak n$ and hence $[y, \mathfrak h_u] \subset \mathfrak h_u$, which implies $y \in \mathfrak a$.
Consequently, $x \in \mathfrak a + \mathfrak h_u$ as required.
\end{proof}

\begin{corollary} \label{crl_second_normalizer}
The following assertions hold:
\begin{enumerate}[label=\textup{(\alph*)},ref=\textup{\alph*}]
\item \label{crl_second_normalizer_a}
$N_G(N^0)^0 = N^0$;

\item \label{crl_second_normalizer_b}
$N_G(N_G(N)) = N_G(N)$.
\end{enumerate}
\end{corollary}

\begin{proof}
Part~(\ref{crl_second_normalizer_b}) follows from~(\ref{crl_second_normalizer_a}) thanks to Theorem~\ref{thm_BriP}(\ref{thm_BriP_b}).
Part~(\ref{crl_second_normalizer_a}) is readily implied by Proposition~\ref{prop_normalizer}, but for convenience we provide a direct argument\footnote{This argument was communicated to the author by one of the referees.} not involving any structure results.
Again by Theorem~\ref{thm_BriP}(\ref{thm_BriP_b}), it suffices to prove that $N_G(N)^0 = N^0$.
Let $N^\sharp$ be the common kernel of all characters of~$N$.
Then Theorem~\ref{thm_BriP}(\ref{thm_BriP_a}) yields $N^\sharp \subset H$.
Since $N^\sharp$ is preserved by all automorphisms of~$N$, it is normalized by $N_G(N)$, therefore $N_G(N)$ acts on $N/N^\sharp$ by automorphisms.
As $N/N^\sharp$ is diagonalizable, the connected group $N_G(N)^0$ acts trivially on it and hence normalizes~$H$.
Thus $N_G(N)^0 = N^0$.
\end{proof}

\subsection{Wonderful varieties and Demazure embeddings}
\label{subsec_wv}

\begin{definition}
A $G$-variety $X$ is said to be \textit{wonderful} of rank~$r$ if the following conditions are satisfied:
\begin{enumerate}[label=\textup{(W\arabic*)},ref=\textup{W\arabic*}]
\item
$X$ is smooth and complete;

\item
$X$ contains an open $G$-orbit whose complement is a divisor with normal crossings having $r$ irreducible components $D_1, \ldots, D_r$ (called the \textit{boundary divisors} of~$X$);

\item
for every subset $I \subset \lbrace 1,\ldots,r \rbrace$, the set $\bigcap \limits_{i \in I} D_i \setminus \bigcup \limits_{j \notin I} D_j$ is a single $G$-orbit in~$X$.
\end{enumerate}
\end{definition}

It is known from~\cite{Lu96} that every wonderful $G$-variety is spherical.

Let $X$ be a wonderful $G$-variety of rank $r$ and let $D_1,\ldots, D_r$ be its boundary divisors.
The definition implies that for every subset $I \subset \lbrace 1,\ldots, r\rbrace$ the $G$-variety $X_I = \bigcap \limits_{i \in I} D_i$ is a wonderful $G$-variety of rank $r - |I|$.
In particular, for every $i=1,\ldots, r$ the boundary divisor $D_i$ is a wonderful $G$-variety of rank~$r-1$.
It follows from the general theory \cite[\S\S\,2,\,3]{Kn91} that for every wonderful $G$-variety $X$ of rank~$r$ there is a bijection $\lbrace 1,\ldots, r \rbrace \to \Sigma_X$, $i \mapsto \sigma_i$, such that for every subset $I \subset \lbrace 1,\ldots, r \rbrace$ one has $\Sigma_{X_I} = \Sigma_X \setminus \lbrace \sigma_i \mid i \in I \rbrace$.
In particular, $\Sigma_{D_i} = \Sigma_X \setminus \lbrace \sigma_i \rbrace$ for every $i = 1,\ldots, r$.

An explicit construction of (some) wonderful $G$-varieties is given by Demazure embeddings introduced below.

Suppose that $G$ is semisimple and let $H \subset G$ be a spherical subgroup such that $N_G(H) = H$.
Regard the Lie algebra $\mathfrak h$ as a point of the Grassmannian $\Gr_{\dim \mathfrak h}(\mathfrak g)$ and let $X$ be the closure in $\Gr_{\dim \mathfrak h}(\mathfrak g)$ of the $G$-orbit of~$\mathfrak h$.
Then $X$ is called the \textit{Demazure embedding} of the spherical homogeneous space $G/H$.

The following result was proved in~\cite[Theorem~1.1]{Lo2} with earlier contributions~\cite[Theorem~1.4]{Bri90} and~\cite[Theorem]{Lu02}.

\begin{theorem}
The $G$-variety $X$ is wonderful.
\end{theorem}

\subsection{The general strategy for computing the set of spherical roots}
\label{subsec_gen_strat}

Let $H \subset G$ be a spherical subgroup specified by a regular embedding in a parabolic subgroup~$P \subset G$.

First of all, we compute the weight lattice $\Lambda_G(G/H)$ by using Theorem~\ref{thm_criterion_spherical}.

Next, we perform several reductions.
\begin{enumerate}[label=\textup{\arabic*.},ref=\textup{\arabic*}]
\item
It follows from the definitions that the set $\Sigma_G(G/H)$ is uniquely determined by the cone $\QQ^+\Sigma_G(G/H)$ and the lattice $\Lambda_G(G/H)$.
Consequently, for determining the set $\Sigma_G(G/H)$ it suffices to compute the cone $\QQ^+\Sigma_G(G/H)$.

\item
As $H \subset Z(G)H \subset N_G(H)$, one has $\QQ^+\Sigma_G(G/H) = \QQ^+\Sigma_G(G/(Z(G)H))$ by Proposition~\ref{prop_norm_sph_roots}.
Taking into account the isomorphism
\[
G/(Z(G)H) \simeq (G/Z(G))/((Z(G)H)/Z(G)),
\]
we may replace $G$ with $G/Z(G)$ and $H$ with $(Z(G)H)/Z(G)$ and assume that $G$ is semisimple.

\item
Using Proposition~\ref{prop_normalizer}, we compute explicitly the subgroup $N_0 = N_G(H)^0$.
Then $\QQ^+\Sigma_G(G/H) = \QQ^+\Sigma_G(G/N_0)$ by Proposition~\ref{prop_norm_sph_roots} (in particular, $|\Sigma_G(G/H)| = |\Sigma_G(G/N_0)|$) and $N_G(N_0)^0 = N_0$ by Corollary~\ref{crl_second_normalizer}(\ref{crl_second_normalizer_a}).
Proposition~\ref{prop_finite_index_in_norm} yields $|\Sigma_G(G/N_0)| = \rk_G(G/N_0)$, hence we get the number of spherical roots for~$H$ at this step.
Replacing $H$ with $N^0$ we may assume that $H$ is connected and $N_G(H)^0 = H$.
\end{enumerate}

In what follows we assume that $G$ is semisimple, $H$ is connected, $N_G(H)^0 = H$, and we need to compute the cone $\QQ^+\Sigma_G(G/H)$.
One more application of Proposition~\ref{prop_norm_sph_roots} yields $\QQ^+\Sigma_G(G/H) = \QQ^+\Sigma_G(G/N_G(H))$, hence it remains to compute the latter cone.

Let $X$ be the closure of the $G$-orbit $G\mathfrak h$ in $\Gr_{\dim \mathfrak h}(\mathfrak g)$.
In view of Theorem~\ref{thm_BriP}(\ref{thm_BriP_b}), $N_G(N_G(H)) = N_G(N_G(H)^0) = N_G(H)$, hence $X$ is the Demazure embedding of $G/N_G(H)$.

Now suppose we have found two subalgebras $\mathfrak h_1, \mathfrak h_2 \subset \mathfrak g$ of dimension $\dim \mathfrak h$ such that $\mathfrak h_1, \mathfrak h_2 \in X$ and the two $G$-orbits $G \mathfrak h_1, G\mathfrak h_2$ are different and both have codimension~$1$ in~$X$.
Put $N_i = N_G(\mathfrak h_i)$ for $i = 1,2$.
Then by the discussion in \S\,\ref{subsec_wv} there are two distinct elements $\sigma_1, \sigma_2 \in \Sigma_G(G/N_G(H))$ such that $\Sigma_G(G/N_1) = \Sigma_G(G/N_G(H)) \setminus \lbrace \sigma_1 \rbrace$ and $\Sigma_G(G/N_2) = \Sigma_G(G/N_G(H)) \setminus \lbrace \sigma_2 \rbrace$.
The latter immediately implies
\[
\Sigma_G(G/N_G(H)) = \Sigma_G(G/N_1) \cup \Sigma_G(G/N_2).
\]
Since $\QQ^+\Sigma_G(G/N_i) = \QQ^+\Sigma_G(G/N_i^0)$ for $i = 1,2$ by Proposition~\ref{prop_norm_sph_roots}, it follows that
\[
\QQ^+\Sigma_G(G/N_G(H)) = \QQ^+(\Sigma_G(G/N_1^0) \cup \Sigma_G(G/N_2^0)).
\]
Consequently, we have reduced the problem of computing the set of spherical roots for $H$ to the same problem for two other subgroups $N_1^0, N_2^0$ such that the number of elements in both subsets $\Sigma_G(G/N_1^0)$ and $\Sigma_G(G/N_2^0)$ is strictly less than that in $\Sigma_G(G/H)$.

By \cite[Proposition~1.3(i)]{Bri90}, the subalgebras $\mathfrak h_1, \mathfrak h_2$ are automatically spherical in~$\mathfrak g$, therefore the two subgroups $N_1^0$ and $N_2^0$ can be explicitly computed using Proposition~\ref{prop_normalizer}.

Iterating the above-described procedure yields an algorithm that in a finite number of steps leads to a finite number of spherical subgroups $\widetilde N_1, \ldots, \widetilde N_s \subset G$ such that for all $i = 1,\ldots,s$ the following properties hold:
\begin{itemize}
\item
$N_G(\widetilde N_i)^0 = \widetilde N_i$;

\item
either $|\Sigma_G(G/\widetilde N_i)| \ge 2$ and the set $\Sigma_G(G/\widetilde N_i)$ is already known (for example, from previous works) or $|\Sigma_G(G/\widetilde N_i)| = 1$.
\end{itemize}
According to Proposition~\ref{prop_finite_index_in_norm}, in the case $|\Sigma_G(G/\widetilde N_i)| = 1$ one has $\rk_G(G/\widetilde N_i)=1$ and the unique element in $\Sigma_G(G/\widetilde N_i)$ is the unique primitive element of $\Lambda_G(G/\widetilde N_i)$ expressed as a nonnegative linear combination of simple roots.

\section{Active \texorpdfstring{$C$}{C}-roots and their properties}
\label{sect_active_C-roots}

In this subsection, we obtain generalizations of the results in~\cite[\S\,2.2]{Avd_solv} on the structure of strongly solvable spherical subgroups.

Let $P \supset B^-$ be a parabolic subgroup of~$G$ with standard Levi subgroup~$L$ and let $P^+ \supset B$ be the parabolic subgroup of~$G$ opposite to~$P$.
Let $C$ denote the connected center of~$L$ and retain all the notation and terminology of~\S\,\ref{subsec_Levi_roots}.
For every $\lambda \in \Phi$, the projection of $\mathfrak g$ to $\mathfrak g(\lambda)$ will be always considered with respect to decomposition~(\ref{eqn_decomposition}).
We shall also use the following additional notation:
\begin{itemize}
\item
for every $\lambda \in \Phi$ the symbol $\widehat \lambda$ stands for the highest weight of the $L$-module $\mathfrak g(\lambda)$;

\item
for every $\delta \in \Delta$ the symbol $\overline \delta$ denotes the image of~$\delta$ under the restriction map $\mathfrak X(T) \to \mathfrak X(C)$.
\end{itemize}

Suppose that $H \subset G$ is a spherical subgroup regularly embedded in~$P$, that is, $H_u \subset P_u$.
Replacing $H$ with a conjugate subgroup if necessary, we may assume that $K = L \cap H$ is a Levi subgroup of~$H$.

From now on until the end of this paper, we shall additionally assume that $L' \subset K \subset L$.
Then $\mathfrak p_u / \mathfrak h_u$ is a spherical $K$-module by Proposition~\ref{prop_S=K}.

It will be convenient for us to work with the subspace~$\mathfrak h^\perp \subset \mathfrak g$.
We have
\begin{equation} \label{eqn_h^perp}
\mathfrak h^\perp = \mathfrak p_u \oplus (\mathfrak h^\perp \cap \mathfrak c) \oplus (\mathfrak h^\perp \cap \mathfrak p^+_u).
\end{equation}
Put $\mathfrak u = \mathfrak h^\perp \cap \mathfrak p_u^+$ for short and note that $\mathfrak u$ is a $K$-module in a natural way.
Moreover, there is a natural $K$-module isomorphism $\mathfrak u \simeq (\mathfrak p_u / \mathfrak h_u)^*$.
In particular, we obtain a $K$-module isomorphism $\CC[\mathfrak p_u/\mathfrak h_u] \simeq \mathrm{S}(\mathfrak u)$.
Recall from \S\,\ref{subsec_SM_gen} that $\mathfrak p_u/\mathfrak h_u$ is a spherical $K$-module if and only if the $K$-module $\mathrm{S}(\mathfrak u)$ is multiplicity free.
The latter property of $\mathrm{S}(\mathfrak u)$ will be extensively used throughout this section.

\begin{definition}
An element $\mu \in \Phi^+$ is said to be an \textit{active $C$-root} if $\mathfrak g(-\mu) \not \subset \mathfrak h$.
\end{definition}

Let $\Psi = \Psi(H)$ denote the set of active $C$-roots.

\begin{definition}
Two active $C$-roots $\lambda, \mu$ are said to be \textit{equivalent} (notation: $\lambda \sim \mu$) if there is a $K$-module isomorphism $\mathfrak g(-\lambda) \simeq \mathfrak g(-\mu)$ (or equivalently $\mathfrak g(\lambda) \simeq \mathfrak g(\mu)$).
\end{definition}

Clearly, this definition determines an equivalence relation on the set~$\Psi$.
Let $\widetilde{\Psi}$ denote the set of all equivalence classes for this relation.
For every $\lambda \in \Psi$, let $\Omega_\lambda \in \widetilde \Psi$ be the equivalence class of~$\lambda$.

As $\mathrm{S}(\mathfrak u)$ is a multiplicity-free $K$-module, so is $\mathfrak u$ itself.
Consequently, for every $\Omega \in \widetilde \Psi$ there exists a unique $K$-submodule $\mathfrak u(\Omega) \subset \mathfrak u$ isomorphic to $\mathfrak g(\lambda)$ for all $\lambda \in \Omega$.
It follows from the definitions that
\begin{equation} \label{eqn_u}
\mathfrak u = \bigoplus\limits_{\Omega \in \widetilde \Psi} \mathfrak u(\Omega)
\end{equation}
and the subspaces $\mathfrak u(\Omega) \subset \mathfrak p^+_u$ have the following properties:

\begin{enumerate}[label=\textup{(\arabic*)},ref=\textup{\arabic*}]
\item
for every $\lambda \in \Omega$ the projection $\mathfrak u (\Omega) \to \mathfrak g(\lambda)$ is a $K$-module isomorphism;

\item
for every $\lambda \in \Phi^+ \setminus \Omega$ the projection $\mathfrak u(\Omega) \to \mathfrak g(\lambda)$ is zero.
\end{enumerate}
For future reference, we mention the following decomposition obtained by combining~(\ref{eqn_h^perp}) and~(\ref{eqn_u}):
\begin{equation} \label{eqn_h^perp_refined}
\mathfrak h^\perp = \mathfrak p_u \oplus (\mathfrak h^\perp \cap \mathfrak c) \oplus \bigoplus\limits_{\Omega \in \widetilde \Psi} \mathfrak u(\Omega).
\end{equation}

For every $\Omega \in \widetilde \Psi$ and $\mu \in \Omega$ there are $K$-module isomorphisms
\begin{equation} \label{eqn_quotient}
\bigoplus\limits_{\lambda \in \Omega} \mathfrak g(-\lambda) / [(\bigoplus\limits_{\lambda \in \Omega} \mathfrak g(-\lambda)) \cap \mathfrak h] \simeq \mathfrak g(-\mu) \simeq \mathfrak u(\Omega)^*.
\end{equation}
Consequently, for every two distinct elements $\lambda,\mu \in \Psi$ with $\lambda \sim \mu$ and for
\[
W = (\mathfrak g(-\lambda) \oplus \mathfrak g(-\mu)) \cap \mathfrak h
\]
there are $K$-module isomorphisms $W \simeq \mathfrak g(-\lambda) \simeq \mathfrak g(-\mu)$ and $W$ projects nontrivially (and hence isomorphically) to both $\mathfrak g(-\lambda)$ and~$\mathfrak g(-\mu)$.
In particular, $W$ is a simple $K$-module and each highest-weight vector of it is the sum of two suitable highest-weight vectors of $\mathfrak g(-\lambda)$ and $\mathfrak g(-\mu)$.

\begin{lemma} \label{lemma_sum}
Suppose that $\lambda \in \Psi$ and $\lambda = \mu + \nu$ for some $\mu, \nu \in \Phi^+$.
Then either $\mu \in \Psi$ or $\nu \in \Psi$.
\end{lemma}

\begin{proof}
This follows readily from $[\mathfrak g(-\mu), \mathfrak g(-\nu)] = \mathfrak g(-\lambda)$.
\end{proof}

\begin{lemma} \label{lemma_mu+nu_Psi}
Suppose that $\lambda,\mu,\nu \in \Psi$ and $\lambda = \mu + \nu$. Then
\begin{enumerate}[label=\textup{(\alph*)},ref=\textup{\alph*}]
\item \label{lemma_mu+nu_Psi_a}
$\mu \sim \nu$;

\item \label{lemma_mu+nu_Psi_b}
as a $K$-module, $\mathfrak g(\lambda)$ is isomorphic to a submodule of $\wedge^2 \mathfrak g(\mu)$;

\item \label{lemma_mu+nu_Psi_c}
$\lambda \not \sim \mu$.
\end{enumerate}
\end{lemma}

\begin{proof}
(\ref{lemma_mu+nu_Psi_a})
Assume that $\mu \not \sim \nu$.
Then it follows from Proposition~\ref{prop_bracket}(\ref{prop_bracket_a}) that, as a $K$-module, $\mathfrak g(\lambda)$ is isomorphic to a submodule of $\mathfrak g(\mu) \otimes \mathfrak g(\nu)$, hence of $\mathrm S^2(\mathfrak g(\mu) \oplus \mathfrak g(\nu))$, hence of $\mathrm S^2(\mathfrak u)$.
Besides, the $K$-module $\mathfrak u$ also contains a copy of~$\mathfrak g(\lambda)$.
Consequently, $\mathrm S(\mathfrak u)$ is not multiplicity free, a contradiction.

(\ref{lemma_mu+nu_Psi_b})
Applying Proposition~\ref{prop_bracket}(\ref{prop_bracket_a}) and  part~(\ref{lemma_mu+nu_Psi_a}) we find that $\mathfrak g(\lambda)$ is isomorphic to a submodule of $\mathfrak g(\mu) \otimes\nobreak \mathfrak g(\mu) \simeq \mathrm S^2 \mathfrak g(\mu) \oplus \wedge^2 \mathfrak g(\mu)$.
If $\mathfrak g(\lambda)$ were isomorphic to a submodule of $\mathrm S^2\mathfrak g(\mu)$ then $\mathrm S(\mathfrak u)$ would contain a copy of $\mathfrak g(\lambda)$ in~$\mathfrak u$ and another copy in $\mathrm S^2 (\mathfrak u)$, which is impossible as $\mathrm S(\mathfrak u)$ is multiplicity free.
Thus $\mathfrak g(\lambda)$ is isomorphic to a submodule of $\wedge^2 \mathfrak g(\mu)$.

(\ref{lemma_mu+nu_Psi_c})
Assume that $\lambda \sim \mu$.
By part~(\ref{lemma_mu+nu_Psi_a}) we also have $\mu \sim \nu$, hence there are $K$-module isomorphisms $\mathfrak g(\lambda) \simeq \mathfrak g(\mu) \simeq \mathfrak g(\nu)$.
Then part~(\ref{lemma_mu+nu_Psi_b}) implies that the $K$-module $\mathfrak g(\lambda)$ is isomorphic to a submodule of $\wedge^2 \mathfrak g(\lambda)$, which is impossible by Proposition~\ref{prop_simple_SM_wedge1}.
\end{proof}

\begin{proposition} \label{prop_crucial}
Suppose that $\lambda, \mu \in \Psi$ and $\lambda = \mu + \nu$ for some $\nu \in \Phi^+ \setminus \Psi$.
Then
\begin{enumerate}[label=\textup{(\alph*)},ref=\textup{\alph*}]
\item \label{prop_crucial_a}
$\Omega_\mu + \nu \subset \Omega_\lambda$;

\item \label{prop_crucial_b}
the subspace $\mathfrak u(\Omega_\mu) \subset \mathfrak p^+_u$ is uniquely determined by $\mathfrak u(\Omega_\lambda)$.
\end{enumerate}
\end{proposition}

\begin{proof}
As $\mathfrak g(-\nu) \subset \mathfrak h$ and $\mathfrak h^\perp$ is ($\ad \mathfrak h$)-stable, it follows that $[\mathfrak u(\Omega_\lambda),\mathfrak g(-\nu)] \subset \mathfrak h^\perp$.
Let $W$ denote the projection of $[\mathfrak u(\Omega_\lambda), \mathfrak g(-\nu)]$ to $\mathfrak p^+$ along $\mathfrak p_u$.
As $\mathfrak p_u \subset \mathfrak h^\perp$ and $\nu \notin \Omega_\lambda$, one has $W \subset \mathfrak h^\perp \cap \mathfrak p^+_u = \mathfrak u$; note that $W$ is $K$-stable.
Clearly, $W$ projects nontrivially onto $\mathfrak g(\mu)$, hence there is a $K$-submodule $W(\Omega_\mu) \subset W$ isomorphic to~$\mathfrak g(\mu)$.
As $\mathfrak u$ is multiplicity free, we conclude that $W(\Omega_\mu)$ is uniquely determined and necessarily coincides with $\mathfrak u(\Omega_\mu)$.
This implies both claims.
\end{proof}

\begin{proposition} \label{prop_nonroot}
Suppose that $\lambda, \mu \in \Psi$, $\lambda \ne \mu$, and $\lambda \sim \mu$. Then $\lambda - \mu \notin \Phi$.
\end{proposition}

\begin{proof}
Assume that $\nu = \lambda - \mu \in \Phi$.
Interchanging $\lambda$ and $\mu$ if needed we may assume $\nu \in \Phi^+$.
Then $\nu \notin \Psi$ by Lemma~\ref{lemma_mu+nu_Psi}(\ref{lemma_mu+nu_Psi_c}).
But in this case Proposition~\ref{prop_crucial}(\ref{prop_crucial_a}) yields $\Omega_\mu + \nu \subset \Omega_\mu$, which is impossible.
\end{proof}

\begin{corollary}
Suppose that $\lambda, \mu \in \Psi$ and $\lambda \sim \mu$.
Then the angle between $\lambda$ and $\mu$ is non-acute.
\end{corollary}

\begin{proof}
If $(\lambda,\mu) > 0$ then Proposition~\ref{prop_pairs_of_C-roots}(\ref{prop_pairs_of_C-roots_b}) implies $\lambda - \mu \in \Phi$, which contradicts Proposition~\ref{prop_nonroot}.
Thus $(\lambda, \mu) \le 0$ as required.
\end{proof}

\begin{proposition} \label{prop_tough}
Suppose that $\lambda, \mu, \nu \in \Psi$, $\lambda = \mu + \nu$, and $\Omega_\mu \ne \lbrace \mu \rbrace$.
Then the following assertions hold.
\begin{enumerate}[label=\textup{(\alph*)},ref=\textup{\alph*}]
\item \label{prop_tough_a}
$\Omega_\mu = \lbrace \mu_1,\mu_2 \rbrace$ for two distinct elements $\mu_1,\mu_2$ \textup($\mu$ and $\nu$ are among these\textup) such that $\Omega_\lambda \supset \lbrace 2\mu_1, \mu_1 + \mu_2 \rbrace \ni \lambda$ and $2\mu_2 \notin \Phi$.
In particular, $\Omega_\mu + \mu_1 \subset \Omega_\lambda$.

\item \label{prop_tough_b}
The pair $(L', \mathfrak u(\Omega_\mu))$ is equivalent to $(\SL_2, \CC^2)$ and $\mathfrak u(\Omega_\lambda)$ is a trivial one-di\-men\-sion\-al $L'$-module.

\item \label{prop_tough_c}
The subspace $\mathfrak u(\Omega_\mu) \subset \mathfrak p^+_u$ is uniquely determined by~$\mathfrak u(\Omega_\lambda)$.
\end{enumerate}
\end{proposition}

\begin{proof}
Lemma~\ref{lemma_mu+nu_Psi}(\ref{lemma_mu+nu_Psi_a}) yields $\mu \sim \nu$, thus $\nu \in \Omega_\mu$.
Put $k = |\Omega_\mu|$; the hypothesis implies $k \ge 2$.
Let $\mu_1, \mu_2,\ldots, \mu_k$ be all the distinct elements of~$\Omega_\mu$ ($\mu$ and $\nu$ are among these).
Let $L_0$ denote the product of simple factors of $L'$ that act nontrivially on $\mathfrak u(\Omega_\mu)$.
Let also $\alpha_1,\ldots,\alpha_l \in \Pi$ be the simple roots of~$L_0$.
For each $i = 1,\ldots,k$, let $\beta_i \in \Delta^+$ be the lowest weight of $\mathfrak g(\mu_i)$ regarded as an $L$-module.
Then $(\beta_i, \alpha_j) \le 0$ for all $i = 1,\ldots,k$ and $j = 1,\ldots, l$.
If $\beta_i - \beta_j \in \Delta$ for some $i \ne j$ then $\mu_i - \mu_j = \overline \beta_i - \overline\beta_j \in \Phi$, which is impossible by Proposition~\ref{prop_nonroot}.
Thus $\beta_i - \beta_j \notin \Delta$ and so $(\beta_i,\beta_j) \le 0$ for all $i \ne j$.
Consequently, the angles between the roots in the set $E = \lbrace \beta_1,\ldots, \beta_k,\alpha_1,\ldots, \alpha_l \rbrace$ are pairwise non-acute.
Since $E \subset \Delta^+$, it follows that the roots in $E$ are linearly independent and hence they form a system of simple roots of some root system.
Let $\mathrm{DD}(E)$ denote the Dynkin diagram of the set~$E$.

By Lemma~\ref{lemma_mu+nu_Psi}(\ref{lemma_mu+nu_Psi_b}) the $K$-module $\mathfrak g(\lambda)$ is isomorphic to a submodule of $\wedge^2 \mathfrak g(\mu)$.
As $\lambda \not\sim \mu$ by Lemma~\ref{lemma_mu+nu_Psi}(\ref{lemma_mu+nu_Psi_c}), the $K$-module $\mathfrak g(\mu) \oplus \mathfrak g(\lambda)$ is isomorphic to a submodule of~$\mathfrak u$.
Then Proposition~\ref{prop_simple_SM_wedge2} leaves only the two cases considered below.

\textit{Case}~1: $L_0 \simeq \Sp_{2n}$ ($n \ge 2$) and there are $L_0$-module isomorphisms $\mathfrak g(\mu) \simeq \CC^{2n}$ and $\mathfrak g(\lambda) \simeq \CC^1$.
Then $l = n$ and we let $\alpha_1,\ldots,\alpha_n$ be the standard numbering of the simple roots of~$\Sp_{2n}$.
Clearly, $(\beta_i, \alpha_1) < 0$ and $(\beta_i, \alpha_j) = 0$ for all $i=1,\ldots,k$ and $j = 2,\ldots,n$, hence the diagram $\mathrm{DD}(E)$ has the following two properties:
\begin{itemize}
\item
the nodes corresponding to $\alpha_{n-1},\alpha_n$ are joined by a double edge;

\item
the node corresponding to $\alpha_1$ is joined by an edge with $k+1$ nodes corresponding to $\beta_1,\ldots, \beta_k,\alpha_2$.
\end{itemize}
As $k+1\ge 3$, these two conditions cannot hold simultaneously.

\textit{Case}~2: $L_0 \simeq \SL_n$ ($n \ge 2$) and there are $L_0$-module isomorphisms $\mathfrak g(\mu) \simeq \CC^n$ and $\mathfrak g(\lambda) \simeq \wedge^2 \CC^n$.
Then $l = n-1$ and we let $\alpha_1,\ldots, \alpha_{n-1}$ be the standard numbering of the simple roots of~$\SL_n$.
We may also assume that $(\beta_i, \alpha_1) < 0$ and $(\beta_i, \alpha_j) = 0$ for all $i = 1,\ldots, k$ and $j = 2,\ldots, n-1$.
Note that in the diagram $\mathrm{DD}(E)$ the node corresponding to $\alpha_1$ is joined by an edge with the nodes corresponding to $\beta_1,\ldots, \beta_k$ and also~$\alpha_2$ (when $n \ge 3$).
Further we consider two subcases.

\textit{Subcase}~2.1: $2\mu_i \notin \Psi$ for all $i = 1,\ldots, k$.
Then one automatically has $\mu \ne \nu$ and $\mathfrak g(-2\mu_i) \subset \mathfrak h$ for all $i = 1,\ldots, k$.

Taking into account Proposition~\ref{prop_properties_of_g(lambda)}(\ref{prop_properties_of_g(lambda)_c}) we find that there are $L_0$-module isomorphisms $\mathfrak g(-\mu) \simeq \mathfrak g(-\nu) \simeq (\CC^n)^*$ and $\mathfrak g(-\lambda) \simeq \wedge^2 \mathfrak g(-\nu)$.
Choose a highest-weight vector $x_1 \in \mathfrak g(-\mu)$ and put $y_1 = [e_{-\alpha_1}, x_1]$, so that $y_1$ is a nonzero vector in $\mathfrak g(-\mu)$ not proportional to~$x_1$.
Fix an $L_0$-module isomorphism $\varphi \colon \mathfrak g(-\mu) \xrightarrow{\sim} \mathfrak g(-\nu)$ and put $x_2 = \varphi(x_1)$, $y_2 = \varphi(y_1)$.

Then the natural map $\mathfrak g(-\mu) \times \mathfrak g(-\nu) \to \mathfrak g(-\lambda)$, $(x,y) \mapsto [x,y]$, induces a chain of $L_0$-module homomorphisms
\[
\mathfrak g(-\mu) \otimes \mathfrak g(-\nu) \xrightarrow{\varphi \otimes \operatorname{id}} \mathfrak g(-\nu) \otimes \mathfrak g(-\nu) \to \wedge^2 \mathfrak g(-\nu) \xrightarrow{\psi} \mathfrak g(-\lambda)
\]
where the middle arrow is the natural projection and $\psi$ is an isomorphism.
It is easy to see that the image of the element $x_1\otimes y_2 - y_1 \otimes x_2 \in \mathfrak g(-\mu) \otimes \mathfrak g(-\nu)$ under the first two maps equals $x_2 \otimes y_2 - y_2\otimes x_2 \ne 0$, hence $[x_1,y_2]-[y_1,x_2] \ne 0$.

As $\mu, \nu \in \Psi$ and $\mu \sim \nu$, the $K$-module $W = (\mathfrak g(-\mu) \oplus \mathfrak g(-\nu))\cap \mathfrak h$ projects isomorphically to both $\mathfrak g(-\mu)$ and $\mathfrak g(-\nu)$, hence $W$ has a highest-weight vector of the form $x_0 = x_1 + ax_2$ for some $a \in \CC^\times$.
Then $y_0 = [e_{-\alpha_1},x_0] = y_1 + ay_2 \in W$.
Since $[W,W], \mathfrak g(-2\mu), \mathfrak g(-2\nu) \subset \mathfrak h$, it follows that $0 \ne a([x_1,y_2]-[y_1,x_2]) = [x_0,y_0] - [x_1,y_1] - a^2[x_2, y_2] \in \mathfrak g(-\lambda) \cap \mathfrak h$.
As $\mathfrak g(-\lambda)$ is simple as a $K$-module, we have $\mathfrak g(-\lambda) \subset \mathfrak h$ and $\lambda \notin \Psi$, a contradiction.

\textit{Subcase}~2.2: there exists $i \in \lbrace 1,\ldots, k \rbrace$ such that $2\mu_i \in \Psi$.
Renumbering the elements $\mu_1,\ldots, \mu_k$ we may assume $2\mu_1 \in \Psi$.

Consider the $L$-equivariant surjective map $\pi \colon \wedge^2 \mathfrak g(\mu_1) \to \mathfrak g(2\mu_1)$.
As $L_0$-modules, we have $\mathfrak g(\mu_1) \simeq \CC^n$ and therefore $\wedge^2 \mathfrak g(\mu_1) \simeq \wedge^2 \CC^n$, hence $\wedge^2 \mathfrak g(\mu_1)$ is simple and $\pi$ is an $L$-module isomorphism.
Clearly, the set of $T$-weights of $\mathfrak g(\mu_1)$ is
\begin{equation} \label{eqn_T-weights}
\lbrace \beta_1, \beta_1 + \alpha_1, \beta_1 + \alpha_1 + \alpha_{2}, \ldots, \beta_1 + \alpha_1 + \ldots + \alpha_{n-1} \rbrace.
\end{equation}
It follows that the set of $T$-weights of $\mathfrak g(2\mu_1)$ consists of all sums $\omega_1 + \omega_2$ with $\omega_1,\omega_2$ being distinct elements of the set~(\ref{eqn_T-weights}).
In particular, one of these weights is $2\beta_1+\alpha_1$, which implies $2\beta_1 + \alpha_1 \in \Delta^+$.
As $|E| \ge 3$, the roots $\beta_1$ and $\alpha_1$ cannot generate a root subsystem of $\Delta$ of type~$\mathsf G_2$, hence the only possibility is that they generate a root subsystem of type~$\mathsf B_2$ with $\beta_1$ being short and $\alpha_1$ being long.
Consequently, in the diagram $\mathrm{DD}(E)$ the nodes corresponding to $\beta_1$ and $\alpha_1$ are joined by a double edge with the arrow directed to~$\beta_1$.
In view of the classification of connected Dynkin diagrams, the latter immediately implies $k = n = 2$.

If $\mathfrak g(2\mu_2) \ne 0$ then repeating the above argument for $\mathfrak g(\mu_2)$ and $\mathfrak g(2\mu_2)$ would yield another double edge in the diagram $\mathrm{DD}(E)$, which is impossible.
Thus $\mathfrak g(2\mu_2) = 0$ and $2\mu_2 \notin \Phi$.

For $i=1,2$ fix a highest-weight vector $x_i \in \mathfrak g(-\mu_i)$ and put $y_i = [e_{-\alpha_1}, x_i]$.
Then $\mathfrak g(-2\mu_1)$ is spanned by $[x_1,y_1]$ and $2\mu_2 \notin \Phi$ implies $[x_2,y_2]=0$.
Consider the $K$-module $W = (\mathfrak g(-\mu_1) \oplus \mathfrak g(-\mu_2))\cap \mathfrak h$.
Then $W$ projects isomorphically to both $\mathfrak g(-\mu_1)$ and $\mathfrak g(-\mu_2)$, hence $W$ has a highest-weight vector of the form $x_0 = x_1 + ax_2$ for some $a \in \CC^\times$, and $y_0 = y_1 + ay_2 \in W$.

Now suppose $\mu_1 + \mu_2 \notin \Psi$.
Then $[x_1,y_2],[y_1,x_2] \in \mathfrak g(-\mu_1-\mu_2) \subset \mathfrak h$.
As $[W,W] \subset \mathfrak h$, we get $[x_1,y_1] = [x_0,y_0] - a[x_1,y_2] - a[x_2,y_1] \in \mathfrak h$ and hence $\mathfrak g(-2\mu_1) \subset \mathfrak h$, a contradiction.
Thus $\mu_1 + \mu_2 \in \Psi$.
Applying Lemma~\ref{lemma_mu+nu_Psi}(\ref{lemma_mu+nu_Psi_b}) we find that both $\mathfrak g(2\mu_1), \mathfrak g(\mu_1 + \mu_2)$ are isomorphic to $\wedge^2 \mathfrak g(\mu)$ as $K$-modules, which yields $2\mu_1 \sim \mu_1 + \mu_2$ and thus completes the proof of parts~(\ref{prop_tough_a}) and~(\ref{prop_tough_b}).

Arguing similarly to Subcase~2.1, we find that the element $[x_1,y_2] - [y_1,x_2]$ is nonzero and hence spans $\mathfrak g(-\mu_1 - \mu_2)$.
Next, observe that $[W,W]$ is one-dimensional and spanned by $[x_0,y_0] = [x_1,y_1] + a([x_1,y_2]-[y_1,x_2])$, which implies $[W,W] = (\mathfrak g(-2\mu_1) \oplus \mathfrak g(-\mu_1-\mu_2)) \cap \mathfrak h$.
Thus the latter subspace uniquely determines the value of~$a$, which in turn uniquely determines~$W$.
Since for every $\Omega \in \widetilde \Psi$ the subspaces $\mathfrak u(\Omega)$ and $(\bigoplus \limits_{\rho \in \Omega} \mathfrak g(-\rho)) \cap \mathfrak h$ of~$\mathfrak g$ uniquely determine each other, we conclude that $\mathfrak u(\Omega_\mu)$ is uniquely determined by $\mathfrak u(\Omega_\lambda)$ as required in part~(\ref{prop_tough_c}).
\end{proof}

The following example shows that the situation described in Proposition~\ref{prop_tough} does occur.

\begin{example} \label{crl_ge2_not_active}
Consider the group $G = \SO_7$ preserving the symmetric bilinear form on~$\CC^7$ whose matrix has ones on the antidiagonal and zeros elsewhere.
Then the Lie algebra~$\mathfrak g$ consists of all $(7\times7)$-matrices that are skew-symmetric with respect to the antidiagonal.
We choose $B,B^-,T$ to be the subgroup of all upper triangular, lower triangular, diagonal matrices, respectively, contained in~$G$.
Then $\Pi = \lbrace \alpha_1, \alpha_2, \alpha_3 \rbrace$ where $\alpha_1(t)=t_1t_{2}^{-1}$, $\alpha_2(t)=t_2t_{3}^{-1}$, $\alpha_3(t) = t_3$ for all $t = \diag(t_1,t_2,t_3,1,t_3^{-1},t_2^{-1},t_1^{-1}) \in T$.
We consider a connected subgroup $H \subset G$ regularly embedded in a parabolic subgroup $P \supset B^-$ such that the Lie algebras $\mathfrak h$ and $\mathfrak p$ consist of all matrices in $\mathfrak g$ having the form
\[
\begin{pmatrix}
x+y & 0 & 0 & 0 & 0 & 0 & 0\\
a & x & * & 0 & 0 & 0 & 0\\
b & * & y & 0 & 0 & 0 & 0\\
2c & -b & a & 0 & 0 & 0 & 0\\
* & c & 0 & -a & -y & * & 0\\
* & 0 & -c & b & * & -x & 0\\
0 & * & * & -2c & -b & -a & -x-y\\
\end{pmatrix}
\ \text{and} \
\begin{pmatrix}
* & 0 & 0 & 0 & 0 & 0 & 0\\
* & * & * & 0 & 0 & 0 & 0\\
* & * & * & 0 & 0 & 0 & 0\\
* & * & * & 0 & 0 & 0 & 0\\
* & * & 0 & * & * & * & 0\\
* & 0 & * & * & * & * & 0\\
0 & * & * & * & * & * & *\\
\end{pmatrix},
\]
respectively.
Then
$L$ consists of all block-diagonal matrices with the sequence of blocks $(t, A, 1, (A^\natural)^{-1},t^{-1})$ where $t \in \CC^\times$, $A \in \GL_2$, and $A^\natural$ stands for the transpose of $A$ with respect to the antidiagonal;
$K$ consists of all matrices in $L$ satisfying $t = \det A$;
and $C$ consists of all diagonal matrices of the form $\diag(t_1,t_2,t_2,1,t_2^{-1}, t_2^{-1},t_1^{-1})$.
The pair $(K, \mathfrak p_u/\mathfrak h_u)$ is equivalent to $(\GL_2, \CC^2 \oplus \CC^1)$ with the action of $\GL_2$ given by $(A, (x,y)) \mapsto ((\det A)^{-1}Ax, (\det A)^{-1}y)$.
As the latter module is spherical, we find that $H$ is spherical in $G$ by Proposition~\ref{prop_S=K}.
One has $\Psi = \lbrace \mu_1,\mu_2, \lambda_1, \lambda_2 \rbrace$ with $\widehat \mu_1 = \alpha_2 + \alpha_3$, $\widehat \mu_2 = \alpha_1 + \alpha_2$, $\widehat \lambda_1 = \alpha_2 + 2\alpha_3$, and $\widehat \lambda_2 = \alpha_1+ \alpha_2 + \alpha_3$; the set $\widetilde \Psi$ consists of two classes $\Omega_1 = \lbrace \mu_1, \mu_2 \rbrace$ and $\Omega_2 = \lbrace \lambda_1, \lambda_2 \rbrace$.
Note that $\lambda_1 = 2\mu_1$ and $\lambda_2 = \mu_1+ \mu_2$, so that $\Omega_2 = \Omega_1 + \mu_1$.
Finally, the pair $(L', \mathfrak u(\Omega_1))$ is equivalent to $(\SL_2, \CC^2)$ and $\mathfrak u(\Omega_2)$ is a trivial one-dimensional $L'$-module.

\end{example}

We now introduce the following notation:
\begin{gather*}
\Psi_0 = \lbrace \lambda \in \Psi \mid |\Omega_\lambda| \ge 2 \rbrace, \\
\Psi_0^{\max} = \lbrace \lambda \in \Psi_0 \mid \nexists\, \mu \in \Phi^+ \ \text{with} \ \lambda + \mu \in \Psi \rbrace,
\\
\widetilde \Psi_0 = \lbrace \Omega \in \widetilde \Psi \mid \Omega \subset \Psi_0 \rbrace,
\\
\widetilde \Psi_0^{\max} = \lbrace \Omega \in \widetilde \Psi \mid \Omega \subset \Psi_0^{\max} \rbrace.
\end{gather*}

The following observations are implied by the above definitions along with Propositions~\ref{prop_crucial}(\ref{prop_crucial_a}) and~\ref{prop_tough}(\ref{prop_tough_a}).

\begin{remark} \label{rem_important}
~
\begin{enumerate}[label=\textup{(\alph*)},ref=\textup{\alph*}]
\item \label{rem_important_a}
One has $\Psi_0^{\max} = \bigcup \limits_{\Omega \in \widetilde \Psi_0^{\max}} \Omega$.

\item \label{rem_important_b}
For every $\Omega \in \widetilde \Psi_0 \setminus \widetilde \Psi_0^{\max}$ there exist $\mu_1, \ldots, \mu_m \in \Phi^+$ and $\Omega_1, \ldots, \Omega_m \in \widetilde \Psi_0$ such that $\Omega + \mu_1 + \ldots + \mu_j \subset \Omega_j$ for all $j=1,\ldots, m$ and $\Omega_m \in \widetilde \Psi_0^{\max}$.

\item \label{rem_important_c}
As a consequence of~(\ref{rem_important_b}), for every $\Omega \in \widetilde \Psi_0$ there exist $\theta \in \ZZ^+ \Phi^+$ and $\Omega' \in \widetilde \Psi_0^{\max}$ such that $\Omega + \theta \subset \Omega'$.
\end{enumerate}
\end{remark}

\begin{lemma} \label{lemma_Psi0max_non-acute}
The angles between the $C$-roots in $\Psi_0^{\max}$ are pairwise non-acute, and these $C$-roots are linearly independent.
\end{lemma}

\begin{proof}
Assume that $(\lambda, \mu) > 0$ for two distinct elements $\lambda,\mu \in \Psi_0^{\max}$.
Then $\lambda - \mu \in \Phi$ by Proposition~\ref{prop_pairs_of_C-roots}(\ref{prop_pairs_of_C-roots_b}), which contradicts the definition of~$\Psi_0^{\max}$.
Since all the $C$-roots in $\Psi_0^{\max}$ are contained in an open half-space of $\QQ\mathfrak X(C)$, they are linearly independent.
\end{proof}

\begin{proposition}
Up to conjugation by an element of~$C$, $H$ is uniquely determined by the pair $(K,\Psi)$.
\end{proposition}

\begin{proof}
The equivalence relation on $\Psi$ is recovered from its definition: $\lambda \sim \mu$ if and only if $\mathfrak g(\lambda) \simeq \mathfrak g(\mu)$ as $K$-modules.
In view of decomposition~(\ref{eqn_h^perp_refined}) it now suffices to show that for every $\Omega \in \widetilde \Psi$ the subspace $\mathfrak u(\Omega) \subset \mathfrak p_u^+$ is uniquely determined up to conjugation by an element of~$C$.
If $|\Omega|=1$ and $\Omega = \lbrace \lambda \rbrace$ then $\mathfrak u(\Omega) = \mathfrak g(\lambda)$.
If $|\Omega| \ge 2$ then $\Omega \in \widetilde \Psi_0$.
Combining Remark~\ref{rem_important} with Propositions~\ref{prop_crucial}(\ref{prop_crucial_b}) and~\ref{prop_tough}(\ref{prop_tough_c}) we see that all the $K$-modules $\mathfrak u(\Omega)$ with $\Omega \in \widetilde \Psi_0$ are uniquely determined by the $K$-modules $\mathfrak u (\Omega)$ with $\Omega \in \widetilde \Psi_0^{\max}$.
For every $\lambda \in \Psi_0^{\max}$, fix a highest-weight vector $v_\lambda$ in $\mathfrak g(\lambda)$.
Then for every $\Omega \in \widetilde \Psi_0^{\max}$ a highest-weight vector of $\mathfrak u(\Omega)$ has the form $\sum \limits_{\lambda \in \Omega} a_\lambda v_\lambda$ where $a_\lambda \in \CC^\times$.
Since the set $\Psi_0^{\max}$ is linearly independent (see Lemma~\ref{lemma_Psi0max_non-acute}), conjugating $\mathfrak h$ by an appropriate element of~$C$ we may reach the situation where $a_\lambda = 1$ for all $\lambda \in \Psi_0^{\max}$.
\end{proof}

\section{Implementation of the general strategy}
\label{sect_comp_SR}

Retain all the notation of~\S\,\ref{sect_active_C-roots} and recall that we work with spherical subgroups $H$ satisfying $L' \subset K \subset L$.
In this section, for such subgroups $H$ we implement the general strategy from~\S\,\ref{subsec_gen_strat} for computing the set of spherical roots.

\subsection{Outline}

Thanks to one of the reductions described in~\S\,\ref{subsec_gen_strat}, computing the set of spherical roots of~$H$ easily reduces to the case of semisimple~$G$.
Having this in mind, throughout the whole section we assume that $G$ is semisimple.

For a given spherical subgroup~$H$, to implement one iteration of the general strategy from~\S\,\ref{subsec_gen_strat}, we go through the following steps:
\begin{enumerate}
\item
compute the subgroup $N_G(H)^0$ using Proposition~\ref{prop_normalizer} and replace $H$ with $N_G(H)^0$;

\item
present a collection of one-parameter degenerations of the algebra $\mathfrak h$ inside~$\mathfrak g$ such that for every such degeneration $\widetilde{\mathfrak h} \subset \mathfrak g$ the $G$-orbits $G\mathfrak h, G\widetilde {\mathfrak h} \subset \Gr_{\dim \mathfrak h}(\mathfrak g)$ satisfy $\dim G\mathfrak h = \dim G\widetilde{\mathfrak h} + 1$;

\item \label{step3}
show that for any two different one-parameter degenerations $\widetilde {\mathfrak h}_1, \widetilde {\mathfrak h}_2$ from the above collection the $G$-orbits $G\widetilde {\mathfrak h}_1, G\widetilde {\mathfrak h}_2 \subset \Gr_{\dim \mathfrak h}(\mathfrak g)$ are different.
\end{enumerate}

Depending on the structure of~$H$, we consider two different types of one-parameter degenerations of~$\mathfrak h$, namely, degenerations via the multiplicative group $(\CC^\times, \times)$ and via the additive group $(\CC, +)$; we call them \textit{multiplicative} and \textit{additive} degenerations, respectively.

Multiplicative degenerations (see \S\,\ref{subsec_degen_mult}) apply in the situation $\Psi_0 \ne \varnothing$, they are in bijection with the set $\Psi_0^{\max}$.
The construction of such degenerations involves one-parameter subgroups of~$C$.

Additive degenerations (see \S\,\ref{subsec_degen_add}) apply in the situation $\Psi_0 = \varnothing$, $\Psi \ne \varnothing$, they are in bijection with the set~$\Psi$.
The construction of such degenerations involves one-parameter root unipotent subgroups of~$B^-$.

It is easy to see that, to perform step~(\ref{step3}), for $\Psi_0 \ne \varnothing$ there are always at least two different multiplicative degenerations and for $\Psi_0 = \varnothing$, $\Psi \ne \varnothing$ one can find two different additive degenerations unless $|\Psi|=1$.
We call the cases with $|\Psi|=1$ \textit{primitive} and classify them all in \S\,\ref{subsec_primitive_cases}.
Moreover, it turns out that for all primitive cases the corresponding sets of spherical roots are already known, which completes the algorithm for computing the set of spherical roots for spherical subgroups $H$ with $L' \subset K \subset L$.
In what follows, we shall refer to this algorithm as the \textit{base algorithm}.

\subsection{Description of the group \texorpdfstring{$N_G(H)^0$}{N\_G(H)\^{}0}}
\label{subsec_normalizer}

We fix $P,L,H,K$ as before.

As $Z_L(K)^0 = C$, Proposition~\ref{prop_normalizer} yields $N_G(H)^0 = D^0L' \rightthreetimes H_u$ where $D = C \cap N_L(\mathfrak h_u)$.
In particular, if $\Psi_0 = \varnothing$ then $D = C$ and $N_G(H)^0 = L \rightthreetimes H_u$.
To describe $D$ in the general case, consider the lattice
\[
\Xi = \ZZ\lbrace \mu - \nu \mid \mu, \nu \in \Psi_0^{\max} \ \text{and} \ \mu \sim \nu \rbrace \subset \mathfrak X(C)
\]
and put $\Xi^\sat = \QQ\Xi \cap \mathfrak X(C)$.
The following lemma determines $D$ and $D^0$.

\begin{lemma} \label{lemma_XiAA0}
The following assertions hold.
\begin{enumerate}[label=\textup{(\alph*)},ref=\textup{\alph*}]
\item \label{lemma_XiAA0_a}
$\Xi = \ZZ\lbrace \mu - \nu \mid \mu, \nu \in \Psi \ \text{and} \ \mu \sim \nu \rbrace$.

\item \label{lemma_XiAA0_b}
The group $D$ \textup(resp.~$D^0$\textup) is the common kernel of all characters in~$\Xi$ \textup(resp.~$\Xi^\sat$\textup).
\end{enumerate}
\end{lemma}

\begin{proof}
(\ref{lemma_XiAA0_a})
This follows from Remark~\ref{rem_important}(\ref{rem_important_c}).

(\ref{lemma_XiAA0_b})
The assertion for~$D$ follows from its definition and part~(\ref{lemma_XiAA0_a}).
The assertion for $D^0$ follows from that for~$D$.
\end{proof}

In what follows we assume that $N_G(H)^0 = H$, so that $K = D^0L'$.
We also let $X$ denote the closure of the $G$-orbit $G\mathfrak h$ in $\Gr_{\dim \mathfrak h}(\mathfrak g)$, so that $X$ is the Demazure embedding of $G/N_G(H)$.

\subsection{Reduction of the ambient group}
\label{subsec_reduction_AG}

In this subsection we describe a natural reduction that under certain conditions enables one to pass from the pair $(G,H)$ to another pair $(G_0,H_0)$ with a ``smaller'' group~$G_0$.
This reduction, based essentially on the parabolic induction,
keeps all the combinatorics of active roots unchanged and preserves the set of spherical roots.
In principle, it can be applied before any step of the base algorithm, but in our paper it will be especially useful in \S\S\,\ref{subsec_primitive_cases},\,\ref{subsec_further_opt}.

Consider the set $\Pi_0 = \bigcup \limits_{\lambda \in \Psi} \Supp \widehat \lambda$.

\begin{lemma} \label{lemma_reduction_AG}
Let $F$ be a simple factor of $L'$ and let $\Pi_F \subset \Pi$ be the set of simple roots of~$F$.
\begin{enumerate}[label=\textup{(\alph*)},ref=\textup{\alph*}]
\item \label{lemma_reduction_AG_a}
If $F$ acts nontrivially on $\mathfrak u$ then $\Pi_F \subset \Pi_0$.

\item \label{lemma_reduction_AG_b}
If $F$ acts trivially on $\mathfrak u$ then $(\Pi_F, \Pi_0) = 0$.
\end{enumerate}
\end{lemma}

\begin{proof}
(\ref{lemma_reduction_AG_a})
Let $\lambda \in \Psi$ be such that $F$ acts nontrivially on $\mathfrak g(\lambda)$.
Then $e_\alpha$ acts nontrivially on $\mathfrak g(\lambda)$ for each $\alpha \in \Pi_F$, therefore the difference between the highest and lowest weights of $\mathfrak g(\lambda)$ is a linear combination of simple roots in $\Pi_F$ with positive coefficients, which implies $\Pi_F \subset \Supp \widehat \lambda \subset \Pi_0$.

(\ref{lemma_reduction_AG_b})
Put $\mathrm E = \lbrace \alpha \in \Delta^+ \mid \mathfrak g_{-\alpha} \not\subset \mathfrak h \rbrace$.
As $F$ acts trivially on $\mathfrak u$, one has $(\Pi_F, \alpha) = 0$ for all $\alpha \in \mathrm E$.
Since $\mathfrak h$ is a subalgebra of~$\mathfrak g$, for every $\alpha \in \mathrm E$ and every expression $\alpha = \beta + \gamma$ with $\beta, \gamma \in \Delta^+$ at least one of the summands belongs to~$\mathrm E$ and hence both summands belong to $\ZZ \mathrm E$.
Then $\Pi_0 \subset \ZZ\mathrm E$ by \cite[Lemma~5.11]{Avd_solv_inv} and thus $(\Pi_F, \Pi_0) = 0$.
\end{proof}

Let $L_0 \subset G$ be the standard Levi subgroup with $\Pi_{L_0} = \Pi_0$.
Put $G_0 = L_0'$ and $H_0 = G_0 \cap H$.
Then it is easy to see that $H_0$ is regularly embedded in the parabolic subgroup $P \cap G_0 \subset G_0$ with standard Levi subgroup $L \cap G_0$ and the Levi subgroup $K \cap G_0 \subset H_0$ satisfies $(L \cap G_0)' \subset K \cap G_0 \subset L \cap G_0$.
Thanks to Lemma~\ref{lemma_reduction_AG}, the connected center of $L \cap G_0$ equals $C \cap G_0$ and $(L \cap G_0)'$ coincides with the product of all simple factors of $L'$ contained in~$G_0$.
Combining Lemma~\ref{lemma_reduction_AG} with Proposition~\ref{prop_properties_of_g(lambda)}%
(\ref{prop_properties_of_g(lambda)_a}) we obtain the following property: if $\mathfrak g(\lambda) \cap \mathfrak g_0 \ne \lbrace 0 \rbrace$ for some $\lambda \in \Phi$ then $\mathfrak g(\lambda) \subset \mathfrak g_0$ and $\mathfrak g(\lambda)$ is simple as an $(L \cap G_0)$-module.
It follows that the objects $\Psi$, $\widetilde \Psi$, $\mathfrak u$ are naturally identified with those for~$H_0$ and the pairs $(K, \mathfrak h_0)$, $(K,\mathfrak u)$ are equivalent to $(K \cap G_0, \mathfrak h_0)$, $(K \cap G_0, \mathfrak u)$, respectively.

We say that the pair $(G_0, H_0)$ is obtained from $(G,H)$ by \textit{reduction of the ambient group}.

In the next statement, the set of simple roots of~$G_0$, which is~$\Pi_0$, is regarded as a subset of~$\Pi$.

\begin{proposition} \label{prop_reduction_AG_sph_roots}
One has $\Sigma_G(G/H) = \Sigma_{G_0}(G_0/H_0)$.
\end{proposition}

\begin{proof}
Let $\widetilde L \supset L$ be the standard Levi subgroup of $G$ such that $\Pi_{\widetilde L} = \Pi_L \cup \Pi_0$.
Then the homogeneous space $G/H$ is parabolically induced from $\widetilde L/(H \cap \widetilde L)$, which implies $\Lambda_G(G/H) = \Lambda_{\widetilde L}(\widetilde L / (H \cap \widetilde L))$ and $\Sigma_G(G/H) =  \Sigma_{\widetilde L}(\widetilde L/(H \cap \widetilde L))$ by Proposition~\ref{prop_par_ind}.
Thanks to Lemma~\ref{lemma_reduction_AG}, $G_0$ is a normal subgroup of $\widetilde L$ and all simple factors of $\widetilde L'$ not contained in~$G_0$ are contained in~$H$.
As $H = N_G(H)^0$, Lemma~\ref{lemma_XiAA0} implies $Z(\widetilde L) \subset H$.
Consequently, $Z(\widetilde L)$ and all simple factors of $\widetilde L'$ not contained in~$G_0$ act trivially on $\widetilde L / (H \cap \widetilde L)$, which implies $\widetilde L/(H \cap \widetilde L) \simeq G_0/H_0$.
Thus restricting characters from $T$ to $T \cap G_0$ identifies $\Lambda_{\widetilde L}(\widetilde L / (H \cap \widetilde L))$ with $\Lambda_{G_0}(G_0/H_0)$ and $\Sigma_{\widetilde L}(\widetilde L/(H \cap \widetilde L))$ with $\Sigma_{G_0}(G_0/H_0)$.
\end{proof}

\subsection{Multiplicative degenerations}
\label{subsec_degen_mult}

This type of degenerations applies to the situation $\Psi_0 \ne \varnothing$.
The construction of such a degeneration depends on the choice of an active $C$-root $\lambda \in \Psi_0^{\max}$, which is assumed to be fixed throughout this subsection.

According to Remark~\ref{rem_important}(\ref{rem_important_c}), for every $\Omega \in \widetilde \Psi_0$ we fix a choice of $\theta_\Omega \in \ZZ^+ \Phi^+$ and $\Omega' \in \widetilde \Psi_0^{\max}$ such that $\Omega + \theta_\Omega \subset \Omega'$.
In this situation, we say that $\Omega$ is of type~1 if $\lambda \notin \Omega + \theta_\Omega$ and of type~2 otherwise.
We now put $\widetilde \Psi_1 = \lbrace \Omega \in \widetilde \Psi_0 \mid \Omega \ \text{is of type~1}\rbrace$ and $\widetilde \Psi_2 = \lbrace \Omega \in \widetilde \Psi_0 \mid \Omega \ \text{is of type~2}\rbrace$.
For every $\Omega \in \widetilde \Psi$ we put
\[
\overline \Omega =
\begin{cases}
\Omega & \text{if} \ \Omega \notin \widetilde \Psi_0 \ \text{or} \ \Omega \in \widetilde \Psi_1;\\
\Omega \setminus \lbrace \lambda - \theta_\Omega \rbrace & \text{if} \ \Omega \in \widetilde \Psi_2.
\end{cases}
\]

Recall from Lemma~\ref{lemma_Psi0max_non-acute} that the $C$-roots in $\Psi_0^{\max}$ are linearly independent.
Then we can choose an element $\varrho \in \Hom_\ZZ(\mathfrak X(C), \ZZ)$ such that $\varrho(\lambda) < 0$ and $\varrho(\lambda') = 0$ for all $\lambda' \in \Psi_0^{\max} \setminus \lbrace \lambda \rbrace$.
Let $\phi \colon \CC^\times \to C$ be the one-parameter subgroup of $C$ corresponding to~$\varrho$, that is, $\chi(\phi(t)) = t^{\varrho(\chi)}$ for all $\chi \in \mathfrak X(C)$ and $t \in \CC^\times$.

For every $t \in \CC^\times$, we put $\mathfrak h_t = \phi(t)\mathfrak h$.
According to Proposition~\ref{prop_limits}, there exists $\lim \limits_{t \to \infty} \mathfrak h_t$, we denote it by~$\mathfrak h_\infty$.
In what follows, $\mathfrak h_\infty$ is referred to as the \textit{multiplicative degeneration} of~$\mathfrak h$ defined by~$\lambda$.

\begin{proposition} \label{prop_limitI}
There is a decomposition
\begin{equation} \label{eqn_limitI}
\mathfrak h_\infty^\perp = \mathfrak p_u \oplus (\mathfrak h^\perp \cap \mathfrak c) \oplus \bigoplus\limits_{\Omega \in \widetilde \Psi} \mathfrak u_\infty(\Omega)
\end{equation}
where for any $\Omega \in \widetilde \Psi$ the subspace $\mathfrak u_\infty(\Omega)$ is a $K$-module that projects isomorphically onto each $\mathfrak g(\mu)$ with $\mu \in \overline \Omega$ and projects trivially to each $\mathfrak g(\mu)$ with $\mu \in \Phi^+ \setminus \overline \Omega$.
In particular, for every $\Omega \in \widetilde \Psi$ there is a $K$-module isomorphism $\mathfrak u_\infty(\Omega) \simeq \mathfrak u(\Omega)$.
\end{proposition}

\begin{proof}
Observe that $\mathfrak h_\infty^\perp = \lim \limits_{t \to \infty} \mathfrak h_t^\perp$.
Clearly, the first two summands of decomposition~(\ref{eqn_h^perp_refined}) are $\phi(t)$-stable, hence
\[
\mathfrak h_t^\perp = \mathfrak p_u \oplus (\mathfrak h^\perp \cap \mathfrak c) \oplus \bigoplus\limits_{\Omega \in \widetilde \Psi} \mathfrak u_t(\Omega)
\]
where $\mathfrak u_t(\Omega) = \phi(t) \mathfrak u(\Omega)$ for all $\Omega \in \widetilde \Psi$ and $t \in \CC^\times$.
For every $\Omega \in \widetilde \Psi$, put $\mathfrak u_\infty(\Omega) = \lim \limits_{t \to \infty} \mathfrak u_t(\Omega)$, this limit exists by Proposition~\ref{prop_limits}.
Since each $\mathfrak u(\Omega)$ projects nontrivially only to the subspaces $\mathfrak g(\mu)$ with $\mu \in \Omega$, it follows that $\mathfrak u_\infty(\Omega) \subset \bigoplus \limits_{\mu \in \Omega} \mathfrak g(\mu)$, which readily implies~(\ref{eqn_limitI}).

It is easy to see that for $\Omega \notin \widetilde \Psi_2$ the subspace $\mathfrak u(\Omega)$ is $\phi(t)$-stable, hence $\mathfrak u_\infty(\Omega) = \mathfrak u(\Omega)$.
It remains to consider the case~$\Omega \in \widetilde \Psi_2$.
Recall $\theta_\Omega \in \ZZ^+\Phi^+$ and $\Omega' \in \widetilde \Psi_0^{\max}$ such that $\Omega + \theta_\Omega \subset \Omega'$.
Put $\mu_0 = \lambda - \theta_\Omega \in \Omega$, so that $\Omega \setminus \lbrace \mu_0 \rbrace = \overline \Omega$.
Fix a basis $E(\Omega)$ in $\mathfrak u(\Omega)$.
For every $v \in E(\Omega)$ and every $\mu \in \Omega$, let $v_\mu$ be the projection of $v$ to $\mathfrak g(\mu)$, so that $v = \sum \limits_{\mu \in \Omega} v_\mu$.
Then for every $v \in E(\Omega)$ one has $\phi(t)v = \sum\limits_{\mu \in \Omega} t^{\varrho(\mu)} v_\mu$.
Multiplying this vector by $t^{\varrho(\theta_\Omega)}$, we obtain
\[
t^{\varrho(\theta_\Omega)} \phi(t)v = \sum\limits_{\mu \in \Omega} t^{\varrho(\mu+\theta_\Omega)} v_\mu =
t^{\varrho(\lambda)} v_{\mu_0} + \sum \limits_{\mu \in \overline \Omega} v_\mu.
\]
Now observe the following:
\begin{itemize}
\item
the set $\lbrace t^{\varrho(\theta_\Omega)}\phi(t) v \mid v \in E(\Omega) \rbrace$ is a basis of~$\mathfrak u_t(\Omega)$ for all $t \in \CC^\times$;

\item
for every $v \in E(\Omega)$ the limit $\lim \limits_{t \to \infty} (t^{\varrho(\theta_\Omega)} \phi(t)v)$ exists and equals $v_\infty = \sum \limits_{\mu \in \overline \Omega} v_\mu$;

\item
the set $E_\infty(\Omega) = \lbrace v_\infty \mid v \in E(\Omega) \rbrace$ is linearly independent.
\end{itemize}
It follows from the above observations that $E_\infty(\Omega)$ is a basis of~$\mathfrak u_\infty(\Omega)$, which implies the required property of~$\mathfrak u_\infty(\Omega)$.
\end{proof}

Let $H_\infty \subset G$ be the connected subgroup with Lie algebra $\mathfrak h_\infty$.
Then decomposition~(\ref{eqn_limitI}) implies

\begin{theorem} \label{thm_typeI}
The following assertions hold.
\begin{enumerate}[label=\textup{(\alph*)},ref=\textup{\alph*}]
\item \label{thm_typeI_a}
The subgroup $H_\infty$ is regularly embedded in~$P$ and $K$ is a Levi subgroup of~$H_\infty$.

\item \label{thm_typeI_b}
There is a $K$-module isomorphism $\mathfrak p_u/(\mathfrak h_\infty)_u \simeq \mathfrak p_u / \mathfrak h_u$.

\item
One has $\Psi(H_\infty) = \bigcup \limits_{\Omega \in \widetilde \Psi} \overline \Omega$.
Moreover, $\widetilde \Psi(H_\infty) = \lbrace \overline \Omega \mid \Omega \in \widetilde \Psi\rbrace$.
\end{enumerate}
\end{theorem}

We note that $H$ is spherical in~$G$ by \cite[Proposition~1.3(i)]{Bri90}; however, this can be verified directly via Theorem~\ref{thm_typeI}(\ref{thm_typeI_a},\,\ref{thm_typeI_b}) and Proposition~\ref{prop_S=K}.

Now put $N = N_G(H_\infty)^0$ for short.
From Proposition~\ref{prop_normalizer} we know that $N = D_\infty L' \rightthreetimes (H_\infty)_u$ for some connected subgroup $D_\infty \subset C$.
To describe $D_\infty$ more precisely, we consider the following sublattices of~$\mathfrak X(C)$:
\begin{gather*}
\Xi_\infty = \ZZ\lbrace \mu - \nu \mid \mu, \nu \in \Psi_0^{\max}\setminus \lbrace \lambda \rbrace \ \text{and} \ \mu \sim \nu \rbrace \subset \mathfrak X(C), \\
\Xi_\infty^\sat = \QQ\Xi_\infty \cap \mathfrak X(C).
\end{gather*}
The next lemma is obtained similarly to Lemma~\ref{lemma_XiAA0}.

\begin{lemma} \label{lemma_Xi0D0}
The following assertions hold.
\begin{enumerate}[label=\textup{(\alph*)},ref=\textup{\alph*}]
\item \label{lemma_Xi0D0_a}
$\Xi_\infty = \ZZ\lbrace \mu - \nu \mid \mu, \nu \in \Psi(H_\infty) \ \text{and} \ \mu \sim \nu \rbrace$.

\item \label{lemma_Xi0D0_b}
The group $D_\infty$ is the common kernel of all characters in~$\Xi_\infty^\sat$.
\end{enumerate}
\end{lemma}

Put $K_\infty = D_\infty \cdot L'$; this is a Levi subgroup of~$N$.

Let $\iota \colon \mathfrak X(T) \to \mathfrak X(T \cap K)$ and $\iota_\infty \colon \mathfrak X(T) \to \mathfrak X(T \cap K_\infty)$ be the character restriction maps.

Put $\mathfrak u_\infty = \bigoplus \limits_{\Omega \in \widetilde \Psi} \mathfrak u_\infty(\Omega)$ and note the $K$-module isomorphism $\mathfrak u_\infty^* \simeq \mathfrak p_u / (\mathfrak h_\infty)_u$.

\begin{theorem} \label{thm_Nprop_typeI}
The following assertions hold.
\begin{enumerate}[label=\textup{(\alph*)},ref=\textup{\alph*}]
\item \label{thm_Nprop_typeI_a}
$\dim G\mathfrak h_\infty = \dim G\mathfrak h - 1$.

\item \label{thm_Nprop_typeI_b}
$\QQ\Lambda(G/N) \cap \lbrace \widehat \mu \mid \mu \in \Psi \rbrace = \lbrace \widehat \mu \mid \mu \in \Psi(H_\infty) \rbrace$.
In particular, $\widehat \lambda \notin \QQ\Lambda(G/N)$ and $\widehat \mu \in \QQ\Lambda(G/N)$ for all $\mu \in \Psi_0^{\max} \setminus \lbrace \lambda \rbrace$.
\end{enumerate}
\end{theorem}

\begin{proof}
(\ref{thm_Nprop_typeI_a})
Note that $\Xi = \Xi_\infty \oplus \ZZ(\mu - \lambda)$ for every $\mu \in \Omega_\lambda \setminus \lbrace \lambda \rbrace$.
Then $\dim D_\infty - \dim D^0 = 1$ and thus $\dim N - \dim H = 1$ as required.

(\ref{thm_Nprop_typeI_b})
Recall from Proposition~\ref{prop_limitI} that $\mathfrak u \simeq \mathfrak u_\infty$ as $K$-modules.
For every $\Omega \in \widetilde \Psi$ fix an element $\nu_\Omega \in \overline\Omega$.
Clearly, $\lbrace \iota_\infty(\widehat \nu_\Omega) \mid \Omega \in \widetilde \Psi \rbrace \subset F_{K_\infty}(\mathfrak u_\infty^*)$ and $\lbrace \iota(\widehat \nu_\Omega) \mid \Omega \in \widetilde \Psi \rbrace \subset F_{K}(\mathfrak u_\infty^*)$.
Choose a subset $I \subset \mathfrak X(T)$ such that $\lbrace \widehat \nu_\Omega \mid \Omega \in \widetilde \Psi \rbrace \subset I$ and $\left.\iota_\infty\right|_I$ is a bijection onto~$F_{K_\infty}(\mathfrak u_\infty^*)$.
Taking into account the inclusion $K \subset K_\infty$ and Proposition~\ref{prop_free_generators}(\ref{prop_free_generators_a}), we deduce that
$\left.\iota\right|_I$ is a bijection onto~$F_{K}(\mathfrak u_\infty^*)$.
Now consider the following two sublattices of~$\mathfrak X(T)$:
\begin{gather*}
\widehat \Xi = \ZZ \lbrace \widehat \mu - \widehat \nu \mid \mu, \nu \in \Psi_0^{\max} \ \text{and} \ \mu \sim \nu \rbrace, \\
\widehat \Xi_\infty = \ZZ\lbrace \widehat \mu - \widehat \nu \mid \mu, \nu \in \Psi_0^{\max}\setminus \lbrace \lambda \rbrace \ \text{and} \ \mu \sim \nu \rbrace.
\end{gather*}
Clearly, the restriction of $\widehat \Xi$ (resp.~$\widehat \Xi_\infty$) to $C$ is $\Xi$ (resp.~$\Xi_\infty$) and the restrictions of both $\widehat \Xi$, $\widehat \Xi_\infty$ to $T \cap L'$ are trivial.
Taking into account the decompositions
\begin{gather*}
\QQ\mathfrak X(T) \simeq \QQ\mathfrak X(C) \oplus \QQ\mathfrak X(T \cap L'), \\
\QQ\mathfrak X(T \cap K) \simeq \QQ\mathfrak X(D^0) \oplus \QQ\mathfrak X(T \cap L'), \\
\QQ\mathfrak X(T \cap K_\infty) \simeq \QQ\mathfrak X(D_\infty) \oplus \QQ\mathfrak X(T \cap L')
\end{gather*}
along with Lemmas~\ref{lemma_XiAA0}(\ref{lemma_XiAA0_b}) and \ref{lemma_Xi0D0}(\ref{lemma_Xi0D0_b}), we obtain
\begin{gather}
\QQ \Ker \iota = \QQ \widehat \Xi, \label{eqn_Ker_iota} \\
\QQ \Ker \iota_\infty = \QQ \widehat \Xi_\infty. \label{eqn_Ker_iota0}
\end{gather}
Then Proposition~\ref{prop_S=K} implies
\begin{gather}
\QQ\Lambda_G(G/H) = \QQ\widehat \Xi \oplus \QQ I, \label{eqn_WL} \\
\QQ\Lambda_G(G/N) = \QQ\widehat \Xi_\infty \oplus \QQ I. \label{eqn_WL0}
\end{gather}
Now take any $\mu \in \Psi(H_\infty)$.
Then $\mu - \nu_{\Omega_\mu} \in \Xi_\infty$ by Lemma~\ref{lemma_Xi0D0}(\ref{lemma_Xi0D0_a}), hence $\widehat \mu - \widehat \nu_{\Omega_\mu} \in \Ker \iota_\infty$.
As $\widehat\nu_{\Omega_\mu} \in I$, formulas (\ref{eqn_WL0}) and (\ref{eqn_Ker_iota0}) imply $\widehat \mu \in \QQ\Lambda_G(G/N)$.
Next take any $\mu \in \Psi \setminus \Psi(H_\infty)$.
Then $\mu + \theta_{\Omega_\mu} = \lambda$ and the element $\nu = \nu_{\Omega_\mu} + \theta_{\Omega_\mu}$ belongs to~$\Omega_\lambda$, hence the definitions of $\widehat \Xi$ and $\widehat \Xi_\infty$ imply $\widehat \Xi = \widehat \Xi_\infty \oplus \ZZ(\widehat \lambda - \widehat \nu)$.
Now observe that the restrictions of $\widehat \lambda - \widehat \nu$ and $\widehat \mu - \widehat \nu_{\Omega_\mu}$ to both $C$ and $T \cap L'$ coincide, hence $\widehat \lambda - \widehat \nu = \widehat \mu - \widehat \nu_{\Omega_\mu}$ and $\widehat \Xi = \widehat \Xi_\infty \oplus \ZZ(\widehat \mu - \widehat \nu_{\Omega_\mu})$.
Comparing this with~(\ref{eqn_WL}) and~(\ref{eqn_WL0}) we find that $\QQ\Lambda_G(G/H) = \QQ\Lambda_G(G/N) \oplus \QQ(\widehat \mu - \widehat \nu_{\Omega_\mu})$.
Since $\widehat \nu_{\Omega_\mu} \in I \subset \QQ\Lambda_G(G/N)$, it follows that $\widehat \mu \notin \QQ\Lambda_G(G/N)$.
\end{proof}

\begin{corollary}
For different choices of $\lambda \in \Psi^{\max}_0$, the resulting algebras $\mathfrak h_\infty$ belong to different $G$-orbits of codimension~$1$ in~$X$.
\end{corollary}

\begin{remark}
As was pointed out by a referee, the setup of multiplicative degenerations appeared earlier (but was used differently) in~\cite[proof of Theorem~5.3.1]{BP14}.
\end{remark}

\subsection{Additive degenerations}
\label{subsec_degen_add}

This type of degenerations applies when $\Psi_0 = \varnothing$ and $\Psi \ne \varnothing$, which is assumed in what follows.
As $N_G(H)^0 = H$, one has $K = L$ according to the discussion in~\S\,\ref{subsec_normalizer}.
Then decomposition~(\ref{eqn_h^perp_refined}) takes the form
\begin{equation} \label{eqn_h_perpII}
\mathfrak h^\perp = \mathfrak p_u \oplus \bigoplus \limits_{\mu \in \Psi} \mathfrak g(\mu).
\end{equation}

The construction of an additive degeneration depends on the choice of an active $C$-root $\lambda \in \Psi$, which is assumed to be fixed throughout this subsection.

Put $\delta = \widehat \lambda$ and let $\mathfrak s(\delta) \simeq \mathfrak{sl}_2$ be the subalgebra of~$\mathfrak g$ spanned by $e_\delta$, $h_\delta$, and $e_{-\delta}$.
Consider the one-parameter unipotent subgroup $\phi \colon \CC \to G$ given by $\phi(t) = \exp(te_{-\delta})$.
For every $t \in \CC$, we put $\mathfrak h_t = \phi(t)\mathfrak h$.
According to Proposition~\ref{prop_limits}, there exists $\lim \limits_{t \to \infty} \mathfrak h_t$; we denote it by~$\mathfrak h_\infty$.
In what follows, $\mathfrak h_\infty$ is referred to as the \textit{additive degeneration} of~$\mathfrak h$ defined by~$\lambda$.

Note that $\mathfrak h_t^\perp = \phi(t) \mathfrak h^\perp$ and $\mathfrak h_\infty^\perp = \lim \limits_{t \to \infty} \mathfrak h_t^\perp$.

To describe the subalgebra $\mathfrak h_\infty$, we introduce the set \[
Y(\delta) = \lbrace \alpha \in \Delta \mid \alpha+ \delta \notin \Delta \rbrace.
\]
For every $\alpha \in Y(\delta)$, let $V(\alpha) \subset \mathfrak g$ be the $\mathfrak s(\delta)$-submodule generated by~$e_\alpha$.
The following properties of $V(\alpha)$ are straightforward:
\begin{itemize}
\item
$V(\alpha)$ is a simple $\mathfrak s(\delta)$-module with highest weight $\delta^\vee(\alpha)$;

\item
$e_\alpha$ is a highest-weight vector of~$V(\alpha)$;

\item
$V(\alpha)$ is $T$-stable.
\end{itemize}
Then there is the following decomposition of $\mathfrak g$ into a direct sum of $\mathfrak s(\delta)$-submodules:
\begin{equation} \label{eqn_g_decompII}
\mathfrak g = (h_\delta^\perp \cap \mathfrak t) \oplus \bigoplus \limits_{\alpha \in Y(\delta)} V(\alpha).
\end{equation}
Comparing this with~(\ref{eqn_h_perpII}) we find that
\begin{equation} \label{eqn_h^perp_decompII}
\mathfrak h^\perp = \bigoplus \limits_{\alpha \in Y(\delta)} (\mathfrak h^\perp \cap V(\alpha)).
\end{equation}
By Proposition~\ref{prop_limits}, for every $\alpha \in Y(\delta)$ there exists $\lim \limits_{t \to \infty} (\mathfrak h^\perp_t \cap V(\alpha))$, which we shall denote by $(\mathfrak h^\perp \cap V(\alpha))_\infty$.
Then decompositions~(\ref{eqn_g_decompII}) and~(\ref{eqn_h^perp_decompII}) imply

\begin{proposition} \label{prop_h_infty}
There is the decomposition
\[
\mathfrak h_\infty^\perp = \bigoplus \limits_{\alpha \in Y(\delta)} (\mathfrak h^\perp \cap V(\alpha))_\infty.
\]
\end{proposition}

For every $\alpha \in Y(\delta)$ the limit $(\mathfrak h^\perp \cap V(\alpha))_\infty$ is determined using Proposition~\ref{prop_limitII_prelim}.
Since the subspace $\mathfrak h^\perp \cap V(\alpha) \subset V(\alpha)$ is $h_\delta$-stable, $(\mathfrak h^\perp \cap V(\alpha))_\infty$ is described in terms of shifting $h_\delta$-weight subspaces in $\mathfrak h^\perp \cap V(\alpha)$ as explained in the paragraph after Proposition~\ref{prop_limitII_prelim}.
We shall use this description in our analysis of the structure of~$\mathfrak h_\infty$.

To state the main properties of $\mathfrak h_\infty$, we apply the construction of \S\,\ref{subsec_char_of_SM} with $G = L$, $V = \mathfrak u$, and $\omega = \delta$.
Put $Q = \lbrace g \in L \mid \Ad(g) \mathfrak g_{\delta} = \mathfrak g_{\delta} \rbrace$; this is a parabolic subgroup of~$L$ containing $B_L$.
Let $Q^- \supset B^- \cap L$ be the parabolic subgroup of~$L$ opposite to~$Q$.
Let $M$ be the standard Levi subgroup of~$Q$ and let $M_0$ be the stabilizer of $e_\delta$ in~$M$.
Put $\Delta^+_L(\delta) = \lbrace \alpha \in \Delta_L^+ \mid (\alpha, \delta) > 0 \rbrace$, so that $\Delta^+_L(\delta) = \Delta^+_L \setminus \Delta^+_M$.
Regard the element $e_{-\delta}$ as a linear function on~$\mathfrak u$ via the fixed $G$-invariant inner product on~$\mathfrak g$.
Put
\[
\widetilde {\mathfrak u} = \lbrace x \in \mathfrak u \mid (\mathfrak q_u e_{-\delta})(x) = 0 \rbrace
\]
and $\mathfrak u_0 = \widetilde {\mathfrak u} \cap \Ker e_{-\delta}$.
Note that
\[
\mathfrak u_0 = \bigoplus \limits_{\substack{\alpha \in \Delta^+ \colon \overline \alpha \in \Psi, \\ \delta - \alpha \notin \Delta^+_L(\delta) \cup \lbrace 0 \rbrace}} \mathfrak g_{\alpha}.
\]
Then there is the decomposition $\mathfrak u = \mathfrak g_{\delta} \oplus [\mathfrak q^-_u, \mathfrak g_\delta] \oplus \mathfrak u_0$ into a direct sum of $M$-modules.

We consider the decomposition $\mathfrak h_\infty^\perp = (\mathfrak h_\infty^\perp \cap \mathfrak p_u) \oplus (\mathfrak h_\infty^\perp \cap \mathfrak l) \oplus (\mathfrak h_\infty^\perp \cap \mathfrak p^+_u)$.
Put $\mathfrak u_\infty = \mathfrak h_\infty^\perp \cap \mathfrak p_u^+$ for short.

\begin{proposition} \label{prop_h0_typeII}
The following assertions hold.
\begin{enumerate}[label=\textup{(\alph*)},ref=\textup{\alph*}]
\item \label{prop_h0_typeII_a}
$\mathfrak h_\infty^\perp \cap \mathfrak p_u = \mathfrak p_u$.

\item \label{prop_h0_typeII_b}
$\mathfrak h_\infty^\perp \cap \mathfrak l = \mathfrak q_u^- \oplus \langle h_\delta \rangle$.

\item \label{prop_h0_typeII_c}
The subspace $\mathfrak u_\infty$ is $M$-stable and there is an $M_0$-module isomorphism $\mathfrak u_0 \simeq \mathfrak u_\infty$. Moreover, under this isomorphism each highest-weight vector in $\mathfrak u_0$ of $T$-weight~$\alpha$ corresponds to a highest-weight vector in $\mathfrak u_\infty$ of $T$-weight $\alpha - k_\alpha \delta$ for some $k_\alpha \in \ZZ^+$.
\end{enumerate}
\end{proposition}

\begin{proof}
(\ref{prop_h0_typeII_a})
Since $[\mathfrak g_{-\delta},\mathfrak p_u] \subset \mathfrak p_u$, it follows that every root subspace $\mathfrak g_\alpha \subset \mathfrak p_u$ is shifted to itself under the degeneration, which implies the claim.

(\ref{prop_h0_typeII_b})
Observe that $\delta \in Y(\delta)$, $V(\delta) = \mathfrak s(\delta)$, and $V(\delta) \cap \mathfrak h^\perp = \mathfrak g_{-\delta} \oplus \mathfrak g_{\delta}$, which yields $V(\delta) \cap \mathfrak h_\infty^\perp = \mathfrak g_{-\delta} \oplus \langle h_\delta \rangle$.
As $\mathfrak s(\delta)$ acts trivially on $\mathfrak h_\delta^\perp \cap \mathfrak t$ and the latter subspace has zero intersection with~$\mathfrak h^\perp$, we get $\mathfrak h_\infty^\perp \cap \mathfrak t = \langle h_\delta \rangle$.
Now take any $\alpha \in \Delta_L$ and recall from~(\ref{eqn_h_perpII}) that $\mathfrak g_\alpha \notin \mathfrak h^\perp$.
If $(\alpha, \delta) \ge 0$ then $\mathfrak g_\alpha \subset \mathfrak q$ and $[\mathfrak g_\alpha, \mathfrak g_\delta] = 0$ as $\delta$ is the highest weight of~$\mathfrak g(\lambda)$.
Consequently, $\alpha \in Y(\delta)$ and hence no root subspace in $\mathfrak h^\perp$ shifts to~$\mathfrak g_{\alpha}$ under the degeneration.
If $(\alpha, \delta) < 0$ then $\mathfrak g_\alpha \subset \mathfrak q_u^-$, $\alpha + \delta \in \Delta$, and $\mathfrak g_{\alpha+\delta} \subset [\mathfrak q_u^-, \mathfrak g_\delta] \subset \mathfrak g(\lambda) \subset \mathfrak u$.
Since $[\mathfrak g_\alpha, \mathfrak g_{-\delta}] \subset \mathfrak p_u$, it follows that each nonzero subspace $\mathfrak g_{\alpha - k\delta}$ with $k \ge 1$ shifts to itself under the degeneration and hence $\mathfrak g_{\alpha+\delta}$ shifts to~$\mathfrak g_{\alpha}$.
The claim follows.

(\ref{prop_h0_typeII_c})
As a byproduct of the above discussion, we obtain that the only root subspaces in $\mathfrak u$ whose shift under the degeneration belongs to $\mathfrak p_u \oplus \mathfrak l$ are those in $\mathfrak g_\delta \oplus [\mathfrak q_u^-,\mathfrak g_\delta]$.
Consequently, for every $\alpha \in \Delta^+$ with $\mathfrak g_\alpha \subset \mathfrak u_0$ the subspace $\mathfrak g_\alpha$ shifts to some root subspace in $\mathfrak p^+_u$ under the degeneration.
Let $\alpha \in \Delta^+$ be the highest weight of a simple $M$-submodule $\mathfrak u_0(\alpha) \subset \mathfrak u_0$ and let $k \in \ZZ^+$ be such that $\mathfrak g_\alpha$ shifts to $\mathfrak g_{\alpha - k\delta}$.
Choose $\beta \in Y(\delta)$ such that $\mathfrak g_\alpha \subset V(\beta)$.
Note that the subalgebras $\mathfrak m_0$ and $\mathfrak s(\delta)$ commute, therefore the element $e_\beta$ is a highest-weight vector of an $(\mathfrak m_0 \oplus \mathfrak s(\delta))$-module $W \subset \mathfrak g$ isomorphic to $\mathfrak u_0(\alpha) \otimes V(\beta)$.
Note that $\mathfrak u_0(\alpha)  \subset W$, $\mathfrak g_{\alpha - k\delta} \subset W$, and $e_{\alpha - k\delta}$ is a highest-weight vector of a simple $M_0$-submodule $\mathfrak u_0(\alpha - k\delta) \subset W$.
Since $\mathfrak h^\perp$ is $M_0$-stable, it follows that for every $\gamma \in \Delta^+$ with $\mathfrak g_\gamma \subset \mathfrak u_0(\alpha)$ the subspace $\mathfrak g_\gamma$ shifts to $\mathfrak g_{\gamma - k\delta} \in \mathfrak u_0(\alpha - k \delta)$ under the degeneration, which implies all the claims.
\end{proof}

Put $R = Q^- \rightthreetimes P_u$; then $R$ is a standard parabolic subgroup of~$G$ containing $B^-$ and having $M$ as a Levi subgroup.
Let $H_\infty \subset G$ be the connected subgroup with Lie algebra~$\mathfrak h_\infty$ and consider the subgroup $N = N_G(H_\infty)^0$.
We note that $H_\infty$ is spherical in~$G$ by \cite[Proposition~1.3(i)]{Bri90}, although this can be verified directly via Proposition~\ref{prop_S=K} and Theorems~\ref{thm_typeII}(\ref{thm_typeII_a}) and~\ref{thm_sph_modules_imp}(\ref{thm_sph_modules_imp_a}).

\begin{theorem} \label{thm_typeII}
The following assertions hold.
\begin{enumerate}[label=\textup{(\alph*)},ref=\textup{\alph*}]
\item \label{thm_typeII_a}
The subgroup $H_\infty$ is regularly embedded in $R$, $M_0$ is a Levi subgroup of~$H_\infty$, and $\mathfrak r_u / (\mathfrak h_\infty)_u \simeq \mathfrak u_\infty^*$ as $M_0$-modules.

\item \label{thm_typeII_b}
$N = M \rightthreetimes (H_\infty)_u$.

\item \label{thm_typeII_c}
$\dim G\mathfrak h_\infty = \dim G\mathfrak h - 1$.

\item \label{thm_typeII_d}
$\Lambda_G(G/H) = \Lambda_G(G/N) \oplus \ZZ\delta$.
In particular, $\delta \notin \ZZ\Lambda_G(G/N)$.
\end{enumerate}
\end{theorem}

\begin{proof}
(\ref{thm_typeII_a})
This follows directly from Proposition~\ref{prop_h0_typeII}.

(\ref{thm_typeII_b})
As $H$ is spherical in~$G$ and $\mathfrak (\mathfrak h_\infty)_u$ is normalized by~$T$, the assertion is implied by Proposition~\ref{prop_normalizer}.

(\ref{thm_typeII_c})
Part~(\ref{thm_typeII_b}) yields $\dim N - \dim H = 1$, which is equivalent to the required equality.

(\ref{thm_typeII_d})
Proposition~\ref{prop_S=K} implies $\Lambda_G(G/H) = \Lambda_L(\mathfrak u^*)$ and similarly $\Lambda_G(G/N) = \Lambda_M(\mathfrak u_\infty^*)$.
Next, one has $\Lambda_L(\mathfrak u^*) = \Lambda_M(\widetilde{\mathfrak u}^*)$ by Theorem~\ref{thm_sph_modules_imp}(\ref{thm_sph_modules_imp_b}).
In view of the decomposition $\widetilde{\mathfrak u} = \mathfrak g_\delta \oplus\nobreak \mathfrak u_0$ we obtain $\Lambda_M(\widetilde{\mathfrak u}^*) = \ZZ \delta \oplus \Lambda_M (\mathfrak u_0^*)$.
Combining Proposition~\ref{prop_h0_typeII}(\ref{prop_h0_typeII_c}) with Proposition~\ref{prop_free_generators} we find that there is a bijection $F_M(\mathfrak u_0^*) \to F_M(\mathfrak u_\infty^*)$ taking each element $\alpha$ to $\alpha - k_\alpha \delta$ for some $k_\alpha \in \ZZ^+$.
It remains to notice that $\Lambda_M(\mathfrak u_0^*) = \ZZ F_M(\mathfrak u_0^*)$ and $\Lambda_M(\mathfrak u_\infty^*) = \ZZ F_M(\mathfrak u_\infty^*)$ by Proposition~\ref{prop_SM_WL_and_WM}.
\end{proof}

The objects $\mathfrak h_\infty, N, \ldots$ defined above depend on the initial choice of an active $C$-root $\lambda \in \Psi$.
To emphasize this dependence, we shall write $\mathfrak h_\infty(\lambda), N(\lambda), \ldots$.
Now choose two different active $C$-roots $\lambda_1,\lambda_2 \in \Psi$ and put $\delta_1 = \widehat \lambda_1, \delta_2 = \widehat \lambda_2$.

\begin{proposition}
The algebras $\mathfrak h_\infty(\lambda_1)$ and $\mathfrak h_\infty(\lambda_2)$ belong to different $G$-orbits of codimension~$1$ in~$X$.
\end{proposition}

\begin{proof}
We shall establish the claim by showing that
\begin{equation} \label{eqn_two_lattices}
\Lambda_G(G/N(\lambda_1)) \ne \Lambda_G(G/N(\lambda_2)).
\end{equation}
From Theorem~\ref{thm_typeII}(\ref{thm_typeII_d}) we know that $\delta_i \notin \Lambda_G(G/N(\lambda_i))$ for $i = 1,2$.
According to Proposition~\ref{prop_h0_typeII}(\ref{prop_h0_typeII_c}) there exist non-negative integers $k_1, k_2$ such that $\delta_2' = \delta_2 - k_1 \delta_1$ is a highest weight of the $M(\lambda_1)$-module $\mathfrak u_\infty(\lambda_1)$ and $\delta_1' = \delta_1 - k_2 \delta_2$ is a highest weight of the $M(\lambda_2)$-module $\mathfrak u_\infty(\lambda_2)$.
Then $\delta_2' \in \Gamma_{M(\lambda_1)}(\mathfrak u_\infty(\lambda_1)^*)$ and  $\delta_1' \in \Gamma_{M(\lambda_2)}(\mathfrak u_\infty(\lambda_2)^*)$, hence $\delta_2' \in \Lambda_G(G/N(\lambda_1))$ and $\delta_1' \in \Lambda_G(G/N(\lambda_2))$ by Propositions~\ref{prop_SM_WL_and_WM} and~\ref{prop_S=K}.
If $k_1 = 0$ then we obtain $\delta_2 \in \Lambda_G(G/N(\lambda_1))$, which yields~(\ref{eqn_two_lattices}).
Similarly, (\ref{eqn_two_lattices}) follows if $k_2 = 0$.
Now assume $k_1 > 0, k_2 > 0$.
Observe that $\delta_1', \delta_2' \in \Delta^+$; then $\delta_1'+\delta_2' \in \ZZ^+\Delta^+ \setminus \lbrace 0 \rbrace$.
On the other hand, $\delta'_1 + \delta'_2 = -(k_1-1)\delta_1 - (k_2-1)\delta_2 \in - \ZZ^+\Delta^+$, a contradiction.
\end{proof}

\subsection{Primitive cases}
\label{subsec_primitive_cases}

Here we investigate the case when $|\Psi| = 1$, so that $\Psi = \lbrace \lambda \rbrace$ for some $\lambda \in \Phi^+$.
Let $\alpha \in \Delta^+$ be the lowest weight of the $L$-module $\mathfrak g(\lambda)$.

\begin{lemma}
$\alpha \in \Pi$.
\end{lemma}

\begin{proof}
Assume that $\alpha = \beta + \gamma$ for some $\beta,\gamma \in \Delta^+$.
Then $\overline \alpha = \overline \beta + \overline \gamma$.
Since $\alpha$ is the lowest weight of~$\mathfrak g(\overline \alpha)$, one has $\beta,\gamma \notin \Delta_L^+$ and so $\overline \beta, \overline \gamma \in \Phi^+$.
Then one of $\overline \beta,\overline \gamma$ must be an active $C$-root by Lemma~\ref{lemma_sum}, a contradiction.
Thus $\alpha \in \Pi$.
\end{proof}

Applying the reduction of the ambient group (see \S\,\ref{subsec_reduction_AG}) we may assume $\Supp \widehat \lambda = \Pi$.
Then the following conditions hold:

\begin{enumerate}[label=\textup{(P\arabic*)},ref=\textup{P\arabic*}]
\item \label{P1}
$G$ is a simple group;

\item
$H$ is regularly embedded in a maximal parabolic subgroup $P \supset B^-$ with standard Levi subgroup $L$ such that $\Pi_L = \Pi \setminus \lbrace \alpha \rbrace$;

\item \label{P3}
$H = L \rightthreetimes P_u'$.
\end{enumerate}
(The last condition is implied by Proposition~\ref{prop_properties_of_g(lambda)}(\ref{prop_properties_of_g(lambda)_b}).)

All spherical homogeneous spaces $G/H$ satisfying the above properties are classified in the following theorem, which also provides the corresponding set of spherical roots in each case.

\begin{theorem} \label{thm_primitive_cases}
Suppose that $G$ is a simple group, $\alpha \in \Pi$, and $H \subset G$ is a subgroup satisfying conditions \textup(\ref{P1}\textup)--\textup(\ref{P3}\textup).
Then
\begin{enumerate}[label=\textup{(\alph*)},ref=\textup{\alph*}]
\item \label{thm_primitive_cases_a}
$H$ is a spherical subgroup of $G$ if and only if, up to an automorphism of the Dynkin diagram of~$G$, the pair $(\mathfrak g,\alpha)$ appears in Table~\textup{\ref{table_primitive_cases}};

\item \label{thm_primitive_cases_b}
For each pair $(\mathfrak g,\alpha)$ listed in Table~\textup{\ref{table_primitive_cases}} the set $\Sigma_G(G/H)$ is given in the fifth column of that table.
\end{enumerate}
\end{theorem}

\begin{table}[h]
\caption{} \label{table_primitive_cases}

\begin{center}
\renewcommand{\tabcolsep}{3pt}%
\begin{tabular}{|c|l|l|c|l|c|}
\hline
No. & \multicolumn{1}{|c|}{($\mathfrak g,\alpha$)} & \multicolumn{1}{|c|}{($L', \mathfrak g(-\overline\alpha)$)} & $\rk$ & \multicolumn{1}{|c|}{$\Sigma_G(G/H)$} & Note \\

\hline
\no\lefteqn{^*}\label{No1} &
($\mathfrak{sl}_n,\alpha_k$) &
(${\SL_k} {\times} {\SL_{n-k}},
\CC^k {\otimes} \CC^{n-k}$) &
$k$ &
\renewcommand{\tabcolsep}{0pt}%
\begin{tabular}{l}
$\alpha_i {+} \alpha_{n-i}$ for $1 {\le} i {\le} k{-}1$, \\[-2pt]
$\alpha_k {+} \ldots {+} \alpha_{n-k}$
\end{tabular}
&
\renewcommand{\tabcolsep}{0pt}%
\begin{tabular}{l}
$n {\ge} 2$, \\[-2pt]
$k {\le} n/2$
\end{tabular}
\\

\hline

\no &
($\mathfrak{sp}_{2n},\alpha_1$) &
(${\Sp_{2n-2}}, \CC^{2n-2}$) &
$1$ &
$\alpha_1 {+} 2(\sum_{i=2}^{n-1} \alpha_i) {+} \alpha_n$
&
$n {\ge} 2$
\\

\hline

\no\label{No3} &
($\mathfrak{sp}_{2n},\alpha_2$) &
(${\SL_2}{\times}{\Sp_{2n-4}}, \CC^2{\otimes}\CC^{2n-4}$) &
$3$ &
\renewcommand{\tabcolsep}{0pt}%
\begin{tabular}{l}
$\alpha_1 {+} \alpha_3$, $\alpha_2$,\\
$\alpha_3 {+} 2(\sum_{i=4}^{n-1}\alpha_i) {+} \alpha_n$
\end{tabular}
&
$n {\ge} 4$
\\

\hline

\no\label{No4} &
($\mathfrak{sp}_{2n},\alpha_3$) &
(${\SL_3}{\times}{\Sp_{2n-6}}, \CC^{3}{\otimes}\CC^{2n-6}$) &
$6$ &
\renewcommand{\tabcolsep}{0pt}%
\begin{tabular}{l}
$\alpha_1$, $\alpha_2$, $\alpha_3$, $\alpha_4$, $\alpha_5$,\\
$\alpha_5 {+} 2(\sum_{i=6}^{n-1}\alpha_i) {+} \alpha_n$
\end{tabular}
&
$n {\ge} 6$
\\

\hline

\no\label{No5} &
($\mathfrak{sp}_{10},\alpha_3$) &
(${\SL_3}{\times}{\Sp_{4}}, \CC^{3}{\otimes}\CC^{4}$) &
$5$ &
$\alpha_1$, $\alpha_2$, $\alpha_3$, $\alpha_4$, $\alpha_5$
&
\\

\hline

\no\label{No6} &
($\mathfrak{sp}_{2n},\alpha_{n-2}$) &
(${\SL_{n-2}{\times}{\Sp_4}}, \CC^{n-2}{\otimes}\CC^4$) &
$6$ &
\renewcommand{\tabcolsep}{0pt}%
\begin{tabular}{l}
$\alpha_1$, $\alpha_2$, $\alpha_3$, $\alpha_{n-1}$, $\alpha_n$,\\[-2pt]
$\alpha_4 {+} \ldots {+} \alpha_{n-2}$
\end{tabular}
&
$n {\ge} 6$
\\

\hline

\no\label{No7} &
($\mathfrak{sp}_{2n},\alpha_{n-1}$) &
(${\SL_{n-1}{\times}{\SL_2}}, \CC^{n-1}{\otimes}\CC^2$) &
$2$ &
\renewcommand{\tabcolsep}{0pt}%
\begin{tabular}{l}
$\alpha_{1}{+}\alpha_n$, $\alpha_{2} {+} \ldots {+} \alpha_{n-1}$
\end{tabular}
&
$n {\ge} 3$
\\

\hline

\no\lefteqn{^*}\label{No8} &
($\mathfrak{sp}_{2n},\alpha_{n}$) &
(${\SL_{n}}, \mathrm{S}^2\CC^{n}$) &
$n$ &
\renewcommand{\tabcolsep}{0pt}%
\begin{tabular}{l}
$2\alpha_{i}$ for $1 {\le} i {\le} n{-}1$, $\alpha_{n}$
\end{tabular}
&
$n {\ge} 2$
\\

\hline

\no\lefteqn{^*}\label{No9} &
($\mathfrak{so}_{2n+1},\alpha_1$) &
(${\SO_{2n-1}}, \CC^{2n-1}$) &
$2$ &
$\alpha_1$, $2\alpha_{2} {+} \ldots {+} 2\alpha_{n}$
&
$n {\ge} 3$
\\

\hline

\no &
($\mathfrak{so}_{2n+1},\alpha_n$) &
(${\SL_{n}}, \CC^{n}$) &
$1$ &
$\alpha_{1} {+} \ldots {+} \alpha_{n}$
&
$n {\ge} 3$
\\

\hline

\no\lefteqn{^*}\label{No11} &
($\mathfrak{so}_{2n},\alpha_1$) &
(${\SO_{2n-2}}, \CC^{2n-2}$) &
$2$ &
\renewcommand{\tabcolsep}{0pt}%
\begin{tabular}{l}
$\alpha_1$, \\[-2pt]
$2\alpha_{2} {+} \ldots {+} 2\alpha_{n-2} {+} \alpha_{n-1} {+} \alpha_{n}$
\end{tabular}
&
$n {\ge} 4$
\\

\hline

\no\lefteqn{^*}\label{No12} &
($\mathfrak{so}_{4n+2},\alpha_{2n+1}$) &
(${\SL_{2n+1}}, \wedge^2\CC^{2n+1}$) &
$n$ &
\renewcommand{\tabcolsep}{0pt}%
\begin{tabular}{l}
$\alpha_{2i-1} {+} 2\alpha_{2i} {+} \alpha_{2i+1}$ \\[-2pt]
for $1 {\le} i {\le} n{-}1$, \\[-2pt]
$\alpha_{2n-1} {+} \alpha_{2n} {+} \alpha_{2n+1}$
\end{tabular}
&
$n {\ge} 2$
\\

\hline

\no\lefteqn{^*}\label{No13} &
($\mathfrak{so}_{4n},\alpha_{2n}$) &
(${\SL_{2n}}, \wedge^2\CC^{2n}$) &
$n$ &
\renewcommand{\tabcolsep}{0pt}%
\begin{tabular}{l}
$\alpha_{2i-1} {+} 2\alpha_{2i} {+} \alpha_{2i+1}$ \\[-2pt]
for $1 {\le} i {\le} n{-}1$, 
$\alpha_{2n}$
\end{tabular}
&
$n {\ge} 2$
\\

\hline

\no &
$(\mathfrak g_2, \alpha_1)$ &
$(\SL_2, \CC^2)$ &
$1$ &
$\alpha_1 {+} \alpha_2$ &
\\

\hline

\no\label{No15} &
$(\mathfrak f_4, \alpha_3)$ &
$({\SL_3} {\times} {\SL_2}, \CC^3 {\otimes} \CC^2)$ &
$2$ &
$\alpha_1 {+} \alpha_4$, $\alpha_2 {+} \alpha_3$ &
\\

\hline

\no\label{No16} &
$(\mathfrak f_4, \alpha_4)$ &
$(\Spin_7, R(\varpi_3))$ &
$2$ &
$\alpha_1 {+} 2\alpha_2 {+} 3\alpha_3$, $\alpha_4$ &
\\

\hline

\no\lefteqn{^*}\label{No17} &
$(\mathfrak e_6, \alpha_6)$ &
$(\Spin_{10}, R(\varpi_5))$ &
$2$ &
\renewcommand{\tabcolsep}{0pt}%
\begin{tabular}{l}
$\alpha_1 {+} \alpha_3 {+} \alpha_4 {+} \alpha_5 {+} \alpha_6$,\\[-2pt]
$2\alpha_2 {+} \alpha_3 {+} 2\alpha_4 {+} \alpha_5$
\end{tabular}
&
\\

\hline

\no\lefteqn{^*}\label{No18} &
$(\mathfrak e_7, \alpha_7)$ &
$(\mathsf E_6, R(\varpi_1))$ & 3 &
\renewcommand{\tabcolsep}{0pt}%
\begin{tabular}{l}
$2\alpha_1 {+} \alpha_2 {+} 2\alpha_3 {+} 2\alpha_4 {+} \alpha_5$,\\[-2pt]
$\alpha_2 {+} \alpha_3 {+} 2\alpha_4 {+} 2\alpha_5 {+} 2\alpha_6$, $\alpha_7$\\
\end{tabular}
&
\\

\hline
\end{tabular}
\end{center}
\end{table}

In Table~\ref{table_primitive_cases}, the symbols $R(\lambda)$ and $\varpi_i$ have the same meaning as in Table~\ref{table_spherical_modules}.

\begin{remark}
For convenience of the reader, for each pair $(\mathfrak g, \alpha)$ in Table~\ref{table_primitive_cases} we also included the information on the $L'$-module structure of~$\mathfrak g(-\overline \alpha)$ (up to equivalence) as well as the value of rank of~$\mathfrak g(-\overline \alpha)$ as a spherical $L$-module.
(This rank equals the cardinality of the set~$\Sigma_G(G/H)$.)
\end{remark}

\begin{proof}[Proof of Theorem~\textup{\ref{thm_primitive_cases}}]
(\ref{thm_primitive_cases_a})
Suppose that the pair $(\mathfrak g, \alpha)$ is fixed.
Taking into account the $L$-module isomorphism $\mathfrak p_u/\mathfrak h_u \simeq \mathfrak g(-\overline \alpha)$ and applying Proposition~\ref{prop_S=K}, we find that $H$ is spherical if and only if $\mathfrak g(-\overline \alpha)$ is a spherical $L$-module.
Clearly, $C$ acts on $\mathfrak g(-\overline \alpha)$ via the character $-\overline \alpha$ and the highest weight of $\mathfrak g(-\overline \alpha)$ as an $L'$-module is uniquely determined by the numbers $\beta^\vee(-\alpha)$ with $\beta \in \Pi \setminus \lbrace \alpha \rbrace$.
Having determined the structure of $\mathfrak g(-\overline \alpha)$ as a $(C \times L')$-module, one then checks if $\mathfrak g(-\overline \alpha)$ is spherical using Theorem~\ref{thm_simple_sph_modules}.

A case-by-case check of all possible pairs $(\mathfrak g,\alpha)$ shows that $\mathfrak g(-\overline \alpha)$ is a spherical $L$-module if and only if $(\mathfrak g,\alpha)$ appears in Table~\ref{table_primitive_cases}.

(\ref{thm_primitive_cases_b})
In all the cases in Table~\ref{table_primitive_cases}, the subgroup $H$ turns out to be wonderful and the corresponding set of spherical roots is already known.
Below we give references for all the cases.

If the group $P_u'$ is trivial then $H$ is reductive; these cases are marked with an asterisk in Table~\ref{table_primitive_cases}.
The corresponding sets of spherical roots are taken from~\cite[\S\,3]{BP15}.
More precisely, cases~\ref{No1}, \ref{No8}, \ref{No9}, \ref{No11}, \ref{No12}, \ref{No13}, \ref{No17}, \ref{No18} in Table~\ref{table_primitive_cases} correspond to cases~3, 12, 8, 14, 16, 16, 18, 22 in loc.~cit., respectively.

If $\rk_G(G/H) = 1$ then the unique spherical root is read off directly from the weight lattice and equals the highest weight of the $L$-module $\mathfrak g(\overline \alpha)$.

If $\rk_G(G/H) = 2$ and the subgroup $H$ is not reductive then the corresponding set of spherical roots is taken from~\cite[\S\,3]{Was}.
More precisely, cases~\ref{No7}, \ref{No15}, \ref{No16} in Table~\ref{table_primitive_cases} correspond to cases $4'$ of Table~C, 1 of Table~F, 2 of Table~F in loc.~cit., respectively.

For the remaining cases~\ref{No3}--\ref{No6} in Table~\ref{table_primitive_cases}, the information on the corresponding sets of spherical roots follows from~\cite[\S\,3.3 and Proposition~3.3.1]{BP16}.
This information is also given in an explicit form in Table~2 of the preprint~\cite{BP09}; cases~3, 5, 11, 18 of that table correspond to our cases~\ref{No3}, \ref{No5}, \ref{No6}, \ref{No4} in Table~\ref{table_primitive_cases}, respectively.
\end{proof}

\section{An optimization of the base algorithm for the case \texorpdfstring{$K = L$}{K = L}}
\label{sect_optimization}

Let $G$ be semisimple and retain all the notation of~\S\,\ref{sect_active_C-roots}.
Throughout this section, we work with subgroups $H$ satisfying $K = L$, so that $\Psi_0 = \varnothing$ and $\mathfrak u = \bigoplus \limits_{\mu \in \Psi} \mathfrak g(\mu)$.

Before proceeding, we make two notation conventions.

Firstly, for every standard Levi subgroup $\widetilde L \subset G$ (not necessarily equal to~$L$) with connected center $\widetilde C$ and every $\xi \in \mathfrak X(\widetilde C)$, the $\widetilde C$-weight subspace of $\mathfrak g$ of weight $\xi$ will be denoted by~$\mathfrak g(\xi)$.
If $\xi$ is a $\widetilde C$-root then the highest weight of $\mathfrak g(\xi)$ as an $\widetilde L$-module will be denoted by~$\widehat \xi$ and we put by definition $\Supp \xi = \Supp \widehat \xi$.
For every set $\widetilde \Phi_0$ of $\widetilde C$-roots we put $\Supp \widetilde \Phi_0 = \bigcup \limits_{\xi \in \widetilde \Phi_0} \Supp \xi$.
Finally, if $\widetilde L \subset L$ then for every $\xi \in \mathfrak X(\widetilde C)$ we shall write $\overline \xi$ for the restriction of $\xi$ to~$\mathfrak X(C)$.
In this section, various pieces of the above notation will be used simultaneously for different standard Levi subgroups of~$G$, but in each case the corresponding Levi subgroup will be clear from the context.

Secondly, if $N \subset G$ is a spherical subgroup different from $H$ and we want to consider analogues for $N$ of objects like $\Psi, \mathfrak u, \ldots$ defined for $H$ then we shall denote them like $\Psi(N), \mathfrak u(N), \ldots$.

\subsection{The main idea}
\label{subsec_main_idea}

There is a decomposition into a disjoint union
\begin{equation} \label{eqn_SM-dec}
\Psi = \Psi_1 \cup \ldots \cup \Psi_p
\end{equation}
with the following properties:
\begin{itemize}
\item
for every simple factor $F$ of $L'$ acting nontrivially on $\mathfrak u$ there exists a unique $i \in \lbrace 1,\ldots, p \rbrace$ such that $F$ acts trivially on each $\mathfrak g(\mu)$ with $\mu \notin \Psi_i$;

\item
for every $i = 1,\ldots,p$, the saturation of the $L$-module $\mathfrak u^i = \bigoplus \limits_{\mu \in \Psi_i} \mathfrak g(\mu)$ is indecomposable\footnote{An equivalent reformulation is as follows: $\mathfrak u^i$ is indecomposable as an $L'$-module.}.
\end{itemize}
In what follows, decomposition~(\ref{eqn_SM-dec}) will be called the \textit{SM-decomposition}\footnote{`SM' is used as an abbreviation for `saturation of a module'.} of~$\Psi$.
Note that the components of this decomposition are uniquely determined up to permutation.

Our optimization of the base algorithm applies when $p \ge 2$ and rests on the following idea.
Each spherical root of $G/H$ is somehow ``controlled'' by exactly one component of the SM-decomposition of~$H$, and to ``extract'' all spherical roots controlled by a given component $\Psi_i$ we perform a special chain of additive degenerations to obtain a new spherical subgroup $H_i$ such that the pair $(L,\mathfrak u^i)$ is equivalent to $(L(H_i), \mathfrak u(H_i))$ and the spherical roots of $H_i$ are precisely those of $G/H$ controlled by~$\Psi_i$.
In this way, we obtain a fast algorithm that reduces computing the spherical roots for $H$ to the same problem for several other spherical subgroups for which the SM-decomposition consists of a single component.

We remark that the classification of saturated indecomposable spherical modules mentioned in~\S\,\ref{subsec_SM_classification} implies $|\Psi_i| \le 2$ for all~$i = 1,\ldots, p$.
Although we shall not make use of this property in our arguments in this section, it clearly shows that the reduction described in the previous paragraph yields a substantial optimization of the base algorithm.

\subsection{Auxiliary results}

Fix $\lambda \in \Psi$ and define $\delta$, $\mathfrak u_0$, $\mathfrak u_\infty$, $N$, $M$, $M_0$ as in~\S\,\ref{subsec_degen_add}.
Then $\mathfrak u_\infty = \mathfrak u(N)$ according to our notation conventions.
We know from the proof of Proposition~\ref{prop_h0_typeII}(\ref{prop_h0_typeII_c}) that there is a bijection $\alpha \mapsto \pi(\alpha)$ between the $T$-weights of $\mathfrak u_0$ and the $T$-weights of $\mathfrak u(N)$ such that $\mathfrak g_\alpha$ shifts to $\mathfrak g_{\pi(\alpha)}$ under the degeneration.
Recall that $\pi(\alpha) = \alpha - k_\alpha \delta$ with $k_\alpha \ge 0$.
We say that a $T$-stable subspace $V \subset \mathfrak u_0$ shifts to a $T$-stable subspace $W \subset \mathfrak u(N)$ under the degeneration if the set of $T$-weights of $W$ is the image of that of~$V$ under~$\pi$.
For future reference, we state the following reformulation of Proposition~\ref{prop_h0_typeII}(\ref{prop_h0_typeII_c}):

\begin{enumerate}[label=\textup{($\diamond$)},ref=\textup{$\diamond$}]
\item \label{property}
under the degeneration, every simple $M$-module $V \subset \mathfrak u_0$ with highest weight $\alpha$ shifts to a simple $M$-module $W \subset \mathfrak u(N)$ with highest weight~$\alpha - k_\alpha \delta$ for some $k_\alpha \ge 0$; moreover, $V \simeq W$ as $M_0$-modules.
\end{enumerate}

Given a subset $\Theta \subset \Psi$ and an element $\nu \in \Psi$, we say that $\nu$ is an \textit{upper} element of $\Theta$ if $\nu \in \Theta$ and $\mu - \nu \notin \Phi^+$ for all $\mu \in \Theta \setminus \lbrace \nu \rbrace$.
Observe that every nonempty subset of $\Psi$ contains at least one upper element.

\begin{lemma} \label{lemma_stable_under_degen}
Suppose that $\lambda \in \Theta$ for some $\Theta \subset \Psi$ and $\lambda$ is an upper element of~$\Theta$.
Then the subspace $\mathfrak u_0 \cap ( \bigoplus \limits_{\mu \in \Theta} \mathfrak g(\mu))$ shifts to itself under the degeneration.
\end{lemma}

\begin{proof}
Assume there is a simple $M$-submodule $W \subset \mathfrak u_0 \cap (\bigoplus \limits_{\mu \in \Theta} \mathfrak g(\mu))$ that does not shift to itself under the degeneration and let $\alpha$ be its highest weight.
Then $\alpha - \delta \in \Delta^+\setminus \Delta_L^+$.
Let $\mu \in \Theta$ be such that $W \subset \mathfrak g(\mu)$.
Then $\overline \alpha - \overline \delta = \mu - \lambda \in \Phi^+$, a contradiction.
\end{proof}

Let
$
\Psi(N) = \Psi_1(N) \cup \ldots \cup \Psi_q(N)
$
be the SM-decomposition of~$\Psi(N)$.

For every $i = 1,\ldots, p$ let $S_i$ be the product of simple factors of $L'$ that act nontrivially on~$\mathfrak u^i$.

Until the end of this subsection, we assume that some $i \in \lbrace 1,\ldots, p \rbrace$ is fixed and $\lambda \notin \Psi_i$.
The next lemma is straightforward.

\begin{lemma} \label{lemma_S_i}
The following assertions hold.
\begin{enumerate}[label=\textup{(\alph*)},ref=\textup{\alph*}]
\item
$\mathfrak u^i \subset \mathfrak u_0$.

\item
$S_i$ is a normal subgroup of $M'$.

\item
$S_i$ is the product of all simple factors of $M'$ that act nontrivially on~$\mathfrak u^i$.

\item
$S_i$ acts trivially on each simple $M$-submodule of $\mathfrak u_0$ not contained in~$\mathfrak u^i$.

\item
Every simple $L$-submodule of $\mathfrak u^i$ remains simple when regarded as an $M$-module.
\end{enumerate}
\end{lemma}

Combining Lemma~\ref{lemma_S_i} with~(\ref{property}) we obtain

\begin{proposition} \label{prop_unique_j}
There exists a unique $j \in \lbrace 1,\ldots,q \rbrace$ with the following properties:
\begin{enumerate}[label=\textup{(\arabic*)},ref=\textup{\arabic*}]
\item
$\mathfrak u^i$ shifts to $\mathfrak u^j(N)$ under the degeneration;

\item
$S_j(N) = S_i$;

\item
there is a bijection $\Psi_i \to \Psi_j(N)$, $\mu \mapsto \mu_*$, such that for every $\mu \in \Psi_i$ one has an $S_i$-module isomorphism $\mathfrak g(\mu_*) \simeq \mathfrak g(\mu)$ and $\widehat \mu_* = \widehat \mu - c_\mu \delta$ where $c_\mu \ge 0$.
\end{enumerate}
\end{proposition}

\begin{lemma} \label{lemma_SR_of_SM}
One has $\Sigma_M(\mathfrak u(N)^*) \subset \Sigma_L(\mathfrak u^*)$.
\end{lemma}

\begin{proof}
This is implied by Propositions~\ref{prop_SM_red_SR1} and~\ref{prop_free_generators}(\ref{prop_free_generators_d}) along with~(\ref{property}).
\end{proof}

\begin{lemma} \label{lemma_WL_quotient}
Let $j \in \lbrace 1,\ldots,q\rbrace$ be such that $\mathfrak u^i$ shifts to $\mathfrak u^j(N)$ under the degeneration.
Then $\Lambda_L((\mathfrak u/\mathfrak u^i)^*) = \Lambda_M((\mathfrak u(N)/\mathfrak u^j(N))^*) \oplus \ZZ\delta$.
\end{lemma}

\begin{proof}
As $\lambda \notin \Psi_i$, Theorem~\ref{thm_sph_modules_imp}(\ref{thm_sph_modules_imp_b}) yields $\Lambda_L((\mathfrak u/\mathfrak u^i)^*) = \Lambda_M((\mathfrak u_0/\mathfrak u^i)^*) \oplus \ZZ\delta$.
Combining~(\ref{property}) with Proposition~\ref{prop_free_generators}(\ref{prop_free_generators_b}) we find that there is a bijection
\[
F_M((\mathfrak u_0/\mathfrak u^i)^*) \to F_M((\mathfrak u(N)/\mathfrak u^j(N))^*)
\]
that takes each $\alpha$ to $\alpha - k_\alpha \delta$ for some $k_\alpha \ge 0$.
As $\Lambda_M((\mathfrak u_0/\mathfrak u^i)^*) = \ZZ F_M((\mathfrak u_0/\mathfrak u^i)^*)$ and $\Lambda_M((\mathfrak u(N)/\mathfrak u^j(N))^*) = \ZZ F_M((\mathfrak u(N)/\mathfrak u^j(N))^*)$ by Proposition~\ref{prop_SM_WL_and_WM}, we conclude that $\Lambda_M((\mathfrak u_0/\mathfrak u^i)^*) \oplus \ZZ\delta = \Lambda_M((\mathfrak u(N)/\mathfrak u^j(N))^*) \oplus \ZZ\delta$.
\end{proof}

For every $i = 1,\ldots, p$ put
$\Upsilon_i= \lbrace \mu \in \Psi \setminus \Psi_i \mid \Supp \mu \subset \Supp(\Psi_i) \rbrace$.
Consider also the $\QQ$-vector space $\Xi=\QQ\Sigma_L(\mathfrak u^*)$.

\begin{proposition} \label{prop_opt_A}
Suppose that $\lambda$ is an upper element of $\Upsilon_i$ and let $j \in \lbrace 1,\ldots, q \rbrace$ be as in Proposition~\textup{\ref{prop_unique_j}}.
Then for every $\nu \in \Upsilon_j(N)$ there exists $\mu \in \Upsilon_i$ such that $\mathfrak g(\nu) \subset \mathfrak g(\mu)$ and $\widehat \nu \equiv \widehat \mu \mod \Xi$.
\end{proposition}

\begin{proof}
Applying~(\ref{property}) we find that there is $s \ge 0$ such that $\widehat \nu + s\delta$ is the highest weight of a simple $M$-module in~$\mathfrak u_0$.
As $\Supp \nu \subset \Supp \Psi_j(N) \subset \Supp(\Psi_i)$ and $\Supp \delta \subset \Supp(\Psi_i)$, we find that $\Supp(\widehat\nu + s\delta) \subset \Supp (\Psi_i)$ and hence $\mu = \overline \nu + s \overline \delta \in \Upsilon_i$.
By Lemma~\ref{lemma_stable_under_degen}, the subspace $\mathfrak u_0 \cap(\bigoplus \limits_{\mu \in \Upsilon_i} \mathfrak g(\mu))$ shifts to itself under the degeneration, which implies $s = 0$ and proves the first claim.
The second claim follows from Proposition~\ref{prop_SM_red_SR2}.
\end{proof}

For a given pair $(H, \Psi_k)$ with $k \in \lbrace 1,\ldots,p \rbrace$, we shall consider the following condition:
\begin{enumerate}[label=\textup{($\sharp$)},ref=\textup{$\sharp$}]
\item \label{H_part_case}
$\mu -\nu \notin \Phi^+$ for all $\mu \in \Psi_k$ and $\nu \in \Psi \setminus \Psi_k$.
\end{enumerate}

\begin{proposition} \label{prop_opt_B}
Suppose that $(H, \Psi_i)$ satisfies~\textup(\ref{H_part_case}\textup) and $\lambda$ is an upper element of $\Psi \setminus \Psi_i$.
Then
\begin{enumerate}[label=\textup{(\alph*)},ref=\textup{\alph*}]
\item \label{prop_opt_B_a}
the subspace $\mathfrak u_0$ shifts to itself under the degeneration; in particular, $\mathfrak u(N) = \mathfrak u_0$;

\item \label{prop_opt_B_b}
there exists a unique $j \in \lbrace 1,\ldots,q \rbrace$ with the following properties:
\begin{enumerate}[label=\textup{(\arabic*)},ref=\textup{\arabic*}]
\item \label{prop_opt_B_1}
$\mathfrak u^i = \mathfrak u^j(N)$;

\item \label{prop_opt_B_2}
$S_j(N) = S_i$;

\item \label{prop_opt_B_3}
the pair $(N, \Psi_j(N))$ satisfies~\textup(\ref{H_part_case}\textup).
\end{enumerate}
\end{enumerate}
\end{proposition}

\begin{proof}
(\ref{prop_opt_B_a})
It follows from the hypothesis that $\lambda$ is an upper element of $\Psi$.
Then the claim is implied by Lemma~\ref{lemma_stable_under_degen}.

(\ref{prop_opt_B_b})
Taking into account part~(\ref{prop_opt_B_a}) and applying Proposition~\ref{prop_unique_j} we find that there exists a unique $j \in \lbrace 1,\ldots, q \rbrace$ satisfying (\ref{prop_opt_B_1}) and~(\ref{prop_opt_B_2}), so it remains to prove~(\ref{prop_opt_B_3}).
Thanks to part~(\ref{prop_opt_B_a}), the restriction to $C$ of any element in $\Psi(N)$ belongs to~$\Psi$.
In view of~(\ref{prop_opt_B_1}), for every $\xi \in \Psi_j(N)$ and $\theta \in \Psi(N) \setminus \Psi_j(N)$ one has $\overline \xi \in \Psi_i$ and $\overline \theta \in \Psi \setminus \Psi_i$.
Then $\overline \xi - \overline \theta \notin \Phi^+ \cup \lbrace 0 \rbrace$ and hence $\xi - \theta \notin \Phi^+(N)$.
\end{proof}

\subsection{The subgroup~\texorpdfstring{$H_i$}{H\_i}}
\label{subsec_subgroup_H_i}

Throughout this subsection, $i$ denotes an arbitrary element of $\lbrace 1,\ldots,p \rbrace$.
The goal of this subsection is to define the subgroup $H_i$ mentioned in~\S\,\ref{subsec_main_idea} and discuss some properties of it.

We begin with describing several algorithms and discussing their properties.

\medskip

Algorithm~\newalg: \label{alg_A}

Input: a pair $(H, \Psi_i)$

Step~\step: \label{step_A1}
if $\Upsilon_i = \varnothing$ then exit and return $(H, \Psi_i)$;

Step~\step: \label{step_A2}
choose an upper element $\lambda \in \Upsilon_i$, compute the additive degeneration $\mathfrak h_\infty$ of $\mathfrak h$ defined by~$\lambda$ and put $N = N_G(\mathfrak h_\infty)^0$;

Step~\step: \label{step_A3}
identify $j$ as in Proposition~\ref{prop_unique_j};

Step~\step: \label{step_A4}
repeat the procedure for the pair $(N, \Psi_j(N))$.

\medskip

\begin{proposition} \label{prop_reduction_A}
Let $(N_i, \Psi_k(N_i))$ denote an output of Algorithm~\textup{\ref{alg_A}}.
Then
\begin{enumerate}[label=\textup{(\alph*)},ref=\textup{\alph*}]
\item \label{prop_reduction_A_a}
$S_k(N_i)=S_i$;

\item \label{prop_reduction_A_b}
there is a bijection $\Psi_i \to \Psi_k(N_i)$, $\mu \mapsto \mu_*$, such that for every $\mu \in \Psi_i$ one has an $S_i$-module isomorphism $\mathfrak g(\mu_*) \simeq \mathfrak g(\mu)$ and $\widehat \mu_* \equiv \widehat \mu - \sum \limits_{\nu \in \Upsilon_i} c_\nu \widehat \nu \mod \Xi$ with $c_\nu \ge 0$ for all~$\nu$;

\item \label{prop_reduction_A_c}
the pair $(N_i, \Psi_k(N_i))$ satisfies~\textup(\ref{H_part_case}\textup).
\end{enumerate}
\end{proposition}

\begin{proof}
Parts~(\ref{prop_reduction_A_a}) and~(\ref{prop_reduction_A_b}) follow directly from the description of the algorithm, Propositions~\ref{prop_unique_j}, \ref{prop_opt_A}, and Lemma~\ref{lemma_SR_of_SM}.
Part~(\ref{prop_reduction_A_c}) is implied by $\Upsilon_k(N_i)=\varnothing$.
\end{proof}

Algorithm~\newalg: \label{alg_B}

Input: a pair $(H, \Psi_i)$ satisfying~(\ref{H_part_case})

Step~\step: \label{step_B1}
if $\Psi \setminus \Psi_i = \varnothing$ then exit and return~$(H, \Psi_i)$;

Step~\step: \label{step_B2}
choose an upper element $\lambda \in \Psi \setminus \Psi_i$, compute the additive degeneration $\mathfrak h_\infty$ of $\mathfrak h$ defined by~$\lambda$ and put $N = N_G(\mathfrak h_\infty)^0$;

Step~\step: \label{step_B3}
identify $j$ as in Proposition~\ref{prop_opt_B}(\ref{prop_opt_B_b});

Step~\step: \label{step_B4}
repeat the procedure for the pair $(N, \Psi_j(N))$.

\medskip

We remark that Algorithm~\ref{alg_B} is well defined by property~(\ref{prop_opt_B_3}) of Proposition~\ref{prop_opt_B}(\ref{prop_opt_B_b}).

The next result follows from the description of the algorithm and Proposition~\ref{prop_opt_B}.

\begin{proposition} \label{prop_reduction_F}
Let $(N_i, \Psi_k(N_i))$ be an output of Algorithm~\textup{\ref{alg_B}}.
Then
\begin{enumerate}[label=\textup{(\alph*)},ref=\textup{\alph*}]
\item \label{prop_reduction_F_a}
$\Psi(N_i)=\Psi_k(N_i)$, that is, the SM-decomposition of $\Psi(N_i)$ has only one component;

\item \label{prop_reduction_F_b}
$\mathfrak u^k(N_i)= \mathfrak u^i$;

\item \label{prop_reduction_F_c}
$S_k(N_i)=S_i$.
\end{enumerate}
\end{proposition}

Algorithm~\newalg: \label{alg_C}

Input: a pair $(H,\Psi_i)$

Step~\step: \label{step_C1}
apply Algorithm~\ref{alg_A} to $(H,\Psi_i)$;

Step~\step: \label{step_C2}
apply Algorithm~\ref{alg_B} to the output of step~\ref{step_C1}.

\medskip

Note that Algorithm~\ref{alg_C} is well defined since the output of Algorithm~\ref{alg_A} satisfies~(\ref{H_part_case}) by Proposition~\ref{prop_reduction_A}(\ref{prop_reduction_A_c}).

Let $(H_i, \Psi_k(H_i))$ denote an output of Algorithm~\ref{alg_C} and let $L_i$ be the Levi subgroup of~$H_i$.

\begin{proposition} \label{prop_H_i}
The following assertions hold.
\begin{enumerate}[label=\textup{(\alph*)},ref=\textup{\alph*}]
\item \label{prop_H_i_a}
$\Psi(H_i)=\Psi_k(H_i)$, that is, the SM-decomposition of $\Psi(H_i)$ has only one component.

\item \label{prop_H_i_b}
$S_k(H_i)=S_i$.

\item \label{prop_H_i_c}
There is a bijection $\Psi_i \to \Psi_k(H_i)$, $\mu \mapsto \mu_*$, such that for every $\mu \in \Psi_i$ one has an $S_i$-module isomorphism $\mathfrak g(\mu_*) \simeq \mathfrak g(\mu)$ and $\widehat \mu_* \equiv \widehat \mu - \sum \limits_{\nu \in \Upsilon_i} c_\nu \widehat \nu \mod \Xi$ with $c_\nu \ge 0$ for all~$\nu$.

\item \label{prop_H_i_d}
$\Sigma_{L_i}(\mathfrak u(H_i)^*) \subset \Sigma_L(\mathfrak u^*)$.

\item \label{prop_H_i_e}
$\Sigma_G(G/H_i) \subset \Sigma_G(G/H)$.

\item \label{prop_H_i_f}
$|\Sigma_G(G/H_i)| = \rk_L (\mathfrak u^i)^*$.
\end{enumerate}
\end{proposition}

\begin{proof}
(\ref{prop_H_i_a})--(\ref{prop_H_i_c}) These assertions follow directly from the description of the algorithm along with Propositions~\ref{prop_reduction_A} and~\ref{prop_reduction_F}.

(\ref{prop_H_i_d})
This follows from the construction and Lemma~\ref{lemma_SR_of_SM}.

(\ref{prop_H_i_e})
It follows from the construction and the discussion in~\S\,\ref{subsec_gen_strat} that each spherical root of $G/H_i$ is proportional to a spherical root of~$G/H$.
By Theorem~\ref{thm_typeII}(\ref{thm_typeII_d}), $\Lambda_G(G/H_i)$ is a direct summand of $\Lambda_G(G/H)$, therefore
each spherical root of $G/H_i$ is primitive in $\Lambda_G(G/H)$ and hence coincides with a spherical root of~$G/H$.

(\ref{prop_H_i_f})
Propositions~\ref{prop_finite_index_in_norm} and~\ref{prop_S=K} yield $|\Sigma_G(G/H_i)| = \rk_{L_i} \mathfrak u(H_i)^*$.
Parts (\ref{prop_H_i_a})--(\ref{prop_H_i_c}) imply that the latter value equals $\rk_L (\mathfrak u^i)^*$.
\end{proof}

\begin{remark}
Being an output of Algorithm~\ref{alg_C}, the subgroup~$H_i$ depends on the sequence of choices of~$\lambda$ at each execution of steps~\ref{step_A2} and~\ref{step_B2}.
The description of Algorithm~\ref{alg_C} along with Lemma~\ref{lemma_WL_quotient} imply that the Levi subgroup $L_i \subset H_i$ is uniquely determined by the formula $\Pi_{L_i} = \lbrace \alpha \in \Pi_L \mid (\alpha, \gamma) = 0 \ \text{for all} \ \gamma \in \Lambda_L((\mathfrak u / \mathfrak u^i)^*) \rbrace$ and hence does not depend on the above-mentioned sequence of choices.
We conjecture that the latter holds true for~$H_i$ itself, that is, the output of Algorithm~\ref{alg_C} is well defined.
\end{remark}

\begin{remark} \label{rem_estimates1}
Put $r = |\Sigma_G(G/H)|$ and $r_i = \rk_L (\mathfrak u^i)^*$ for all $i = 1,\ldots, p$.
Since one additive degeneration reduces the number of spherical roots by one, computing each subgroup $H_i$ via Algorithm~\ref{alg_C} requires $r-r_i$ degenerations.
Thus for computing the whole collection $H_1,\ldots, H_p$ one needs to perform exactly $pr-r$ degenerations, which in any case is no more than $r^2-r$.
We note that the value $r^2-r$ is attained when $H$ is strongly solvable.
\end{remark}

\subsection{The main result}

\begin{theorem} \label{thm_reduction}
There is a disjoint union $\Sigma_G(G/H) = \Sigma_G(G/H_1) \cup \ldots \cup \Sigma_G(G/H_p)$.
\end{theorem}

\begin{proof}
For every $i=1,\ldots, p$, let $\Pi_{S_i}$ be the set of simple roots of~$S_i$.

We have the chain
\[
|\Sigma_G(G/H)| = \rk_L\mathfrak u^* = \sum \limits_{i=1}^p \rk_L (\mathfrak u^i)^* = \sum \limits_{i=1}^p |\Sigma_G(G/H_i)|
\]
where the first equality is implied by Propositions~\ref{prop_finite_index_in_norm} and~\ref{prop_S=K} and the last one follows from Proposition~\ref{prop_H_i}(\ref{prop_H_i_f}).
In view of Proposition~\ref{prop_H_i}(\ref{prop_H_i_e}) it suffices to check that $\Lambda_G(G/H_i) \cap \Lambda_G(G/H_j) = \lbrace 0 \rbrace$ for all $i,j = 1,\ldots, p$ with $i \ne j$.
Note that $\Lambda_G(G/H_i) = \Lambda_{L_i} (\mathfrak u(H_i)^*)$ and $\Lambda_G(G/H_j) = \Lambda_{L_j} (\mathfrak u(H_j)^*)$ by Proposition~\ref{prop_S=K}.
Proposition~\ref{prop_H_i}(\ref{prop_H_i_c}) provides a bijection $\Psi(H) \to \Psi(H_1) \cup \ldots \cup \Psi(H_p)$, $\mu \mapsto \mu_*$, such that for every $k = 1,\ldots,p$ and every $\mu \in \Psi_k$ there is an expression

\begin{equation} \label{eqn_mu_*}
\widehat \mu_* \equiv \widehat \mu - \sum \limits_{\nu \in \Psi \setminus \Psi_k} c_{\mu, \nu} \widehat \nu \mod \Xi \quad \text{with} \quad c_{\mu,\nu} \ge 0.
\end{equation}
Now fix $i,j \in \lbrace 1,\ldots, p \rbrace$ with $i\ne j$ and take any $\sigma \in \Lambda_{L_i} (\mathfrak u(H_i)^*) \cap \Lambda_{L_j} (\mathfrak u(H_j)^*)$.
Then Proposition~\ref{prop_SR_of_SM} yields
\begin{equation} \label{eqn_sigma}
\sum \limits_{\mu \in \Psi_i} a_\mu \widehat \mu_* + \rho_i = \sigma = \sum \limits_{\nu \in \Psi_j} b_\nu \widehat \nu_* + \rho_j
\end{equation}
where $a_\mu,b_\nu \in \QQ$, $\rho_i \in \QQ\Sigma_{L_i}(\mathfrak u(H_i)^*)$, and $\rho_j \in \QQ\Sigma_{L_j}(\mathfrak u(H_j)^*)$.
By Proposition~\ref{prop_H_i}(\ref{prop_H_i_d}), we get
\begin{equation} \label{eqn_equiv}
\sum \limits_{\mu \in \Psi_i} a_\mu \widehat \mu_* \equiv \sum \limits_{\nu \in \Psi_j} b_\nu \widehat \nu_* \mod \Xi.
\end{equation}
It follows from the linear independence of $F_L(\mathfrak u^*)$ and Proposition~\ref{prop_SR_of_SM} that the weights $\widehat \lambda$ with $\lambda \in \Psi$ are linearly independent modulo~$\Xi$, therefore in the expressions of both sums in~(\ref{eqn_equiv}) via~(\ref{eqn_mu_*}) the coefficients at each $\widehat \lambda$ should coincide.
Put
\begin{gather*}
\Psi_i^+ = \lbrace \mu \in \Psi_i \mid a_\mu > 0 \rbrace, \quad
\Psi_i^- = \lbrace \mu \in \Psi_i \mid a_\mu < 0 \rbrace, \\
\Psi_j^+ = \lbrace \nu \in \Psi_j \mid b_\nu > 0 \rbrace, \quad
\Psi_j^- = \lbrace \nu \in \Psi_j \mid b_\nu < 0 \rbrace
\end{gather*}
and assume that $\Psi_j^+ \ne \varnothing$.
Then for every $\nu \in \Psi_j^+$ the coefficient at $\widehat \nu$ in the right-hand side of~(\ref{eqn_equiv}) equals exactly~$b_\nu$ and hence is positive.
In view of~(\ref{eqn_mu_*}), for the coefficient at $\widehat \nu$ in the left-hand side of~(\ref{eqn_equiv}) to be positive it is necessary that there exist $\mu \in \Psi_i^-$ with $c_{\mu,\nu} > 0$.
In particular, we find that $\Psi_i^- \ne \varnothing$.
Similarly, we show that for every $\mu \in \Psi_i^-$ there exists $\nu \in \Psi_j^+$ with $c_{\nu,\mu} > 0$.
Now put $\gamma = \sum \limits_{\mu \in \Psi_i^-} \widehat \mu_* + \sum \limits_{\nu \in \Psi_j^+} \widehat \nu_*$.
Then it follows from the above and~(\ref{eqn_mu_*}) that, modulo~$\Xi$, $\gamma$ is a linear combination of the set $\lbrace \widehat \varkappa \mid \varkappa \in \Psi\rbrace$ with nonpositive coefficients, which yields $\overline \gamma \in - \ZZ^+\Phi^+$ as the restriction of $\Xi$ to $\mathfrak X(C)$ vanishes.
On the other hand, all the summands in the expression for~$\gamma$ are positive roots, hence $\overline \gamma \in \ZZ^+\Phi^+$.
It follows that $\overline \gamma = 0$, therefore $\widehat \mu_*,\widehat \nu_* \in \Delta_L^+$ for all $\mu \in \Psi_i^-$ and $\nu \in \Psi_j^+$.
If $\dim \mathfrak g(\mu_*) \ge 2$ for some~$\mu \in \Psi_i^-$ then there is $\alpha \in \Pi_{S_i}$ such that $(\alpha, \widehat \mu_*) > 0$.
Then $\widehat \mu_*$ is necessarily a root of~$S_i$, which is impossible by the description of~$H_i$.
Consequently, $\dim \mathfrak g(\mu_*) = 1$ for all $\mu \in \Psi_i^-$, which by the definition of the SM-decomposition implies $|\Psi_i| = 1$, $\rho_i = 0$, and $\Psi_i = \Psi_i^-$.
Similarly, we show that $|\Psi_j| = 1$, $\rho_j =0$, and $\Psi_j = \Psi_j^+$.
Let $\mu$ (resp.~$\nu$) be the only element of~$\Psi_i$ (resp.~$\Psi_j$).
Then the left-hand side of~(\ref{eqn_sigma}) equals $a_\mu\widehat\mu_* \in -\ZZ^+\Pi \setminus \lbrace 0 \rbrace$ whereas the right-hand side equals $b_\nu\widehat \nu_* \in \ZZ^+ \Pi \setminus \lbrace 0 \rbrace$, a contradiction.
Thus $\Psi_j^+ = \varnothing$.
Repeating the same argument for $-\sigma$ we find that $\Psi_j^- = \varnothing$.
Similarly, $\Psi_i^+ = \Psi_i^- = \varnothing$, hence $a_\mu = b_\nu = 0$ for all $\mu \in \Psi_i$ and $\nu \in \Psi_j$.
Then~(\ref{eqn_sigma}) takes the form $\rho_i = \sigma = \rho_j$.
Since $\Supp \rho_i \subset \Pi_{S_i}$, $\Supp \rho_j \subset \Pi_{S_j}$, and $\Pi_{S_i} \cap \Pi_{S_j} = \varnothing$, it follows that $\sigma = 0$.
\end{proof}

\subsection{A further optimization}
\label{subsec_further_opt}

For every $i \in \lbrace 1,\ldots, p \rbrace$, let $H_i, L_i$ be as in~\S\,\ref{subsec_subgroup_H_i}, let $(G_i,K_i)$ be the pair obtained from $(G,H_i)$ by reduction of the ambient group, and regard the set of simple roots of~$G_i$ as a subset of~$\Pi$ (see~\S\,\ref{subsec_reduction_AG}).

Thanks to Proposition~\ref{prop_reduction_AG_sph_roots}, Theorem~\ref{thm_reduction} may be reformulated as follows.

\begin{theorem} \label{thm_reduction_ref}
There is a disjoint union $\Sigma_G(G/H) = \Sigma_{G_1}(G_1/K_1) \cup \ldots \cup \Sigma_{G_p}(G_p/K_p)$.
\end{theorem}

The goal of this subsection is to propose a shorter way to compute each pair $(G_i, K_i)$.

For every $i=1,\ldots, p$, the definition of~$\Upsilon_i$ implies that $\mathfrak l \oplus \bigoplus \limits_{\mu \in \Phi^+ \setminus(\Psi_i \cup \Upsilon_i)} \mathfrak g(-\mu)$ is a subalgebra of~$\mathfrak g$.
Let $\widehat H_i \subset G$ be the corresponding connected subgroup.
Note that $H \subset \widehat H_i$, $L$ is a Levi subgroup of~$\widehat H_i$, and $\Psi(\widehat H_i) = \Psi_i \cup \Upsilon_i$.
Observe that $\Psi_i$ is a component of the SM-decomposition of $\Psi(\widehat H_i)$.

\medskip

Algorithm~\newalg: \label{alg_D}

Input: a pair $(H,\Psi_i)$

Step~\step: \label{step_D1}
replace $(H, \Psi_i)$ with $(\widehat H_i, \Psi_i)$;

Step~\step: \label{step_D2}
perform steps \ref{step_A1}, \ref{step_A2}, \ref{step_A3};

Step~\step: \label{step_D3}
repeat the procedure for the pair $(N, \Psi_j(N))$.

\medskip

Observe that the output of Algorithm~\ref{alg_D} depends on the sequence of choices of~$\lambda$ at each execution of step~\ref{step_A2}.

\begin{proposition}
Given a pair~$(H, \Psi_i)$ with $i \in \lbrace 1,\ldots, p \rbrace$, for every implementation of Algorithm~\textup{\ref{alg_D}} with output $(N_i, \Psi_m(N_i))$ there exists an implementation of Algorithm~\textup{\ref{alg_C}} with output $(H_i, \Psi_k(H_i))$ such that the pair $(G_i,K_i)$ is obtained from $(G, N_i)$ by reduction of the ambient group.
\end{proposition}

\begin{proof}
The description of Algorithm~\ref{alg_D} implies that the SM-decomposition of $\Psi(N_i)$ has only one component, so that $\mathfrak u^m(N_i) = \mathfrak u(N_i)$.
Moreover, one clearly has $S_m(N_i) = S_i$.
Note that Algorithm~\ref{alg_D} differs from Algorithm~\ref{alg_A} by only adding step~\ref{step_D1} at each iteration.
Then we choose the implementation of Algorithm~\ref{alg_A} where each choice of~$\lambda$ at step~\ref{step_A2} is the same as that at the corresponding iteration of Algorithm~\ref{alg_D}.
Let $(\widetilde N_i, \Psi_l(\widetilde N_i))$ be the output of this implementation of Algorithm~\ref{alg_A}.
Taking into account the first claim in Proposition~\ref{prop_opt_A}, we find that $\mathfrak u^l(\widetilde N_i) = \mathfrak u^m(N_i) = \mathfrak u(N_i)$.
Now apply any implementation of Algorithm~\ref{alg_B} to $(\widetilde N_i, \Psi_l(\widetilde N_i))$ and let $(H_i, \Psi_k(H_i))$ be the output.
Then Propositions~\ref{prop_reduction_F}(\ref{prop_reduction_F_b}) and~\ref{prop_H_i}(\ref{prop_H_i_a}) imply $\mathfrak u(H_i) = \mathfrak u(N_i)$.
Recall from Proposition~\ref{prop_H_i}(\ref{prop_H_i_b}) that $S_i$ is the product of simple factors of $L_i$ that act nontrivially on $\mathfrak u(H_i)$.
Combining this with $S_i = S_k(N_i)$ and $\mathfrak u(H_i) = \mathfrak u(N_i)$, we conclude that the reduction of the ambient group yields the same result for both pairs $(G,H_i)$ and $(G,N_i)$.
\end{proof}

As can be seen from the construction, Algorithm~\ref{alg_D} avoids a part of degenerations performed in Algorithm~\ref{alg_C} and thus indeed enables one to compute the pair $(G_i,K_i)$ in a shorter way.
The following proposition suggests that in fact this provides a considerable optimization of Algorithm~\ref{alg_C}.

\begin{proposition} \label{prop_SSSS_3_degen}
Suppose that $H$ is strongly solvable.
Then Algorithm~\textup{\ref{alg_D}} always performs no more than three degenerations.
\end{proposition}

\begin{proof}
As $H$ is strongly solvable, one has $\Phi = \Delta$ and each component of the SM-decomposition of~$\Psi$ is a singleton.
Fix $i \in \lbrace 1,\ldots, p \rbrace$ and let $\alpha$ be the unique element of~$\Psi_i$.
Then structure results (see~\cite[\S\,3]{Avd_solv} or~\cite[\S\S\,5.1--5.2]{Avd_solv_inv}) imply that there are a simple root $\alpha_i \in \Supp \alpha$ and active roots $\beta_1,\ldots, \beta_s \in \Upsilon_i$ with the following properties:
\begin{itemize}
\item
$\alpha = \alpha_i + k_1\beta_1 + \ldots + k_s \beta_s$ for some $k_1,\ldots, k_s >0$;

\item
$\alpha_i \notin \Supp \beta_j$ for all $j = 1,\ldots,s$;

\item
$\Supp \beta_j$ is orthogonal to $\Supp \beta_m$ for all $j \ne m$;

\item
$\alpha - k_j \beta_j \in \Delta^+$ for all $j = 1,\ldots, s$;

\item
for every $\beta \in \Upsilon_i$ there exists $j \in \lbrace 1,\ldots,s \rbrace$ such that $\Supp \beta \subset \Supp \beta_j$;

\item
the upper elements of $\Upsilon_i$ are $\beta_1,\ldots, \beta_s$.
\end{itemize}
Clearly, $s$ is the number of edges incident to~$\alpha_i$ in the Dynkin diagram of~$\Supp \alpha$, hence $s \le 3$.
One iteration of Algorithm~\ref{alg_D} chooses $j \in \lbrace 1,\ldots, s \rbrace$ and replaces $H$ with a new subgroup whose set of active roots is $\lbrace \alpha - k_j\beta_j \rbrace \cup \lbrace \beta \in \Upsilon_i \mid \Supp \beta \not \subset \Supp \beta_j \rbrace$.
Consequently, the whole Algorithm~\ref{alg_D} performs exactly $s$ iterations and returns a subgroup whose set of active roots is $\lbrace \alpha_i \rbrace$.
\end{proof}

\begin{remark}
Algorithm~\ref{alg_D} and Theorem~\ref{thm_reduction_ref} reduce computation of the spherical roots for spherical subgroups under consideration to the same problem for several pairs $(G,H)$ satisfying the following conditions:
\begin{enumerate}[label=\textup{(S\arabic*)},ref=\textup{S\arabic*}]
\item \label{C1}
$\Supp \Psi = \Pi$;

\item \label{C2}
the SM-decomposition of $\Psi$ has only one component.
\end{enumerate}
As was already mentioned in~\S\,\ref{subsec_main_idea}, in this case the classification of spherical modules yields $|\Psi| \le 2$.
If $|\Psi|=1$ then we get one of the primitive cases, which are classified in~\S\,\ref{subsec_primitive_cases}.
It is an interesting and feasible problem to classify all pairs $(G,H)$ satisfying~(\ref{C1}), (\ref{C2}), and $|\Psi| = 2$ and compute the sets of spherical roots for them (for example, using our methods).
\end{remark}

\section{Examples}
\label{sect_examples}

In this section, we present several examples of computing the set of spherical roots for spherical subgroups.
In each case, $G = \SL_n$ for some~$n$ and
we choose $B,B^-,T$ to be the subgroup of all upper triangular, lower triangular, diagonal matrices, respectively, contained in~$G$.
Then $\Pi = \lbrace \alpha_1, \alpha_2, \ldots, \alpha_{n-1} \rbrace$ where $\alpha_i(t)=t_it_{i+1}^{-1}$ for all $t = \diag(t_1,\ldots,t_n) \in T$.
In all examples, there is a unique choice of a parabolic subgroup $P \supset B^-$ such that $H$ is regularly embedded in~$P$ and $L' \subset K \subset L$.
The fact that $H$ is spherical in~$G$ is always checked via Proposition~\ref{prop_S=K} and the classification of spherical modules.
The set $F_K(\mathfrak p_u / \mathfrak h_u)$ is always computed via Proposition~\ref{prop_free_generators}(\ref{prop_free_generators_a}) by using the data in~\cite[\S\,5]{Kn98}.

\renewcommand{\arraystretch}{1}%

\begin{example}
$G = \SL_4$, $H$ is the connected subgroup of~$G$ whose Lie algebra consists of all matrices of the form
\[
\begin{pmatrix}
a & 0 & 0 & 0\\
b & x & y & 0\\
c & z & -x & 0\\
d & -c & b & -a
\end{pmatrix}.
\]
For this $H$, the groups $P$, $L$, $C$, $K$ consist of all matrices in $G$ having the form
\[
\begin{pmatrix}
* & 0 & 0 & 0 \\
* & * & * & 0 \\
* & * & * & 0 \\
* & * & * & *
\end{pmatrix},
\begin{pmatrix}
* & 0 & 0 & 0 \\
0 & * & * & 0 \\
0 & * & * & 0 \\
0 & 0 & 0 & *
\end{pmatrix},
\begin{pmatrix}
t_1 & 0 & 0 & 0 \\
0 & t_2 & 0 & 0 \\
0 & 0 & t_2 & 0 \\
0 & 0 & 0 & t_1^{-1}t_2^{-2}
\end{pmatrix},
\begin{pmatrix}
t & 0 & 0 & 0 \\
0 & * & * & 0 \\
0 & * & * & 0 \\
0 & 0 & 0 & t^{-1}
\end{pmatrix},
\]
respectively.
Then $\Psi = \lbrace \lambda_1, \lambda_2 \rbrace$ with $\widehat \lambda_1 = \alpha_1 + \alpha_2$ and $\widehat \lambda_2 = \alpha_2 + \alpha_3$, the set $\widetilde \Psi$ consists of one class $\lbrace \lambda_1, \lambda_2 \rbrace$, and
$\Psi_0 = \Psi_0^{\max} = \Psi$.

Using Lemma~\ref{lemma_XiAA0}(\ref{lemma_XiAA0_b}) or Proposition~\ref{prop_normalizer} we check that $N_G(H)^0 = H$.

Let $\iota \colon \mathfrak X(T) \to \mathfrak X(T \cap K)$ be the character restriction map.
Then
\[
\Lambda_L(L/K) = \Ker \iota = \frac12\ZZ(\alpha_1 - \alpha_3)
\]
and $\mathfrak p_u /\mathfrak h_u$ is a simple $K$-module with lowest weight $-\iota(\alpha_1 + \alpha_2) = -\iota(\alpha_2 + \alpha_3)$, so that $F_K(\mathfrak p_u / \mathfrak h_u) = \lbrace \iota(\alpha_1 + \alpha_2) \rbrace$ and $\Lambda_K(\mathfrak p_u / \mathfrak h_u) = \ZZ \iota(\alpha_1 + \alpha_2)$.
Then Proposition~\ref{prop_S=K} yields $\Lambda_G(G/H) = \ZZ\lbrace\alpha_1 + \alpha_2, \frac12(\alpha_1 - \alpha_3)\rbrace$ and hence $|\Sigma_G(G/H)| = \rk \Lambda_G(G/H) = 2$ by Proposition~\ref{prop_finite_index_in_norm}.

For $i=1,2$ let $\mathfrak h_i$ denote the multiplicative degeneration of~$\mathfrak h$ defined by~$\lambda_i$ and put $N_i = N_G(\mathfrak h_i)^0$.
Then Propositions~\ref{prop_limitI} and~\ref{prop_normalizer} imply that the algebras $\mathfrak n_1$, $\mathfrak n_2$ consist of all matrices in~$\mathfrak g$ of the form
\[
\begin{pmatrix}
* & 0 & 0 & 0 \\
* & * & * & 0 \\
* & * & * & 0 \\
* & 0 & 0 & *
\end{pmatrix},
\begin{pmatrix}
* & 0 & 0 & 0 \\
0 & * & * & 0 \\
0 & * & * & 0 \\
* & * & * & *
\end{pmatrix},
\]
respectively.
As can be seen, both subgroups $N_1,N_2$ are regularly embedded in the same parabolic subgroup~$P$, $\Psi(N_1) = \lbrace \lambda_2 \rbrace$, $\Psi(N_2) = \lbrace \lambda_1 \rbrace$.
Then Proposition~\ref{prop_reduction_AG_sph_roots} and Theorem~\ref{thm_primitive_cases}(\ref{thm_primitive_cases_b}) yield $\Sigma_G(G/N_1) = \lbrace \alpha_2 + \alpha_3 \rbrace$ and $\Sigma_G(G/N_2) = \lbrace \alpha_1 + \alpha_2 \rbrace$.
Since both elements $\alpha_1 + \alpha_2$ and $\alpha_2 + \alpha_3$ are primitive in~$\Lambda_G(G/H)$, we finally obtain $\Sigma_G(G/H) = \lbrace \alpha_1 + \alpha_2, \alpha_2 + \alpha_3 \rbrace$.
\end{example}

\begin{example} \label{ex_2}
$G = \SL_4$, the subgroups $H$, $P$, $L$, $C$ consist of all matrices in $G$ having the form
\[
\begin{pmatrix}
* & * & 0 & 0\\
* & * & 0 & 0\\
0 & 0 & * & 0\\
0 & 0 & * & *
\end{pmatrix},
\begin{pmatrix}
* & * & 0 & 0 \\
* & * & 0 & 0 \\
* & * & * & 0 \\
* & * & * & *
\end{pmatrix},
\begin{pmatrix}
* & * & 0 & 0 \\
* & * & 0 & 0 \\
0 & 0 & * & 0 \\
0 & 0 & 0 & *
\end{pmatrix},
\begin{pmatrix}
t_1 & 0 & 0 & 0 \\
0 & t_1& 0 & 0 \\
0 & 0 & t_2 & 0 \\
0 & 0 & 0 & t_1^{-2}t_2^{-1}
\end{pmatrix},
\]
respectively, and $K = L$.
Then $\Psi = \lbrace \lambda_1, \lambda_2 \rbrace$ with $\widehat \lambda_1 = \alpha_1 + \alpha_2$ and $\widehat \lambda_2 =  \alpha_1 + \alpha_2 + \alpha_3$, the set
$\widetilde \Psi$ consists of two classes $\lbrace \lambda_1 \rbrace$ and $\lbrace \lambda_2 \rbrace$, and $\Psi_0 = \varnothing$.

Using Lemma~\ref{lemma_XiAA0}(\ref{lemma_XiAA0_b}) or Proposition~\ref{prop_normalizer} we check that $N_G(H)^0 = H$.

The $K$-module $\mathfrak p_u / \mathfrak h_u$ is a direct sum of two simple $K$-modules with lowest weights $-(\alpha_1 + \alpha_2)$ and $-(\alpha_1 + \alpha_2 + \alpha_3)$, and so
\[
F_K(\mathfrak p_u / \mathfrak h_u) = \lbrace\alpha_1 + \alpha_2, \alpha_1 + \alpha_2 + \alpha_3, \alpha_1 + 2\alpha_2 + \alpha_3 \rbrace.
\]
Then Proposition~\ref{prop_S=K} yields $\Lambda_G(G/H) = \Lambda_K(\mathfrak p_u / \mathfrak h_u) = \ZZ\lbrace \alpha_1, \alpha_2, \alpha_3 \rbrace$, which implies $|\Sigma_G(G/H)| = \rk \Lambda_G(G/H) = 3$ by Proposition~\ref{prop_finite_index_in_norm}.

For $i=1,2$ let $\mathfrak h_i$ denote the additive degeneration of~$\mathfrak h$ defined by~$\lambda_i$ and put $N_i = N_G(\mathfrak h_i)^0$.
Then Propositions~\ref{prop_h_infty} and~\ref{prop_normalizer} imply that the algebras $\mathfrak n_1, \mathfrak n_2$ consist of all matrices in~$\mathfrak g$ of the form
\[
\begin{pmatrix}
* & 0 & 0 & 0 \\
* & * & 0 & 0 \\
* & * & * & 0 \\
* & 0 & 0 & *
\end{pmatrix},
\begin{pmatrix}
* & 0 & 0 & 0 \\
* & * & 0 & 0 \\
0 & 0 & * & 0 \\
* & * & * & *
\end{pmatrix},
\]
respectively.
We now discuss three different ways on how one can proceed.

Firstly, observe that both subgroups $N_1,N_2$ are regularly embedded in~$B^-$ and hence are strongly solvable.
Since for strongly solvable spherical subgroups there are explicit formulas for all the Luna--Vust invariants  given by \cite[Theorem~5.28]{Avd_solv_inv}, at this point we can apply part~(c) of the above-cited theorem and get
\begin{equation} \label{eqn_example}
\Sigma_G(G/N_1) = \lbrace \alpha_2, \alpha_3 \rbrace \ \text{and} \ \Sigma_G(G/N_2) = \lbrace \alpha_1, \alpha_2 \rbrace.
\end{equation}

Secondly, we can repeat the procedure for each of the subgroups $N_1,N_2$.
We have $\Psi(N_1) = \lbrace \beta_1, \beta_2 \rbrace$ with $\beta_1 = \alpha_2 + \alpha_3$, $\beta_2 = \alpha_3$ and $\Psi(N_2) = \lbrace \gamma_1, \gamma_2 \rbrace$ with $\gamma_1 = \alpha_1 +\nobreak \alpha_2$, $\gamma_2 =\nobreak \alpha_2$.
For $i = 1,2$ let $\mathfrak h_{1i}$ (resp.~$\mathfrak h_{2i}$) denote the additive degeneration of~$\mathfrak n_1$ (resp.~$\mathfrak n_2$) defined by~$\beta_i$ (resp.~$\gamma_i$) and put $N_{1i} = N_G(\mathfrak h_{1i})^0$ (resp.~$N_{2i} = N_G(\mathfrak h_{2i})^0$).
Then Propositions~\ref{prop_h_infty} and~\ref{prop_normalizer} imply that the algebras $\mathfrak n_{11}, \mathfrak n_{12}, \mathfrak n_{22}$ consist of all matrices in~$\mathfrak g$ of the form
\[
\begin{pmatrix}
* & 0 & 0 & 0 \\
* & * & 0 & 0 \\
* & * & * & 0 \\
* & * & 0 & *
\end{pmatrix},
\begin{pmatrix}
* & 0 & 0 & 0 \\
* & * & 0 & 0 \\
* & 0 & * & 0 \\
* & * & * & *
\end{pmatrix},
\begin{pmatrix}
* & 0 & 0 & 0 \\
0 & * & 0 & 0 \\
* & * & * & 0 \\
* & * & * & *
\end{pmatrix},
\]
respectively, and $\mathfrak n_{21} = \mathfrak n_{12}$.
Clearly, $|\Psi(N_{ij})|=1$ for all $i,j = 1,2$, therefore Proposition~\ref{prop_reduction_AG_sph_roots} and Theorem~\ref{thm_primitive_cases} yield $\Sigma_G(G/N_{11}) = \lbrace \alpha_3 \rbrace$, $\Sigma_G(G/N_{12}) = \Sigma_G(G/N_{21}) = \lbrace \alpha_2 \rbrace$, and $\Sigma_G(G/N_{22}) = \lbrace \alpha_1 \rbrace$.

Thirdly, one can apply our optimization of the base algorithm for $N_1$ and~$N_2$.
Clearly, all components of the SM-decompositions of $\Psi(N_1),\Psi(N_2)$ are singletons.
Below we list the results of applying Algorithm~\ref{alg_D} followed by reduction of the ambient group in the various cases;
the output is always $(\SL_2, \text{maximal torus})$ (which is a general feature of strongly solvable spherical subgroups).

$(N_1, \lbrace \beta_1 \rbrace)$: the simple root of $\SL_2$ is identified with~$\alpha_2$; one degeneration performed;

$(N_1, \lbrace \beta_2 \rbrace)$: the simple root of $\SL_2$ is identified with~$\alpha_3$; no degenerations performed;

$(N_2, \lbrace \gamma_1 \rbrace)$: the simple root of $\SL_2$ is identified with~$\alpha_1$; one degeneration performed;

$(N_2, \lbrace \gamma_2 \rbrace)$: the simple root of $\SL_2$ is identified with~$\alpha_2$; no degenerations performed.

By Proposition~\ref{prop_reduction_AG_sph_roots} and Theorem~\ref{thm_primitive_cases}(\ref{thm_primitive_cases_b}) we get~(\ref{eqn_example}).

Since all the three elements $\alpha_1, \alpha_2, \alpha_3$ are primitive in~$\Lambda_G(G/H)$, we finally obtain $\Sigma_G(G/H) = \lbrace \alpha_1, \alpha_2, \alpha_3 \rbrace$.

Note that for computing $\Sigma_G(G/N_1)$ and $\Sigma_G(G/N_2)$ via Algorithm~\ref{alg_D} only one degeneration is required in each case whereas the base algorithm and Algorithm~\ref{alg_C} require two degenerations.
\end{example}

\begin{example}
$G = \SL_7$, the subgroups $H$, $P$, $L$ consist of all matrices in $G$ having the form
\[
\begin{pmatrix}
* & * & 0 & 0 & 0 & 0 & 0\\
* & * & 0 & 0 & 0 & 0 & 0\\
0 & 0 & * & 0 & 0 & 0 & 0\\
* & * & * & * & * & 0 & 0\\
* & * & * & * & * & 0 & 0\\
* & * & * & * & * & * & 0\\
0 & 0 & * & 0 & 0 & 0 & *\\
\end{pmatrix},
\begin{pmatrix}
* & * & 0 & 0 & 0 & 0 & 0\\
* & * & 0 & 0 & 0 & 0 & 0\\
* & * & * & 0 & 0 & 0 & 0\\
* & * & * & * & * & 0 & 0\\
* & * & * & * & * & 0 & 0\\
* & * & * & * & * & * & 0\\
* & * & * & * & * & * & *\\
\end{pmatrix},
\begin{pmatrix}
* & * & 0 & 0 & 0 & 0 & 0\\
* & * & 0 & 0 & 0 & 0 & 0\\
0 & 0 & * & 0 & 0 & 0 & 0\\
0 & 0 & 0 & * & * & 0 & 0\\
0 & 0 & 0 & * & * & 0 & 0\\
0 & 0 & 0 & 0 & 0 & * & 0\\
0 & 0 & 0 & 0 & 0 & 0 & *\\
\end{pmatrix},
\]
respectively, $C$ consists of all matrices in $G$ of the form $\diag (t_1,t_1,t_2,t_3,t_3,t_4,t_5)$, and $K=L$.
Then $\Psi = \lbrace \lambda_1, \lambda_2, \lambda_3, \lambda_4 \rbrace$ with $\widehat \lambda_1 = \alpha_1 + \alpha_2$, $\widehat \lambda_2 =  \alpha_1 + \ldots + \alpha_6$, $\widehat \lambda_3 = \alpha_4+\alpha_5+\alpha_6$, and $\widehat \lambda_4 = \alpha_6$; $\Psi_0 = \varnothing$.
The SM-decomposition of $\Psi$ is $\Psi = \lbrace \lambda_1,\lambda_2 \rbrace \cup \lbrace \lambda_3 \rbrace \cup \lbrace \lambda_4 \rbrace$.

Using Lemma~\ref{lemma_XiAA0}(\ref{lemma_XiAA0_b}) or Proposition~\ref{prop_normalizer} we check that $N_G(H)^0 = H$.

Here are results of applying Algorithm~\ref{alg_D} followed by reduction of the ambient group in the various cases:

$(H, \lbrace \lambda_1,\lambda_2 \rbrace)$: the output is $(G_1, K_1)$ with $G_1 = \SL_4$ whose $i$th simple root is identified with $\alpha_i$ ($i = 1,2,3$) and $K_1$ described below; one degeneration performed;

$(H, \lbrace \lambda_3 \rbrace)$: the output is $(G_2,K_2)$ with $G_2 = \SL_3$ whose $i$th simple root is identified with $\alpha_{i+3}$ ($i = 1,2$) and $K_2$ described below; one degeneration performed;

$(H, \lbrace \lambda_4 \rbrace)$: the output is $(G_3, K_3)$ with $G_3 = \SL_2$ whose simple root is identified with~$\alpha_6$ and $K_3$ being a maximal torus; no degenerations performed.

The subgroups $K_1 \subset G_1$, $K_2 \subset G_2$ consist of all matrices having the form
\[
\begin{pmatrix}
* & * & 0 & 0\\
* & * & 0 & 0\\
0 & 0 & * & 0\\
0 & 0 & * & *\\
\end{pmatrix},
\begin{pmatrix}
* & * & 0\\
* & * & 0\\
0 & 0 & *\\
\end{pmatrix},
\]
respectively.
From Example~\ref{ex_2} we know that $\Sigma_{G_1}(G_1/K_1) = \lbrace \alpha_1, \alpha_2, \alpha_3 \rbrace$.
The cases of $(G_2, K_2)$ and $(G_3, K_3)$ are primitive, and by Theorem~\ref{thm_primitive_cases}(\ref{thm_primitive_cases_b}) the corresponding sets of spherical roots are $\lbrace \alpha_4 + \alpha_5 \rbrace$ and $\lbrace \alpha_6 \rbrace$, respectively.
Applying Theorem~\ref{thm_reduction_ref}, we finally obtain $\Sigma_G(G/H) = \lbrace \alpha_1, \alpha_2, \alpha_3, \alpha_4 + \alpha_5, \alpha_6 \rbrace$.

Note that our computations of $(G_1,K_1)$, $(G_2,K_2)$, $(G_3,K_3)$ via Algorithm~\ref{alg_D} required $1+1+0 = 2$ degenerations.
Computing the same pairs via Algorithm~\ref{alg_C} would require $2 + 4 + 4 = 10$ degenerations (see Remark~\ref{rem_estimates1}).
\end{example}


\end{document}